\documentclass{amsart}
\usepackage{amssymb}
\usepackage{braket}
\usepackage{mathrsfs}
\usepackage[all]{xy}
\usepackage{graphicx}
\usepackage{color}

\usepackage{enumerate, pifont, geometry}
\usepackage{stmaryrd}

\newtheorem{thm}{Theorem}[section]
\newtheorem{cor}[thm]{Corollary}
\newtheorem{prop}[thm]{Proposition}
\theoremstyle{definition}
\newtheorem{dfn}[thm]{Definition}
\newtheorem{ex}[thm]{Example}
\newtheorem{claim}[thm]{Claim}
\newtheorem{lem}[thm]{Lemma}
\newtheorem{fact}[thm]{Fact}

\newtheorem{cond}[thm]{Condition}

\theoremstyle{remark}
\newtheorem{rem}[thm]{Remark}

\newcommand{\Zbb}{\mathbb{Z}}        
\newcommand{\Hom}{\mathrm{Hom}}      
\newcommand{\Ob}{\mathrm{Ob}}        
\newcommand{\Cok}{\mathrm{Cok}}      

\newcommand{\id}{\mathrm{id}}         
\newcommand{\se}{\subseteq}           
\newcommand{\ppr}{^{\prime}}          
\newcommand{\pprr}{^{\prime\prime}}   
\newcommand{\Ker}{\mathrm{Ker}}       
\newcommand{\fa}{\forall}             
\newcommand{\co}{\colon}              
\newcommand{\ci}{\circ}               
\newcommand{\iv}{^{-1}}               

\newcommand{\lra}{\longrightarrow}    
\newcommand{\EQ}{\Leftrightarrow}     
\newcommand{\wt}{\widetilde}          

\newcommand{\afr}{\mathfrak{a}}  
\newcommand{\bfr}{\mathfrak{b}}  
\newcommand{\cfr}{\mathfrak{c}}  
\newcommand{\dfr}{\mathfrak{d}}  
\newcommand{\efr}{\mathfrak{e}}  
\newcommand{\ifr}{\mathfrak{i}}  
\newcommand{\nfr}{\mathfrak{n}}  
\newcommand{\sfr}{\mathfrak{s}}  

\newcommand{\Afr}{\mathfrak{A}}  
\newcommand{\Bfr}{\mathfrak{B}}  
\newcommand{\Cfr}{\mathfrak{C}}  
\newcommand{\Mfr}{\mathfrak{M}}  
\newcommand{\Pfr}{\mathfrak{P}}  
\newcommand{\Sfr}{\mathfrak{S}}  
\newcommand{\Tfr}{\mathfrak{T}}  

\newcommand{\Bp}{\mathfrak{B}_+}  

\newcommand{\Acal}{\mathcal{A}} 
\newcommand{\Ccal}{\mathcal{C}} 
\newcommand{\Kcal}{\mathcal{K}} 
\newcommand{\Mcal}{\mathcal{M}} 
\newcommand{\Ocal}{\mathcal{O}} 
\newcommand{\Pcal}{\mathcal{P}} 
\newcommand{\Scal}{\mathcal{S}} 
\newcommand{\und}{\underline}   
\newcommand{\ovl}{\overline}    
\newcommand{\ov}{\overset}      
\newcommand{\un}{\underset}     
\newcommand{\ti}{\times}        

\newcommand{\bsm}{\begin{smallmatrix}}
\newcommand{\esm}{\end{smallmatrix}}

\newcommand{\al}{\alpha}         
\newcommand{\be}{\beta}          
\newcommand{\gam}{\gamma}        
\newcommand{\vp}{\varphi}        
\newcommand{\kap}{\kappa}        
\newcommand{\lam}{\lambda}       
\newcommand{\om}{\omega}         
\newcommand{\thh}{\theta}        
\newcommand{\del}{\delta}        
\newcommand{\ep}{\varepsilon}    
\newcommand{\sig}{\sigma}        
\newcommand{\ze}{\zeta}          

\newcommand{\Del}{\Delta}        

\numberwithin{equation}{section}

\newcommand{\Nn}{\mathbb{N}}   
\newcommand{\Np}{\mathbb{N}_{>0}}   
\newcommand{\Char}{\mathrm{char}}   
\newcommand{\lcm}{\mathrm{lcm}}   

\newcommand{\type}{\mathrm{type}}     
\newcommand{\HH}{\mathrm{HH}}  
\newcommand{\ufr}{\mathfrak{u}}  
\newcommand{\vfr}{\mathfrak{v}}  
\newcommand{\dia}{^{\diamond}}  
\newcommand{\length}{\ell}  
\newcommand{\mn}{\langle\mu,\nu\rangle}  
\newcommand{\mnp}{\langle\mu\ppr,\nu\ppr\rangle}  
\newcommand{\emn}{\ep_{\mn}}  
\newcommand{\emnp}{\ep_{\mnp}}  
\newcommand{\proj}{\mathrm{proj}} 
\newcommand{\modd}{\mathrm{mod}}  
\newcommand{\End}{\mathrm{End}}   
\newcommand{\thick}{\mathrm{thick}} 
\newcommand{\HHom}{\mathbb{H}\mathrm{om}} 
\newcommand{\ust}{^{\star}}       
\newcommand{\din}{d_{\mathrm{in}}}   
\newcommand{\dout}{d_{\mathrm{out}}} 

\newcommand{\blossom}{\text{\ding{96}}} 
\newcommand{\bls}{^\blossom} 

\newcommand{\Ak}{A^{(k)}}   
\newcommand{\Akp}{A^{(k)\prime}}   

\begin{document}

\title[Finite gentle repetitions and AG-invariants]{Finite gentle repetitions of gentle algebras and their Avella-Alaminos--Geiss invariants}
%
%

\author{Hiroyuki Nakaoka}
\address{Research and Education Assembly, Science and Engineering Area, Research Field in Science, Kagoshima University, 1-21-35 Korimoto, Kagoshima, 890-0065 Japan}
\email{nakaoka@sci.kagoshima-u.ac.jp}

\thanks{The author wishes to thank Prof.~Yann Palu for introducing him this subject, Dr.~Yuya Mizuno for introducing him the notions related to tilting theory, Prof.~Takefumi Kondo for helpful discussions, Prof.~Hideto Asashiba and Prof.~Steffen Koenig for their encouragement and advices, Prof.~Osamu Iyama for questions and comments.}
\thanks{Part of this article was written during the author's stay at the University of St\"{u}ttgart. The author is grateful to all researchers he met there, for their hospitality and for offering him a very good research environment.}
\thanks{The author is supported by JSPS KAKENHI Grant Number 17K18727.}

\begin{abstract}
Among finite dimensional algebras over a field $K$, the class of gentle algebras is known to be closed by derived equivalences. Although a classification up to derived equivalences is usually a difficult problem, 
Avella-Alaminos and Geiss have introduced derived invariants for gentle algebras $A$, which can be calculated combinatorially from their bound quivers, applicable to such classification.

Ladkani has given a formula to describe the dimensions of the Hochschild cohomologies of $A$ in terms of some values of its Avella-Alaminos--Geiss invariants. This in turn implies a cohomological meaning of these values. Since most of the other values do not appear in this formula, it will be a natural question to ask if there is a similar cohomological meaning for such values. In this article, we construct a sequence of gentle algebras $A^{(k)}$ indexed by positive integers by a procedure which we call finite gentle repetitions, in order to relate these values of Avella-Alaminos--Geiss invariants of $A$ 
to the dimensions of Hochschild cohomologies of $A^{(k)}$.

On the way we will see that the Avella-Alaminos--Geiss invariants of $A^{(k)}$ are completely determined by those of $A$. Therefore one may expect that the finite gentle repetitions would preserve derived equivalences in a nice situation. In the latter part of this article, we deal with this problem under some assumptions.
\end{abstract}

\maketitle

\tableofcontents


\section{Introduction}
Among finite dimensional algebras over a field $K$, the class of gentle algebras is known to be closed by derived equivalences, as shown by Schr\"{o}er and Zimmermann in \cite{SZ}. Although a classification up to derived equivalences is usually a difficult problem, 
Avella-Alaminos and Geiss have introduced derived invariants $\phi_A\co\Nn\ti\Nn\to \Nn$ for gentle algebras $A=KQ/I$ in \cite{AG}, which can be calculated combinatorially from its bound quiver $(Q,I)$. These invariants are indeed used in classifications up to derived equivalences of gentle algebras satisfying some conditions, for example in \cite{Av},\cite{Bo},\cite{BM}.

In \cite{La}, Ladkani has given a formula to describe the dimensions of the Hochschild cohomologies $\HH^n(A)$ as the following $\Zbb$-linear sum of Avella-Alaminos--Geiss invariants. 
\begin{fact}\label{FactLadkani}(\cite[Corollary 1]{La})
\begin{itemize}
\item $\dim_K \HH^0(A)=1+\phi_A(1,0)$.
\item $\dim_K \HH^1(A)=1-\chi(Q)+\phi_A(1,1)%
+\begin{cases}
\phi_A(0,1)&\text{if}\ \Char K=2,\\
0&\text{otherwise}.\\
\end{cases}$\\
Here $\chi (Q)=\sharp\{\text{vertices in}\ Q\}-\sharp\{\text{arrows in}\ Q\}$ is the Euler characteristic of the underlying undirected graph of $Q$, which also satisfies
\[ 2\chi(Q)=\displaystyle\sum_{(q,l)\in\Nn\ti\Nn}\phi_A(q,l)(q-l). \]
\item For $n\ge2$,
\[ \dim_K \HH^n(A)=\phi_A(1,n)+a_n\displaystyle\sum_{d|n}\phi_A(0,d)+b_n\!\!\!\!\displaystyle\sum_{d|(n-1)}\phi_A(0,d), \]
where
\[ (a_n,b_n)=
\begin{cases}
(1,0)&\text{if}\ \Char K\ne 2\ \text{and}\ n\ \text{is even},\\
(0,1)&\text{if}\ \Char K\ne 2\ \text{and}\ n\ \text{is odd},\\
(1,1)&\text{if}\ \Char K=2.
\end{cases} \]
\end{itemize}
\end{fact}
A detailed proof for the above result can be found in \cite{RR}, in which Redondo and Rom\'{a}n have formulated for a more general class of string algebras (\cite[Theorem 2]{RR}). In the above formula, we observe that only the values $\phi_A(0,l)$ and $\phi_A(1,l)$ $(l\in\Nn)$ are involved while no $\phi_A(q,l)$ for $q\ge 2$ appears except for the description of $\chi(Q)$. In particular we see that if $(Q,I)$ has no anticycle, it means $\phi_A(0,l)=0$ for any $l\in\Nn$, and thus $\dim_K\HH^n(A)=\phi_A(1,n)$ holds for $n>1$.

It will be a natural question to ask, whether the values $\phi_A(q,l)$ for $q\ge 2$ can be also related to any invariant of cohomological nature. In this article, we construct a sequence of gentle algebras $\Ak$ indexed by $k\in\Np$, in order to relate these $\phi_A(q,l)$ to the dimensions of Hochschild cohomologies $\HH^{n}(\Ak)$.
Definition of AG-invariants will be reviewed in Section \ref{section_Review}. We introduce the construction of $\Ak$ in Section \ref{section_Repetition}, which we call {\it finite gentle repetitions} of $A$. In Section \ref{section_AGRepetition} we describe $\phi_{A^{(k)}}$ in terms of $\phi_A$, and use it to relate $\phi_A(q,l)$ with the dimensions of Hochschild cohomologies when $q\ge 2$. For example, the following is obtained as a corollary of our result.
\begin{cor}(=Corollary \ref{CorCoprime})
Assume that $(Q,I)$ has no anticycle.
Let $(q,l)\in(\Np\ti\Np)\setminus\{(1,1)\}$ be any element. If $q$ and $l$ are coprime, then we have
\[ \phi_A(q,l)=\frac{1}{q}\dim\HH^{q+l-1}(A^{(q)}). \]
\end{cor}
In Section \ref{section_Graded}, for $A^{(k)}$, we also calculate the graded version of AG-invariants introduced by Lekili and Polishchuk in \cite{LP}.

Results in Section \ref{section_AGRepetition} in particular show that the AG-invariants of $A^{(k)}$ are completely determined by those of $A$. Thus one may expect that the procedure of finite gentle repetition preserves derived equivalences in some nice situations. The latter part of this article is dealing with this problem.

In Section \ref{section_APR} we show how generalized Auslander-Platzeck-Reiten reflections in \cite{BM} can be related with the gentle repetitions. In Section \ref{section_Associated}, using some semisimple algebra $V(A)$ associated to $A$, we show how $A^{(k)}$ can be realized as an upper triangular matrix algebra, similarly to the usual repetitive algebra. In Section \ref{section_Characterize} we characterize this $V(A)$ by some homological properties. In Section \ref{section_Preservation}, we will show the preservation of derived equivalences under an assumption of some conditions.

\section{Review on Avella-Alaminos--Geiss invariants}\label{section_Review}


Throughout this article, let $K=\ovl{K}$ be an algebraically closed field, and let $(Q,I)$ be a finite connected bound quiver, and let $A=KQ/I$ be the associated bound quiver algebra. For the basic notions around representations of quivers, we refer to \cite{ASS}. We denote by $e_a\in A$ the idempotent corresponding to a vertex $a\in Q_0$. To exclude the trivial case of $A=K$, we always assume $|Q_1|>0$ in this article. Let $s,t\co Q_1\to Q_0$ denote the maps taking the sources and targets of arrows, respectively.

We use the following notations and terminology.
\begin{dfn}
Let $\rho=\al_1\cdots\al_{l}$ be any path in $Q$, where  $l>0$ and $\al_1,\ldots,\al_{l}\in Q_1$. By definition, it satisfies $t(\al_i)=s(\al_{i+1})$ for any $1\le i<l$.
\begin{itemize}
\item We say $\rho$ {\it starts with} $\al_1$ and {\it ends with} $\al_{l}$.
\item Integer $l$ is called the {\it length} of $\rho$, and denoted by $\length(\rho)$.
\item Source and target vertices of $\rho$ are denoted by $s(\rho)=s(\al_1)$ and $t(\rho)=t(\al_{l})$, respectively.
\item $\rho$ is called a {\it path in} $A$ if $\rho\notin I$. 
\item $\rho$ is called an {\it oriented cycle} in $A$ if it is a path in $A$, 
which moreover satisfies $t(\al_l)=s(\al_1)$ and $\al_l\al_1\notin I$.
We also always assume it is primitive, in the sense that $t(\al_i)\ne a_1$ holds for any $1\le i<l$.
\item $\rho$ is called an {\it antipath in} $A$ if it satisfies $\al_i\al_{i+1}\in I$ for any $1\le i<l$. 
\item $\rho$ is called an {\it anticycle in} $A$ if it is an antipath in $A$, 
which moreover satisfies $t(\al_l)=s(\al_1)$ and $\al_{l}\al_1\in I$.
We also assume it is primitive, in the sense that $t(\al_i)\ne a_1$ holds for any $1\le i<l$.
\end{itemize}
\end{dfn}

Throughout this article, when we say $\al_1\cdots\al_l$ is an (anti)path or an (anti)cycle, it implicitly means that $\al_1,\ldots,\al_l$ are arrows, unless otherwise specified.

\begin{dfn}
Let $a\in Q_0$ be any vertex.
\begin{itemize}
\item $\din(a)=|\{\al\in Q_1\mid t(\al)=a\}|$ denotes the {\it in-degree} of $a$.
\item $\dout(a)=|\{\al\in Q_1\mid s(\al)=a\}|$ denotes the {\it out-degree} of $a$.
\item Idempotent element $e_a\in A$ is regarded as a path of length $0$.
\end{itemize}
\end{dfn}

\begin{dfn}\label{DefGentle}
Let $(Q,I)$ be a connected bound quiver as before.
\begin{enumerate}
\item $(Q,I)$ is called {\it locally gentle} if it satisfies the following conditions.
\begin{itemize}
\item[{\rm (i)}] The admissible ideal $I$ is quadratic monomial. Namely, it is generated by some set $R$ of paths of length $2$. 
\item[{\rm (ii)}] Any vertex $a\in Q_0$ satisfies $\din(a)\le2$ and $\dout(a)\le2$.
\item[{\rm (iii)}] For any arrow $\al\in Q_1$, it satisfies
\begin{eqnarray*}
&|\{\be\in Q_1\mid t(\be)=s(\al)\ \text{and}\ \be\al\in I \}|\le 1,&\\
&|\{\be\in Q_1\mid t(\be)=s(\al)\ \text{and}\ \be\al\notin I \}|\le 1,&\\
&|\{\be\in Q_1\mid s(\be)=t(\al)\ \text{and}\ \al\be\in I \}|\le 1,&\\
&|\{\be\in Q_1\mid s(\be)=t(\al)\ \text{and}\ \al\be\notin I \}|\le 1.&
\end{eqnarray*}
\end{itemize}
In this case, we call $A=KQ/I$ a {\it locally gentle algebra}.
\item $(Q,I)$ is called {\it gentle} if it is locally gentle and has no oriented cycle in $A$. In this case, we call $A=KQ/I$ a {\it gentle algebra}.
\end{enumerate}
\end{dfn}

\begin{rem}
By definition, locally gentle algebra $A$ is gentle if and only if it is finite dimensional over $K$.
\end{rem}

\begin{dfn}\label{DefTA}
Let $A=KQ/I$ be a locally gentle algebra.
\begin{enumerate}
\item A path $\rho$ in $A$ is said to be {\it maximal}, if neither $\al\rho$ nor $\rho\al$ is a path in $A$ for any arrow $\al\in Q_1$. This is also called a {\it non-trivial permitted thread} in \cite{AG}.
Let $\Mfr_A$ denote the set of maximal paths in $A$.
\item Define $\Tfr_A\se Q_0$ by
\[ \Tfr_A=\Set{ a\in Q_0 \ |
\begin{array}{l} \din(a)\le1,\ \dout(a)\le1,\ \text{and}\\
\be\gam\notin I\ \text{holds for any}\ \be,\gam\in Q_1\ \text{with}\ t(\be)=s(\gam)=a\end{array} }. \]
Then $e_a\in A$ for $a\in\Tfr_A$ is called a {\it trivial permitted thread} in \cite{AG}.
\item An antipath $\del$ in $A$ is said to be maximal, if neither $\al\del$ nor $\del\al$ is an antipath in $A$ for any arrow $\al\in Q_1$. This is also called a {\it non-trivial forbidden thread} in \cite{AG}.
Let $\Afr_A$ denote the set of maximal antipaths in $A$. 
Similarly as {\rm (2)}, if a vertex $v\in Q_0$ satisfies
\begin{itemize}
\item $\din(a)\le1$, $\dout(a)\le1$, and
\item $\be\gam\in I$ holds for any $\be,\gam\in Q_1$ with $t(\be)=s(\gam)=a$,
\end{itemize}
then $e_a\in A$ is called a {\it trivial forbidden thread} in \cite{AG}.
\end{enumerate}
\end{dfn}

\begin{rem}\label{RemCountTA}
If $A=KQ/I$ is gentle and if $\rho_1,\ldots,\rho_r\in\Mfr_A$ are the all maximal  paths in $A$, then we have
\[ |\Tfr_A|=2|Q_0|-\sum_{1\le i\le r}(\length(\rho_i)+1). \]
\end{rem}

\bigskip

We recall the definition of the {\it blossoming} quiver from \cite{PPP1},\cite{PPP2} which is also called the {\it fringed} quiver in \cite{BDMTY}. Asashiba had retrofitted the definition of Avella-Alaminos--Geiss invariants with this procedure, without a name in \cite{As2}.
\begin{dfn}\label{DefBlossom}
Let $(Q,I)$ be a locally gentle bound quiver.
\begin{enumerate}
\item Its blossoming quiver $(Q\bls,I\bls)$ is defined as follows. First, for each $v\in Q_0$ we add
\begin{itemize}
\item $2-\din(v)$ incoming arrows to $v$, from new (pairwise distinct) vertices,
\item $2-\din(v)$ incoming arrows to $v$, from new (pairwise distinct) vertices.
\end{itemize}
Then add relations to $I$, so that $(Q\bls,I\bls)$ becomes locally gentle.
Put $A\bls=KQ\bls/I\bls$.
\item A vertex in $Q_0\bls\setminus Q_0$ is called a {\it blossom vertex}.
Any blossom vertex is either a source or a sink in $Q\bls$.
Any maximal path in $(Q\bls,I\bls)$ should start at a source blossom vertex and end at a sink blossom vertex. This gives a bijection between the (necessarily finite) set of source blossom vertices and the set of sink blossom vertices.
\item Label the source blossom vertices as $s_1,\ldots,s_d$, and the sink blossom vertices as $t_1,\ldots,t_d$, so that $s_p$ and $t_p$ are connected by a maximal path, for each $1\le p\le d$. This $d=\dfr_A$ can be calculated as $\dfr_A=2|Q_0|-|Q_1|$.
\item Arrows in $Q_1\bls\setminus Q_1$ are called {\it blossom arrows}.
Let $\sig_p$ (respectively $\tau_p$) denote the unique blossom arrow with $s(\sig_p)=s_p$ (resp. $t(\tau_p)=t_p$).
For each $1\le p\le d$, the maximal path connecting $s_p$ to $t_p$ can be written as $\Pfr_p=\sig_p\wp_p\tau_p$ for some path $\wp_p$ in $A$, which is of one of the following forms.
\begin{itemize}
\item[{\rm (i)}] $\wp_p\in\Mfr_A$. Namely $\wp_p$ is a non-trivial permitted thread.
\item[{\rm (ii)}] $\wp_p=e_a$ for some $a\in\Tfr_A$, namely, it is a trivial permitted thread. Indeed, $\sig_p\tau_p$ is a maximal path in $A\bls$ if and only if $t(\sig_p)=s(\tau_p)\in\Tfr_A$ holds.
\end{itemize}

\item When we emphasize $A$, we denote the above $\Pfr_p$ and $\wp_p$ by $\Pfr_p(A)$ and $\wp_p(A)$ respectively, for $1\le p\le d$. As in {\rm (4)}, we have
\begin{equation}\label{MABSDT}
\Mfr_{A\bls}=\{\Pfr_p(A)=\sig_p\wp_p(A)\tau_p \mid 1\le p\le \dfr_A\}.
\end{equation}
\end{enumerate}
\end{dfn}

\begin{rem}\label{RemGentleUnique}
For any locally gentle bound quiver $(Q,I)$, its blossoming $(Q\bls,I\bls)$ is determined uniquely up to a re-labeling of source and sink vertices, by the following properties.
\begin{itemize}
\item[{\rm (i)}] $(Q\bls,I\bls)$ is a locally gentle bound quiver, which contains $(Q,I)$ as a {\it bound subquiver} in the following sense.
\begin{itemize}
\item $Q_0\se Q\bls_0$ and $Q_1\se Q\bls_1$ hold.
\item For any arrow $\al\in Q_1$, its source $s(\al)$ taken in $Q$ agrees with that in $Q\bls$. Similarly for the targets.
\item For any pair of arrows $\al,\be\in Q_1$, we have $\al\be\in I$ if and only if $\al\be\in I\bls$.
\end{itemize}
\item[{\rm (ii)}] Any vertex $v\in Q_0$ has in-degree $2$ and out-degree $2$ in $Q\bls$.
\item[{\rm (iii)}] Any vertex in $Q\bls_0\setminus Q_0$ is either a source with out-degree $1$ or a sink with in-degree $1$ in $Q\bls$.
\end{itemize}
\end{rem}

\begin{rem}
For any locally gentle algebra $A=KQ/I$, the following are equivalent.
\begin{enumerate}
\item $\Mfr_A\amalg \Tfr_A=\emptyset$.
\item Any vertex $a\in Q_0$ satisfies $\din(a)=\dout(a)=2$, and any arrow $\al\in Q_1$ is a part of an oriented cycle in $A$.
\item $d=0$. Namely, $(Q,I)$ agrees with its blossoming.
\end{enumerate}
In particular, we have $\Mfr_A\amalg \Tfr_A\ne\emptyset$ if $A$ is gentle.
\end{rem}

\begin{dfn}\label{DefAGOrb}
Let $(Q,I)$ be a locally gentle bound quiver, and let $(Q\bls,I\bls)$ be its blossoming, as in Definition \ref{DefBlossom}.
For any $1\le p\le d=\dfr_A$, there is a unique maximal antipath $\Del_p$ in $A\bls$ ending at $t_p$. It should start at some $s_{\Phi(p)}$. This gives a permutation $\Phi=\Phi_A\in\Sfr_d$, where $\Sfr_d$ is the symmetric group of degree $d$. Removing blossom arrows $\tau_p$ and $\sig_{\Phi(p)}$ from $\Del_p$, we obtain a maximal antipath $\del_p$ in $A$, or an idempotent element $\del_p=e_a$ if $s(\tau_p)=t(\sig_{\Phi(p)})=a\in Q_0$.
When we emphasize $A$, denote $\Del_p$ and $\del_p$ by $\Del_p(A)$ and $\del_p(A)$, respectively. By definition we have
\begin{equation}\label{Description_Delta}
\Del_p=\sig_{\Phi(p)}\del_p\tau_p.
\end{equation}
We may regard that $\Phi$ acts on the set $\{\del_p\mid 1\le p\le d\}$ by $\Phi(\del_p)=\del_{\Phi(p)}$. It gives an orbit decomposition
\begin{equation}\label{DecompOrb}
\{\del_p\mid1\le p\le d\}=\Ocal_1\amalg\cdots\amalg\Ocal_c.
\end{equation}
In this article, let us call such orbit an {\it AG-orbit} in $A$. We remark that $(\ref{DecompOrb})$ is an empty set if $d=0$.
\end{dfn}

\begin{dfn}\label{DefAGInv}
Assume that $A=KQ/I$ is a gentle algebra.
\begin{enumerate}
\item For each AG-orbit $\Ocal$, define its {\it type} $\type(\Ocal)=(q,l)$ by
\[ q=|\Ocal|\ \ \text{and}\ \ l=\sum_{\del_p\in\Ocal} \length(\del_p). \]
\item For $(q,l)\in\Nn\ti\Nn$, define $\phi_A(q,l)\in\Nn$ by the following.
\[ \phi_A(q,l)=
\begin{cases}
|\{ \text{anticycles in}\ A\ \text{of length}\ l \}|& \text{for}\ q=0,\\
|\{ \text{AG-orbits in}\ A\ \text{of type}\ (q,l) \}|&\text{for}\ q\ge 1.
\end{cases}\]
\end{enumerate}
\end{dfn}

\begin{rem}
The above definition is a slight paraphrase of the original definition in \cite{AG}, with the idea of Asashiba \cite{As2} to use blossoming and the observation by Bobi\'{n}ski \cite{Bo} involving bijections on the set of forbidden threads.
\end{rem}

\begin{rem}
Graded version of Definition \ref{DefAGInv} has been introduced by Lekili and Polishchuk in \cite{LP}. We will deal with this in Section \ref{section_Graded}.
\end{rem}

By Avella-Alaminos and Geiss, they are shown to be derived invariants, as follows.
\begin{fact}(\cite[Theorem A]{AG})
If $A$ and $B$ are gentle algebras which are derived equivalent, then $\phi_A=\phi_B$.
\end{fact}

\begin{ex}\label{ExGentleBlossom}
Following is an example of a gentle bound quiver $(Q,I)$.
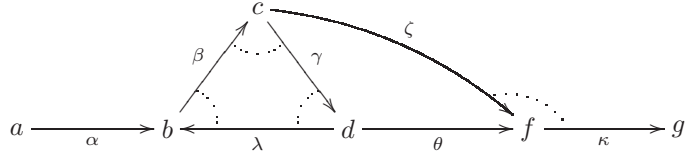
\[
\xy
(-44,0)*+{a}="0";
(-24,0)*+{b}="1";
(-12,16)*+{c}="2";
(0,0)*+{d}="3";
(24,0)*+{f}="4";
(44,0)*+{g}="5";
{\ar_{\al} "0";"1"};
{\ar^{\be} "1";"2"};
{\ar^{\gam} "2";"3"};
{\ar@/^0.8pc/^{\ze} "2";"4"};
{\ar^{\lam} "3";"1"};
{\ar_{\thh} "3";"4"};
{\ar_{\kap} "4";"5"};
%
{\ar@/_0.25pc/@{.} (-17.5,0);(-20,5)}; 
{\ar@/^0.25pc/@{.} (-6.5,0);(-4,5)}; 
{\ar@/_0.25pc/@{.} (-15,11);(-9,11)}; 
{\ar@/^0.4pc/@{.} (19.5,4.5);(29.5,0)}; 
\endxy
\]
The dots indicate which pairs of paths are in the relation. We have $I=\langle\{\be\gam,\gam\lam,\lam\be,\ze\kap\}\rangle$ in this example. For $A=KQ/I$, we have $\Mfr_A=\{\al\be\ze,\gam\thh\kap,\lam\}$, and $\Tfr_A=\{a,g\}\se Q_0$.

Its blossoming $(Q\bls,I\bls)$ is given as follows.
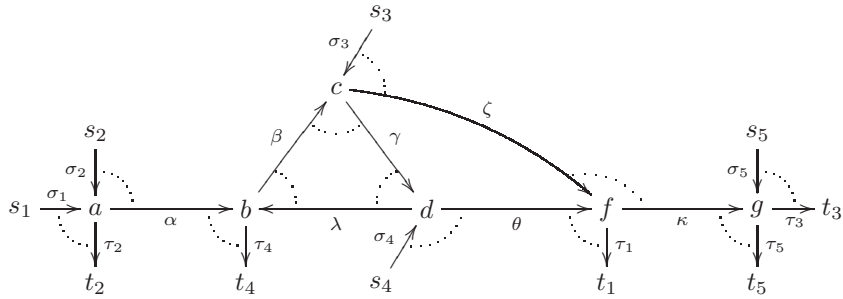
\[
\xy
(-44,0)*+{a}="0";
(-24,0)*+{b}="1";
(-12,16)*+{c}="2";
(0,0)*+{d}="3";
(24,0)*+{f}="4";
(44,0)*+{g}="5";
{\ar_{\al} "0";"1"};
{\ar^{\be} "1";"2"};
{\ar^{\gam} "2";"3"};
{\ar@/^0.8pc/^{\ze} "2";"4"};
{\ar^{\lam} "3";"1"};
{\ar_{\thh} "3";"4"};
{\ar_{\kap} "4";"5"};
%
{\ar@/_0.25pc/@{.} (-17.5,0);(-20,5)}; 
{\ar@/^0.25pc/@{.} (-6.5,0);(-4,5)}; 
{\ar@/_0.25pc/@{.} (-15,11);(-9,11)}; 
{\ar@/^0.4pc/@{.} (19.5,4.5);(29.5,0)}; 
%
%
(-54,0)*+{s_1}="s1";
(-44,10)*+{s_2}="s2";
(-6,26)*+{s_3}="s3";
(-6,-10)*+{s_4}="s4";
(44,10)*+{s_5}="s5";
(24,-10)*+{t_1}="t1";
(-44,-10)*+{t_2}="t2";
(54,0)*+{t_3}="t3";
(-24,-10)*+{t_4}="t4";
(44,-10)*+{t_5}="t5";
{\ar^{\sig_1} "s1";"0"};
{\ar_{\sig_2} "s2";"0"};
{\ar_{\sig_3} "s3";"2"};
{\ar^{\sig_4} "s4";"3"};
{\ar_{\sig_5} "s5";"5"};
{\ar^{\tau_1} "4";"t1"};
{\ar^{\tau_2} "0";"t2"};
{\ar_{\tau_3} "5";"t3"};
{\ar^{\tau_4} "1";"t4"};
{\ar^{\tau_5} "5";"t5"};
{\ar@/_0.40pc/@{.} (-39,0);(-44,5)}; 
{\ar@/_0.40pc/@{.} (-49,0);(-44,-5)}; 
{\ar@/_0.40pc/@{.} (-29,0);(-24,-5)}; 
{\ar@/_0.40pc/@{.} (19,0);(24,-5)}; 
{\ar@/_0.40pc/@{.} (39,0);(44,-5)}; 
{\ar@/_0.40pc/@{.} (49,0);(44,5)}; 
{\ar@/_0.40pc/@{.} (-2,-5);(5,0)}; 
{\ar@/^0.20pc/@{.} (-8.5,20.5);(-5.5,15.5)}; 
\endxy
\]
For $A\bls=KQ\bls/I\bls$, we have
$\Mfr_{A\bls}=\{\sig_1\al\be\ze\tau_1,\,\sig_2\tau_2,\,\sig_3\gam\thh\kap\tau_3,\,\sig_4\lam\tau_4,\,\sig_5\tau_5\}$.
The set of maximal antipaths in $A\bls$ is
$\{\Del_1=\sig_4\thh\tau_1,\,\Del_2=\sig_1\tau_2,\,\Del_3=\sig_5\tau_3,\,\Del_4=\sig_2\al\tau_4,\,\Del_5=\sig_3\ze\kap\tau_5\}$.
Thus the associated permutation is
$\Phi_A=\left(\begin{array}{ccccc}1&2&3&4&5\\ 4&1&5&2&3\end{array}\right)\in\Sfr_5$.
The Avella-Alaminos--Geiss invariants of $A$ can be calculated as follows.
\[ \phi_A(q,l)=\begin{cases}
1&\ \text{if}\ (q,l)=(0,3)\\
1&\ \text{if}\ (q,l)=(2,2)\\
1&\ \text{if}\ (q,l)=(3,2)\\
0&\ \text{otherwise}
\end{cases} \]
\end{ex}

\section{Finite gentle repetition}\label{section_Repetition}

\begin{dfn}\label{DefQkIk}
Let $(Q,I)$ be a connected locally gentle bound quiver with $|Q_1|>0$, as before. Moreover, assume that $\Mfr_A\amalg\Tfr_A\ne\emptyset$ holds.

For any positive integer $k$, define a connected locally gentle bound quiver $(Q^{(k)},I^{(k)})$ in the following 4 steps.
\begin{enumerate}
\item[\und{Step 1}.] Take blossoming $(Q\bls,I\bls)$ of $(Q,I)$. Let $\{ s_1,\ldots,s_d\}, \{ t_1,\ldots,t_d\}$ and $\sig_i,\tau_i,\wp_i$ be as in Definition \ref{DefBlossom}. By the assumption of $\Mfr_A\amalg\Tfr_A\ne\emptyset$, we have $d=\dfr_A>0$.

\item[\und{Step 2}.] For each $1\le i\le k$, prepare a copy of $(Q\bls,I\bls)$, which we denote by $(Q^{[i]\blossom},I^{[i]\blossom})$. We obtain a (not connected) bound quiver
\[ (\Pcal,J):=\un{1\le i\le k}{\coprod}(Q^{[i]\blossom},I^{[i]\blossom})=\big(\un{1\le i\le k}{\coprod}Q^{[i]\blossom},\un{1\le i\le k}{\bigoplus}I^{[i]\blossom}\big). \]
For any $a\in Q\bls_0$, 
denote corresponding elements in $Q^{[i]\blossom}_0$
by $a^{[i]}$. 
Similarly, the element in $KQ^{[i]\blossom}$ corresponding to $\xi\in KQ\bls$ will be denoted by $\xi^{[i]}$. Namely, we define as follows.
\[
Q^{[i]\blossom}_0=\{a^{[i]}\mid a\in Q\bls_0 \},\ \ 
Q^{[i]\blossom}_1=\{\al^{[i]}\mid \al\in Q\bls_1 \},\ \ 
I^{[i]\blossom}=\langle\{\xi^{[i]}\mid \xi\in I\bls \}\rangle.
\]
Obviously, each $(Q^{[i]\blossom},I^{[i]\blossom})$ naturally contains a copy of $(Q,I)$, which we denote by $(Q^{[i]},I^{[i]})$. This is a locally gentle bound quiver obtained as follows, which has $(Q^{[i]\blossom},I^{[i]\blossom})$ as its blossoming.
\[
Q^{[i]}_0=\{a^{[i]}\mid a\in Q_0 \},\ \ 
Q^{[i]}_1=\{\al^{[i]}\mid \al\in Q_1 \},\ \ 
I^{[i]}=\langle\{\xi^{[i]}\mid \xi\in I \}\rangle.
\]
We put $A^{\blossom [i]}=KQ^{\blossom [i]}/I^{\blossom [i]}$ and $A^{[i]}=KQ^{[i]}/I^{[i]}$ for $1\le i\le k$.

\item[\und{Step 3}.] Modify $(\Pcal,J)$ to obtain a bound quiver $(\Pcal\ppr,J\ppr)$ by the following. Roughly speaking, we {\it weld} blossom arrows $\tau_p^{[i]},\sig_p^{[i+1]}$ together, preserving the relations possessed by them.
\begin{enumerate}
\item[{\rm (i)}] For any $1\le i<k$ and $1\le p\le d$, add a new arrow $\varpi_p^{[i]}$ with $s(\varpi_p^{[i]})=s(\tau_p^{[i]})$ and $t(\varpi_p^{[i]})=t(\sig_p^{[i+1]})$ to the quiver $\Pcal$. Then remove arrows $\tau_p^{[i]},\sig_p^{[i+1]}$ and vertices $t_p^{[i]},s_p^{[i+1]}$.

\item[{\rm (ii)}] Let $\Cfr=\{ \varpi_p^{[i]}\mid 1\le i<k,\ 1\le p\le d \}$ denote the set of {\it connecting arrows}. For any arrows $\xi,\xi\ppr\in\Pcal\ppr_1\setminus\Cfr$ and for any $\varpi_p^{[i]},\varpi_{p\ppr}^{[i\ppr]}\in\Cfr$ we define relations in $J\ppr$ as follows, to make it quadratic monomial.
\begin{eqnarray*}
\xi\varpi_p^{[i]}\in J\ppr\ \EQ\ \xi\tau_p^{[i]}\in J,&&
\varpi_p^{[i]}\xi\in J\ppr\ \EQ\ \sig_p^{[i+1]}\xi\in J,\\
\xi\xi\ppr\in J\ppr\ \EQ\ \xi\xi\ppr\in J,&&
\varpi_p^{[i]}\varpi_{p\ppr}^{[i\ppr]}\in J\ppr\ \EQ\ \sig_p^{[i+1]}\tau_{p\ppr}^{[i\ppr]}\in J.
\end{eqnarray*}
\end{enumerate}
Since $d>0$, those $k$ sheets are connected together by connecting arrows, which means that $\Pcal\ppr$ is connected.

\item[\und{Step 4}.] From $(\Pcal\ppr,J\ppr)$, remove the remaining blossom vertices $s_p^{[1]},t_p^{[k]}$ and blossom arrows $\sig_p^{[1]},\tau_p^{[k]}$ for all $1\le p\le d$, to obtain a bound quiver $(Q^{(k)},I^{(k)})$.
\end{enumerate}
\end{dfn}

We continue to use the notations in Definition \ref{DefQkIk}.
\begin{rem}\label{RemRepet}
By definition, $(Q^{(k)},I^{(k)})$ is a bound subquiver of $(\Pcal\ppr,J\ppr)$. Moreover, the set of vertices in $\Pcal\ppr$ can be decomposed as
\[ \Pcal\ppr_0=
Q^{(k)}_0\amalg\{ s_p^{[1]}\mid 1\le p\le d\}\amalg \{ t_p^{[k]}\mid 1\le p\le d\}. \]
Vertices in each components satisfy the following properties.
\begin{itemize}
\item Any vertex $v\in
Q^{(k)}_0$ satisfies $\din(v)=\dout(v)=2$ in $\Pcal\ppr$.
\item $\{ s_p^{[1]}\mid 1\le p\le d\}$ are the source vertices of $\Pcal\ppr$, and $\sig_p^{[1]}$ is the unique arrow with $s(\sig_p^{[1]})=s_p^{[1]}$ for each $1\le p\le d$.
\item $\{ t_p^{[k]}\mid 1\le p\le d\}$ are the sink vertices of $\Pcal\ppr$, and $\tau_p^{[k]}$ is the unique arrow with $t(\tau_p^{[k]})=t_p^{[k]}$ for each $1\le p\le d$.
\end{itemize}
By these properties, we observe that $(Q^{(k)},I^{(k)})$ remains connected after removing vertices and arrows in Step 4, and its blossoming is given by $(\Pcal\ppr,J\ppr)$ as stated in Remark \ref{RemGentleUnique}. 
Thus we may write $(\Pcal\ppr,J\ppr)=(Q^{(k)\blossom},I^{(k)\blossom})$ in the following.
\end{rem}

\begin{rem}
If $k=1$, we have $A^{(1)}=A$.
\end{rem}

\begin{prop}\label{PropMaxPathInAkb}
For any $1\le p\le d$, the following holds.
\begin{enumerate}
\item There is a unique maximal path $\Pfr_p(A^{(k)})$ in $A^{(k)\blossom}=KQ^{(k)\blossom}/I^{(k)\blossom}$ with $s(\Pfr_p(A^{(k)}))=s_p^{[1]}$. 
Moreover, these are the all maximal paths in $A^{(k)\blossom}$.
\item As in Definition \ref{DefBlossom}, let $\sig_p\wp_p\tau_p$ denotes the maximal path in $A\bls$ which connects $s_p$ to $t_p$. Then, the maximal path $\Pfr_p(A^{(k)})$ in $A^{(k)\blossom}$ is given by the following.
\begin{equation}\label{MaxQkIkb}
\Pfr_p(A^{(k)})=\sig_p^{[1]}\wp_p^{[1]}\varpi_p^{[1]}\wp_p^{[2]}\varpi_p^{[2]}
\cdots\wp_p^{[k-1]}\varpi_p^{[k-1]}\wp_p^{[k]}\tau_p^{[k]}. 
\end{equation}
In particular, we have $t(\Pfr_p(A^{(k)}))=t_p^{[k]}$. Equation $(\ref{MaxQkIkb})$ also implies
\begin{equation}\label{MaxQkIk}
\wp_p(A^{(k)})=\wp_p^{[1]}\varpi_p^{[1]}\wp_p^{[2]}\varpi_p^{[2]}
\cdots\wp_p^{[k-1]}\varpi_p^{[k-1]}\wp_p^{[k]}. 
\end{equation}

\end{enumerate}
\end{prop}
\begin{proof}
{\rm (1)} This is a general property of a blossoming quiver, as stated in Definition \ref{DefBlossom}.

{\rm (2)} $\sig_p\wp_p\tau_p$ yields a sequence of maximal paths
\[ \sig_p^{[1]}\wp_p^{[1]}\tau_p^{[1]},\ \sig_p^{[2]}\wp_p^{[2]}\tau_p^{[2]},\ 
\ldots,\ \sig_p^{[k]}\wp_p^{[k]}\tau_p^{[k]} \]
in $A^{[1]\blossom},A^{[2]\blossom},
\ldots,A^{[k]\blossom}$. In Step 3 we welded $\tau_p^{[i]}$ and $\sig_p^{[i+1]}$ together, which turns the above sequence into a path $(\ref{MaxQkIkb})$. This path starts at the source blossoming vertex $s_p^{[1]}$ and end at the sink blossom vertex $t_p^{[k]}$, which means that it is indeed a maximal path in $A^{(k)\blossom}$.
\end{proof}

\begin{dfn}\label{Def_rhostar_iotaa}
Accordingly to the cases in Definition \ref{DefBlossom} {\rm (4)}, we define as follows.
\begin{itemize}
\item[{\rm (i)}] If $\wp_p=\rho$ for some $\rho\in\Mfr_A$, then we write $\varpi_p^{[i]}=\rho^{\ast[i]}$ for any $1\le i<k$. In this case, $(\ref{MaxQkIkb})$ becomes as follows.
\[ \Pfr_p(A^{(k)})=\sig_p^{[1]}\rho^{[1]}\rho^{\ast [1]}\rho^{[2]}\rho^{\ast[2]}
\cdots\rho^{[k-1]}\rho^{\ast[k-1]}\rho^{[k]}\tau_p^{[k]}. \]
\item[{\rm (ii)}] If $\wp_p=e_a$ for some $a\in\Tfr_A$, then we write $\varpi_p^{[i]}=\iota_a^{[i]}$ for any $1\le i<k$. In this case, $(\ref{MaxQkIkb})$ becomes as follows.
\[ \Pfr_p(A^{(k)})=\sig_p^{[1]}\iota_a^{[1]}\iota_a^{[2]}
\cdots\iota_a^{[k-1]}\tau_p^{[k]}. \]
\end{itemize}
\end{dfn}

\begin{cor}\label{CorMaxPathInAk}
If $k\ge2$, any maximal path in $\Ak$ is one of the following forms.
\begin{itemize}
\item[{\rm (i)}] $\rho^{[1]}\rho^{\ast[1]}\rho^{[2]}\rho^{\ast[2]}\cdots\rho^{\ast[k-1]}\rho^{[k]}$ for some $\rho\in\Mfr_A$.
\item[{\rm (ii)}] $\iota_a^{[1]}\iota_a^{[2]}\cdots\iota_a^{[k-1]}$ for some $a\in\Tfr_A$.
\end{itemize}
\end{cor}
\begin{proof}
This is immediate from Proposition \ref{PropMaxPathInAkb}.
%
\end{proof}

\begin{dfn}\label{Defrhoint}
For any $1\le i<j\le k$, we define as follows.
\begin{itemize}
\item[{\rm (i)}] For $\rho\in\Mfr_A$, put $\rho^{[i,j]}=\rho^{\ast[i]}\rho^{[i+1]}\rho^{\ast[i+1]}\cdots\rho^{[j-1]}\rho^{\ast[j-1]}$.
\item[{\rm (ii)}] For $a\in\Tfr_A$, put $\iota_a^{[i,j]}=\iota_a^{[i]}\iota_a^{[i+1]}\cdots\iota_a^{[j-1]}$.
\end{itemize}
\end{dfn}

\begin{lem}\label{LemBasisSheets}
Oriented cycles and anticycles in $\Ak$ are determined as follows.
\begin{enumerate}
\item The set of oriented cycles in $\Ak$ is
\[ \un{1\le i\le k}{\coprod}\{(\al_1\al_2\cdots\al_l)^{[i]}\mid \al_1\al_2\cdots\al_l\ \text{is an oriented cycle in}\ A\}. \]
In particular if $A$ is gentle, then so is $\Ak$.
\item The set of anticycles in $\Ak$ is
\[ \un{1\le i\le k}{\coprod}\{(\al_1\al_2\cdots\al_l)^{[i]}\mid \al_1\al_2\cdots\al_l\ \text{is an anticycle in}\ A\}. \]
\end{enumerate}
\end{lem}

\begin{proof}
By construction, there is no path from $i$-th sheet to $j$-th sheet if $i>j$. Thus each (anti-) cycle in $\Ak$ should be contained in one sheet $A^{[i]}$, which is a copy of corresponding (anti-) cycle in $A$.
\end{proof}

\begin{prop}\label{PropBasisSheets}
For $1\le i,j\le k$, a $K$-basis of $1^{[i]}\Ak1^{[j]}$ is given by the following.
\begin{enumerate}
\item If $i>j$, we have $1^{[i]}\Ak1^{[j]}=0$.
\item If $i=j$, we naturally have $1^{[i]}A^{(k)}1^{[i]}=A^{[i]}$ as $K$-modules. Since there is an isomorphism of $K$-algebras $(-)^{[i]}\co A\ov{\cong}{\lra}A^{[i]}$, we obtain the following $K$-basis of $1^{[i]}\Ak1^{[i]}$.
\[ \{ \thh^{[i]}\mid \thh\ \text{is a path in}\ A\}. \]
\item If $i<j$, the following gives a $K$-basis of $1^{[i]}\Ak1^{[j]}$.
\begin{equation}\label{BasisSheetsB1}
\Set{(\al_{u+1}\cdots\al_l)^{[i]}\rho^{[i,j]}(\al_1\cdots\al_v)^{[j]}\in \Ak\ | \begin{array}{c}\rho=\al_1\cdots\al_l\in\Mfr_A,\\ 0\le u, v\le l\end{array} }\amalg\{\iota_a^{[i,j]}\mid a\in\Tfr_A\}
\end{equation}
For convenience, here we put $\al_{u+1}\cdots\al_l=e_{t(\rho)}$ if $u=l$, and $\al_1\cdots\al_v=e_{s(\rho)}$ if $v=0$.
\end{enumerate}
\end{prop}
\begin{proof}
{\rm (1)} and {\rm (2)} follow from the fact that there is no path from $i$-th sheet to $j$-th sheet for $i>j$.
Let us show {\rm (3)}. We have a $K$-basis
\begin{equation}\label{BasisSheetsB2}
\Set{ \thh\in \Ak\ | \begin{array}{l}\thh\ \text{is a path in}\ \Ak,\\
\text{with}\ s(\thh)\in Q_0^{[i]}\ \text{and}\ t(\thh)\in Q_0^{[j]} \end{array}}
\end{equation}
of $1^{[i]}A^{(k)}1^{[j]}$. If a path $\thh$ in $A^{(k)}$ satisfies $s(\thh)\in Q_0^{[i]}$ and $t(\thh)\in Q_0^{[j]}$, it is not a part of an oriented cycle in $A^{(k)}$ by Lemma \ref{LemBasisSheets}. Thus such a path should be a subpath of some (unique) maximal path, which is one of the forms {\rm (i),(ii)} by Corollary \ref{CorMaxPathInAk}.
Thus for any path $\thh\in \Ak$ with $s(\thh)\in Q_0^{[i]}$ and $t(\thh)\in Q_0^{[j]}$, the following holds.
\begin{itemize}
\item[{\rm (i)}] If $\thh$ is a subpath of $\rho^{[1]}\rho^{\ast[1]}\rho^{[2]}\cdots\rho^{\ast[k-1]}\rho^{[k]}$, then we have
$\thh=(\al_{u+1}\cdots\al_l)^{[i]}\rho^{[i,j]}(\al_1\cdots\al_v)^{[j]}$
for some $0\le u, v\le l$, where $\rho=\al_1\cdots\al_l$.
\item[{\rm (ii)}] If $\thh$ is a subpath of $\iota_a^{[1]}\iota_a^{[2]}\cdots\iota_a^{[k-1]}$, then we have $\thh=\iota_a^{[i,j]}$.
\end{itemize}
Thus $(\ref{BasisSheetsB2})$ agrees with $(\ref{BasisSheetsB1})$.
\end{proof}

\begin{ex}\label{ExGentleRepet}
Let $A=KQ/I$ be the gentle algebra in Example \ref{ExGentleBlossom}. For $k=3$, its three-times gentle repetition $A^{(3)}$ is given by the following bound quiver.
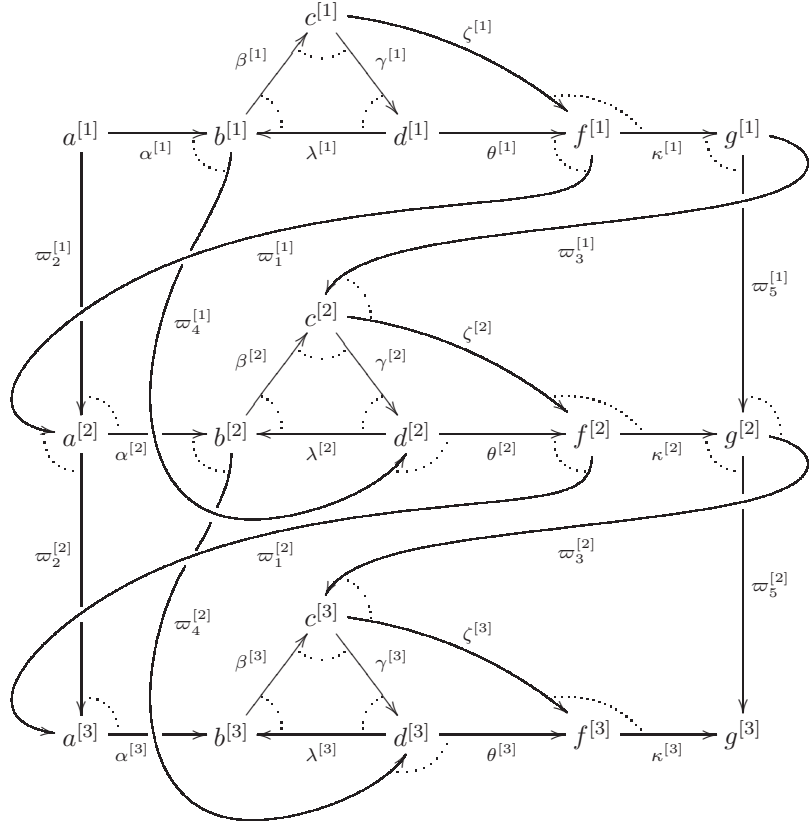
\[
%
\xy
(-44,40)*+{a^{[1]}}="0";
(-24,40)*+{b^{[1]}}="1";
(-12,56)*+{c^{[1]}}="2";
(0,40)*+{d^{[1]}}="3";
(24,40)*+{f^{[1]}}="4";
(44,40)*+{g^{[1]}}="5";
{\ar_{\al^{[1]}} "0";"1"};
{\ar^{\be^{[1]}} "1";"2"};
{\ar^{\gam^{[1]}} "2";"3"};
{\ar@/^0.8pc/^{\ze^{[1]}} "2";"4"};
{\ar^{\lam^{[1]}} "3";"1"};
{\ar_{\thh^{[1]}} "3";"4"};
{\ar_{\kap^{[1]}} "4";"5"};
%
{\ar@/_0.25pc/@{.} (-17.5,40);(-20,45)};
{\ar@/^0.25pc/@{.} (-6.5,40);(-4,45)};
{\ar@/_0.25pc/@{.} (-15,51);(-9,51)};
{\ar@/^0.4pc/@{.} (19,45);(31,40)};
%
%
(-44,0)*+{a^{[2]}}="10";
(-24,0)*+{b^{[2]}}="11";
(-12,16)*+{c^{[2]}}="12";
(0,0)*+{d^{[2]}}="13";
(24,0)*+{f^{[2]}}="14";
(44,0)*+{g^{[2]}}="15";
{\ar|!{(-37.2,0);(-34.2,0)}\hole_(0.34){\al^{[2]}} "10";"11"};
{\ar^{\be^{[2]}} "11";"12"};
{\ar^{\gam^{[2]}} "12";"13"};
{\ar@/^0.8pc/^{\ze^{[2]}} "12";"14"};
{\ar^{\lam^{[2]}} "13";"11"};
{\ar_{\thh^{[2]}} "13";"14"};
{\ar_{\kap^{[2]}} "14";"15"};
%
{\ar@/_0.25pc/@{.} (-17.5,0);(-20,5)};
{\ar@/^0.25pc/@{.} (-6.5,0);(-4,5)};
{\ar@/_0.25pc/@{.} (-15,11);(-9,11)};
{\ar@/^0.4pc/@{.} (19,5);(31,0)};
%
%
(-44,-40)*+{a^{[3]}}="20";
(-24,-40)*+{b^{[3]}}="21";
(-12,-24)*+{c^{[3]}}="22";
(0,-40)*+{d^{[3]}}="23";
(24,-40)*+{f^{[3]}}="24";
(44,-40)*+{g^{[3]}}="25";
{\ar|!{(-37.2,-40);(-34.2,-40)}\hole_(0.34){\al^{[3]}} "20";"21"};
{\ar^{\be^{[3]}} "21";"22"};
{\ar^{\gam^{[3]}} "22";"23"};
{\ar@/^0.8pc/^{\ze^{[3]}} "22";"24"};
{\ar^{\lam^{[3]}} "23";"21"};
{\ar_{\thh^{[3]}} "23";"24"};
{\ar_{\kap^{[3]}} "24";"25"};
%
{\ar@/_0.25pc/@{.} (-17.5,-40);(-20,-35)};
{\ar@/^0.25pc/@{.} (-6.5,-40);(-4,-35)};
{\ar@/_0.25pc/@{.} (-15,-29);(-9,-29)};
{\ar@/^0.4pc/@{.} (19,-35);(31,-40)};
%
%
"4";"10" **\crv{(24,34)&(24,31)&(-40,28)&(-62,0)} ?>*\dir{>};
{\ar|!{(-44,13);(-44,16)}\hole_(0.4){\varpi_2^{[1]}} "0";"10"};
"5";"12" **\crv{(52,40)&(62,30)&(-12,26)&(-12,18)} ?>*\dir{>};
{\ar|!{(44,28);(44,31)}\hole^{\varpi_5^{[1]}} "5";"15"};
(-30,23.5)*+{}="hole12";
"1";"hole12" **\crv{(-24,36)&(-24,34)} ?>*\dir{};
"hole12";"13" **\crv{(-32,20)&(-38,4)&(-28,-18)&(0,-6)} ?>*\dir{>};
(-18,24)*+{_{\varpi_1^{[1]}}}="pi1";
(22,24.5)*+{_{\varpi_3^{[1]}}}="pi3";
(-29,15)*+{_{\varpi_4^{[1]}}}="pi4";
%
%
"14";"20" **\crv{(24,-6)&(24,-9)&(-40,-12)&(-62,-40)} ?>*\dir{>};
{\ar|!{(-44,-27);(-44,-24)}\hole_(0.4){\varpi_2^{[2]}} "10";"20"};
"15";"22" **\crv{(52,0)&(62,-10)&(-12,-14)&(-12,-22)} ?>*\dir{>};
{\ar|!{(44,-12);(44,-9)}\hole^{\varpi_5^{[2]}} "15";"25"};
(-26.5,-10.8)*+{}="hole23plus";
"11";"hole23plus" **\crv{(-24,-4)&(-24,-6)} ?>*\dir{};
(-29.5,-16.5)*+{}="hole23";
"hole23plus";"hole23" **\crv{(-27,-13)} ?>*\dir{};
"hole23";"23" **\crv{(-32,-20)&(-38,-36)&(-28,-58)&(0,-46)} ?>*\dir{>};
(-18,-16)*+{_{\varpi_1^{[2]}}}="pi1";
(22,-15.5)*+{_{\varpi_3^{[2]}}}="pi3";
(-29,-25)*+{_{\varpi_4^{[2]}}}="pi4";%
{\ar@/_0.40pc/@{.} (-29,40);(-25,35)}; 
{\ar@/_0.40pc/@{.} (19,40);(23,35)}; 
{\ar@/_0.40pc/@{.} (39,40);(44,35)}; 
%
{\ar@/_0.40pc/@{.} (-39,0);(-44,5)}; 
{\ar@/_0.40pc/@{.} (-49,0);(-44,-5)}; 
{\ar@/_0.40pc/@{.} (-29,0);(-25,-5)}; 
{\ar@/_0.40pc/@{.} (19,0);(23,-5)}; 
{\ar@/_0.40pc/@{.} (39,0);(44,-5)}; 
{\ar@/_0.40pc/@{.} (49,0);(44,5)}; 
{\ar@/_0.40pc/@{.} (-2,-5);(5,0)}; 
{\ar@/^0.20pc/@{.} (-8.5,20.5);(-5.5,15.5)}; 
{\ar@/_0.40pc/@{.} (-39,-40);(-44,-35)}; 
{\ar@/_0.40pc/@{.} (-2,-45);(5,-40)}; 
{\ar@/^0.20pc/@{.} (-8.5,-19.5);(-5.5,-24.5)}; 
\endxy
\]
The connecting arrows are
\[ \varpi_1^{[i]}=(\al\be\ze)^{\ast[i]},\ \varpi_2^{[i]}=\iota_a^{[i]},\ \varpi_3^{[i]}=(\gam\thh\kap)^{\ast[i]},\ \varpi_4^{[i]}=\lam^{\ast[i]},\ \varpi_5^{[i]}=\iota_g^{[i]} \]
for $i=1,2$.

Its blossoming becomes as follows. We observe that the maximal paths in $A\bls$ are
\begin{eqnarray*}
\Pfr_1&=&\sig_1^{[1]}\al^{[1]}\be^{[1]}\ze^{[1]}(\al\be\ze)^{\ast[1]}\al^{[2]}\be^{[2]}\ze^{[2]}(\al\be\ze)^{\ast[2]}\al^{[3]}\be^{[3]}\ze^{[3]}\tau_1^{[3]},\\
\Pfr_2&=&\sig_2^{[1]}\iota_a^{[1]}\iota_a^{[2]}\tau_2^{[3]},\\
\Pfr_3&=&\sig_3^{[1]}\gam^{[1]}\thh^{[1]}\kap^{[1]}(\gam\thh\kap)^{\ast[1]}\gam^{[2]}\thh^{[2]}\kap^{[2]}(\gam\thh\kap)^{\ast[2]}\gam^{[3]}\thh^{[3]}\kap^{[3]}\tau_3^{[3]},\\
\Pfr_4&=&\sig_4^{[1]}\lam^{[1]}\lam^{\ast[1]}\lam^{[2]}\lam^{\ast[2]}\lam^{[3]}\tau_4^{[3]},\\
\Pfr_5&=&\sig_5^{[1]}\iota_g^{[1]}\iota_g^{[2]}\tau_5^{[3]}.
\end{eqnarray*}

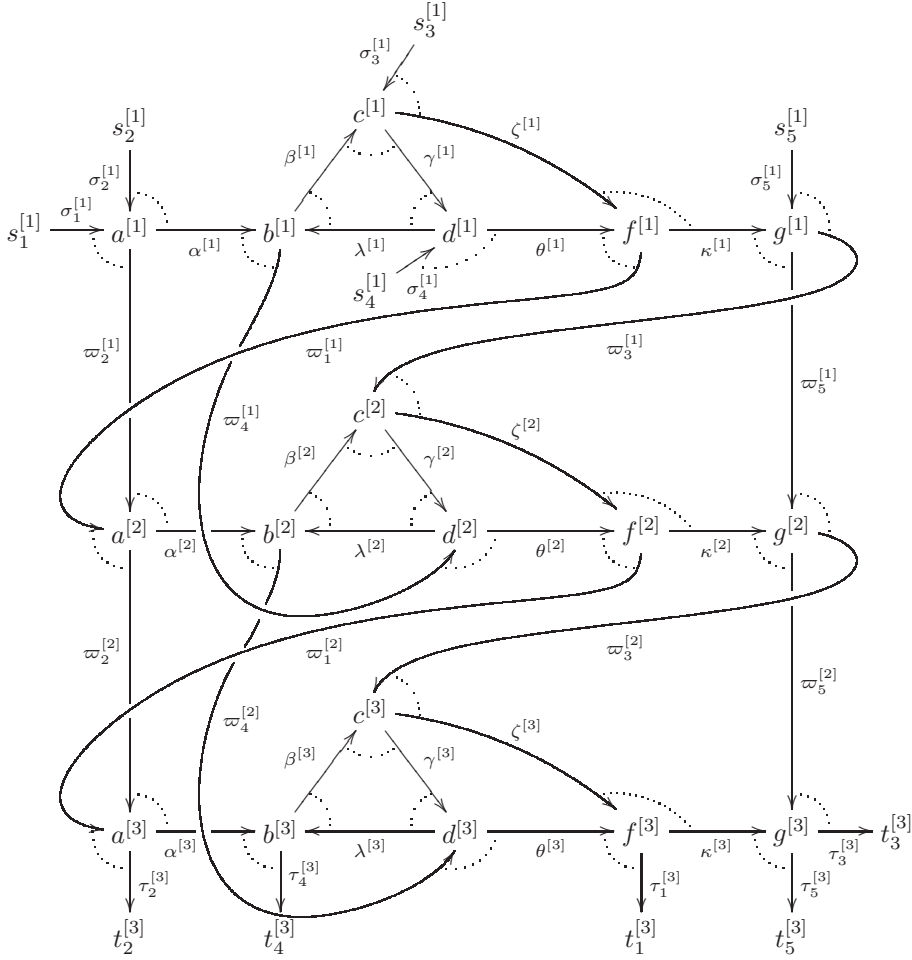
\[
%
\xy
(-44,40)*+{a^{[1]}}="0";
(-24,40)*+{b^{[1]}}="1";
(-12,56)*+{c^{[1]}}="2";
(0,40)*+{d^{[1]}}="3";
(24,40)*+{f^{[1]}}="4";
(44,40)*+{g^{[1]}}="5";
{\ar_{\al^{[1]}} "0";"1"};
{\ar^{\be^{[1]}} "1";"2"};
{\ar^{\gam^{[1]}} "2";"3"};
{\ar@/^0.8pc/^{\ze^{[1]}} "2";"4"};
{\ar^{\lam^{[1]}} "3";"1"};
{\ar_{\thh^{[1]}} "3";"4"};
{\ar_{\kap^{[1]}} "4";"5"};
%
{\ar@/_0.25pc/@{.} (-17.5,40);(-20,45)};
{\ar@/^0.25pc/@{.} (-6.5,40);(-4,45)};
{\ar@/_0.25pc/@{.} (-15,51);(-9,51)};
{\ar@/^0.4pc/@{.} (19,45);(31,40)};
%
%
(-44,0)*+{a^{[2]}}="10";
(-24,0)*+{b^{[2]}}="11";
(-12,16)*+{c^{[2]}}="12";
(0,0)*+{d^{[2]}}="13";
(24,0)*+{f^{[2]}}="14";
(44,0)*+{g^{[2]}}="15";
{\ar|!{(-37.2,0);(-34.2,0)}\hole_(0.34){\al^{[2]}} "10";"11"};
{\ar^{\be^{[2]}} "11";"12"};
{\ar^{\gam^{[2]}} "12";"13"};
{\ar@/^0.8pc/^{\ze^{[2]}} "12";"14"};
{\ar^{\lam^{[2]}} "13";"11"};
{\ar_{\thh^{[2]}} "13";"14"};
{\ar_{\kap^{[2]}} "14";"15"};
%
{\ar@/_0.25pc/@{.} (-17.5,0);(-20,5)};
{\ar@/^0.25pc/@{.} (-6.5,0);(-4,5)};
{\ar@/_0.25pc/@{.} (-15,11);(-9,11)};
{\ar@/^0.4pc/@{.} (19,5);(31,0)};
%
%
(-44,-40)*+{a^{[3]}}="20";
(-24,-40)*+{b^{[3]}}="21";
(-12,-24)*+{c^{[3]}}="22";
(0,-40)*+{d^{[3]}}="23";
(24,-40)*+{f^{[3]}}="24";
(44,-40)*+{g^{[3]}}="25";
{\ar|!{(-37.2,-40);(-34.2,-40)}\hole_(0.34){\al^{[3]}} "20";"21"};
{\ar^{\be^{[3]}} "21";"22"};
{\ar^{\gam^{[3]}} "22";"23"};
{\ar@/^0.8pc/^{\ze^{[3]}} "22";"24"};
{\ar^{\lam^{[3]}} "23";"21"};
{\ar_{\thh^{[3]}} "23";"24"};
{\ar_{\kap^{[3]}} "24";"25"};
%
{\ar@/_0.25pc/@{.} (-17.5,-40);(-20,-35)};
{\ar@/^0.25pc/@{.} (-6.5,-40);(-4,-35)};
{\ar@/_0.25pc/@{.} (-15,-29);(-9,-29)};
{\ar@/^0.4pc/@{.} (19,-35);(31,-40)};
%
%
"4";"10" **\crv{(24,34)&(24,31)&(-40,28)&(-62,0)} ?>*\dir{>};
{\ar|!{(-44,13);(-44,16)}\hole_(0.4){\varpi_2^{[1]}} "0";"10"};
"5";"12" **\crv{(52,40)&(62,30)&(-12,26)&(-12,18)} ?>*\dir{>};
{\ar|!{(44,28);(44,31)}\hole^{\varpi_5^{[1]}} "5";"15"};
(-30,23.5)*+{}="hole12";
"1";"hole12" **\crv{(-24,36)&(-24,34)} ?>*\dir{};
"hole12";"13" **\crv{(-32,20)&(-38,4)&(-28,-18)&(0,-6)} ?>*\dir{>};
(-18,24)*+{_{\varpi_1^{[1]}}}="pi1";
(22,24.5)*+{_{\varpi_3^{[1]}}}="pi3";
(-29,15)*+{_{\varpi_4^{[1]}}}="pi4";
%
%
"14";"20" **\crv{(24,-6)&(24,-9)&(-40,-12)&(-62,-40)} ?>*\dir{>};
{\ar|!{(-44,-27);(-44,-24)}\hole_(0.4){\varpi_2^{[2]}} "10";"20"};
"15";"22" **\crv{(52,0)&(62,-10)&(-12,-14)&(-12,-22)} ?>*\dir{>};
{\ar|!{(44,-12);(44,-9)}\hole^{\varpi_5^{[2]}} "15";"25"};
(-26.5,-10.8)*+{}="hole23plus";
"11";"hole23plus" **\crv{(-24,-4)&(-24,-6)} ?>*\dir{};
(-29.5,-16.5)*+{}="hole23";
"hole23plus";"hole23" **\crv{(-27,-13)} ?>*\dir{};
"hole23";"23" **\crv{(-32,-20)&(-38,-36)&(-28,-58)&(0,-46)} ?>*\dir{>};
(-18,-16)*+{_{\varpi_1^{[2]}}}="pi1";
(22,-15.5)*+{_{\varpi_3^{[2]}}}="pi3";
(-29,-25)*+{_{\varpi_4^{[2]}}}="pi4";%
%
%
(-58,40)*+{s_1^{[1]}}="s1";
(-44,54)*+{s_2^{[1]}}="s2";
(-4,68)*+{s_3^{[1]}}="s3";
(-12,32)*+{s_4^{[1]}}="s4";
(44,54)*+{s_5^{[1]}}="s5";
{\ar^{\sig_1^{[1]}} "s1";"0"};
{\ar_{\sig_2^{[1]}} "s2";"0"};
{\ar_{\sig_3^{[1]}} "s3";"2"};
{\ar_(0.4){\sig_4^{[1]}} "s4";"3"};
{\ar_{\sig_5^{[1]}} "s5";"5"};
%
{\ar@/_0.40pc/@{.} (-39,40);(-44,45)}; 
{\ar@/_0.40pc/@{.} (-49,40);(-44,35)}; 
{\ar@/_0.40pc/@{.} (-29,40);(-25,35)}; 
{\ar@/_0.40pc/@{.} (19,40);(23,35)}; 
{\ar@/_0.40pc/@{.} (39,40);(44,35)}; 
{\ar@/_0.40pc/@{.} (49,40);(44,45)}; 
{\ar@/_0.40pc/@{.} (-5,35);(5,40)}; 
{\ar@/^0.20pc/@{.} (-8.5,60.5);(-5.5,55.5)}; 
%
{\ar@/_0.40pc/@{.} (-39,0);(-44,5)}; 
{\ar@/_0.40pc/@{.} (-49,0);(-44,-5)}; 
{\ar@/_0.40pc/@{.} (-29,0);(-25,-5)}; 
{\ar@/_0.40pc/@{.} (19,0);(23,-5)}; 
{\ar@/_0.40pc/@{.} (39,0);(44,-5)}; 
{\ar@/_0.40pc/@{.} (49,0);(44,5)}; 
{\ar@/_0.40pc/@{.} (-2,-5);(5,0)}; 
{\ar@/^0.20pc/@{.} (-8.5,20.5);(-5.5,15.5)}; 
(24,-54)*+{t_1^{[3]}}="t1";
(-44,-54)*+{t_2^{[3]}}="t2";
(58,-40)*+{t_3^{[3]}}="t3";
(-24,-54)*+{t_4^{[3]}}="t4";
(44,-54)*+{t_5^{[3]}}="t5";
%
%
{\ar^{\tau_1^{[3]}} "24";"t1"};
{\ar^{\tau_2^{[3]}} "20";"t2"};
{\ar_{\tau_3^{[3]}} "25";"t3"};
{\ar^(0.4){\tau_4^{[3]}} "21";"t4"};
{\ar^{\tau_5^{[3]}} "25";"t5"};
{\ar@/_0.40pc/@{.} (-39,-40);(-44,-35)}; 
{\ar@/_0.40pc/@{.} (-49,-40);(-44,-45)}; 
{\ar@/_0.40pc/@{.} (-29,-40);(-24,-45)}; 
{\ar@/_0.40pc/@{.} (19,-40);(24,-45)}; 
{\ar@/_0.40pc/@{.} (39,-40);(44,-45)}; 
{\ar@/_0.40pc/@{.} (49,-40);(44,-35)}; 
{\ar@/_0.40pc/@{.} (-2,-45);(5,-40)}; 
{\ar@/^0.20pc/@{.} (-8.5,-19.5);(-5.5,-24.5)}; 
\endxy
\]
\end{ex}

We remark that the procedure of welding in Definition \ref{DefQkIk} can be slightly generalized to a pair of locally gentle bound quivers, as follows.
\begin{dfn}\label{DefWeld}
Let $(Q,I)$ be a locally gentle bound quiver, $(Q\bls,I\bls)$ be its blossoming, and let $A=KQ/I$ be the associated bound quiver algebra, as before.

Let $(Q\ppr,I\ppr)$ be another locally gentle bound quiver with its blossoming $(Q^{\prime\blossom},I^{\prime\blossom})$, and let $B=KQ/I$ be the associated bound quiver algebra. To distinguish, let $\{s_p^A\mid1\le p\le \dfr_A\}$ (respectively, $\{s_p^B\mid1\le p\le \dfr_B\}$) denote the set of blossom source vertices in $Q\bls$ (resp. $Q^{\prime\blossom}$), and let $\sig_p^A$ (resp. $\sig_p^B$) denote the arrow satisfying $s(\sig_p^A)=s_p^A$ (resp. $s(\sig_p^B)=s_p^B$). Similarly for $t_p^A,t_p^B,\tau_p^A,\tau_p^B$ for each $p$.

Assume that they satisfy $\dfr_A=\dfr_B>0$, which we put $d$. If we are given a permutation $\sfr\in\Sfr_d$, then the {\it welding} of $A$ and $B$ by $\sfr$ is a locally gentle bound quiver algebra $A\un{\sfr}{\ci}B$ associated to the bound quiver $(Q\un{\sfr}{\ci}Q\ppr,I\un{\sfr}{\ci}I\ppr)$ defined by the following.
\begin{itemize}
\item $(Q\un{\sfr}{\ci}Q\ppr)_0=Q_0\amalg Q_0\ppr$.
\item $(Q\un{\sfr}{\ci}Q\ppr)_1=Q_1\amalg Q_1\ppr\amalg\{\varpi_p\mid1\le p\le d\}$, with
\[ s(\varpi_p)=s(\tau_p^A)\ \ \text{and}\ t(\varpi_p)=t(\sig_{\sfr(p)}^B).  \]
For each $\al\in Q_1$, its source $s(\al)$ and target $t(\al)$ in $Q\un{\sfr}{\ci}Q\ppr$ are the same as those in $Q$. Similarly for arrows $\be\in Q\ppr_1$.
\item $I\un{\sfr}{\ci}I\ppr$ is generated by the following set of paths.
\begin{eqnarray*}
&&\{\al\be\mid \al,\be\in Q_1,t(\al)=s(\be)\ \text{and}\ \al\be\in I\}\amalg
\{\al\be\mid\al,\be\in Q_1\ppr,t(\al)=s(\be)\ \text{and}\ \al\be\in I\ppr\}\\
&&\!\!\amalg\,\{\al\varpi_p\mid\al\in Q_1,1\le p\le d\ \text{and}\ \al\tau_p^A\in I\bls \}\amalg\{\varpi_p\be\mid\be\in Q\ppr_1,1\le p\le d\ \text{and}\ \sig_{\sfr(p)}^B\be\in I^{\ppr\blossom} \}.
\end{eqnarray*}
\end{itemize}
\end{dfn}

\begin{rem}
Especially if $A=B$, we may take the same labeling of blossom vertices for $Q\bls=Q^{\prime\blossom}$, and take the identity permutation $\sfr=\id$. In this case, we have $A^{(2)}=A\un{\id}{\ci}A$. Moreover by Proposition \ref{PropMaxPathInAkb}, we may also construct $A^{(k)}$ inductively as
\[ A^{(k)}=A^{(k-1)}\un{\id}{\ci}A=\cdots=((((A\un{\id}{\ci}A)\un{\id}{\ci}A)\un{\id}{\ci}\cdots)\un{\id}{\ci}A)\un{\id}{\ci}A. \]
\end{rem}

\section{Avella-Alaminos--Geiss invariants of gentle repetitive algebras}\label{section_AGRepetition}

In this section, we assume that $A=KQ/I$ is a gentle algebra, and let $k$ be a positive integer. By Lemma \ref{LemBasisSheets} {\rm (1)}, it follows that $\Ak=KQ^{(k)}/I^{(k)}$ is also gentle.
\begin{lem}\label{LemAGRepet}
Let $k$ be any positive integer.
\begin{enumerate}
\item Each anticycle in $A$ of length $l\in\Np$ yields $k$ copies of anticycles in $\Ak$ of length $l$.
All anticycles in $\Ak$ arise in this way.
\item Each AG-orbit in $A$ of type $(q,l)\in\Np\ti\Nn$ yields $\gcd(q,k)$ pieces of AG-orbits in $\Ak$ of type $(\frac{L}{k},\frac{L}{q}l+\frac{L}{k}(k-1))$, where $L=\lcm(q,k)$ is the least common multiplier and $\gcd(q,k)$ is the greatest common divisor of $q$ and $k$.
All AG-orbits in $\Ak$ arise in this way.
\end{enumerate}
\end{lem}
\begin{proof}
{\rm (1)} This immediately follows from Lemma \ref{LemBasisSheets} {\rm (2)}.

{\rm (2)} We continue to use the notations in Definition \ref{DefQkIk}. 
Remark that we have $\dfr_A=\dfr_{A^{(k)}}$, which we put $d$.
Also, let $\Del_p(A)=\Del_p=\sig_{\Phi(p)}\del_p\tau_p$ be the maximal antipath connecting $s_{\Phi(p)}$ to $t_p$.

Then $s_1^{[1]},\ldots,s_d^{[1]}$ give the blossom source vertices, and $t_1^{[k]},\ldots,t_d^{[k]}$ give the blossom sink vertices for $(Q^{(k)\blossom},I^{(k)\blossom})$ respectively. For $1\le p\le d$, we have
\begin{equation}\label{EqP}
\wp_p(A^{(k)})=\wp_p^{[1]}\varpi_p^{[1]}\wp_p^{[2]}\varpi_p^{[2]}\cdots\wp_p^{[k-1]}\varpi_p^{[k-1]}\wp_p^{[k]}
\end{equation}
by Proposition \ref{PropMaxPathInAkb} {\rm (2)}, where $\wp_p=\wp_p(A)$.

For any $1\le p\le d$, each $\Del_p(\Ak)$ becomes
\[ \Del_p(\Ak)=\sig_{\Phi^k(p)}^{[1]}(\del_{\Phi^{k-1}(p)})^{[1]}\varpi_{\Phi^{k-1}(p)}^{[1]}(\del_{\Phi^{k-2}(p)})^{[2]}\varpi_{\Phi^{k-2}(p)}^{[2]}\cdots(\del_{\Phi(p)})^{[k-1]}\varpi_{\Phi(p)}^{[k-1]}(\del_p)^{[k]}\tau_p^{[k]}, \]
and thus we obtain
\begin{equation}\label{EqA}
\del_p(\Ak)=(\del_{\Phi^{k-1}(p)})^{[1]}\varpi_{\Phi^{k-1}(p)}^{[1]}(\del_{\Phi^{k-2}(p)})^{[2]}\varpi_{\Phi^{k-2}(p)}^{[2]}\cdots(\del_{\Phi(p)})^{[k-1]}\varpi_{\Phi(p)}^{[k-1]}(\del_p)^{[k]}.
\end{equation}
In particular we have
\begin{equation}\label{Eq_LLL}
\length\big(\del_p(\Ak)\big)=(k-1)+\sum_{0\le m\le k-1}\length\big(\del_{\Phi^{m}(p)}(A)\big),
\end{equation}
and the permutation $\Psi=\Phi_{A^{(k)}}$ for $\Ak$ is given by $\Psi=\Phi^k\in\Sfr_d$.

For any AG-orbit in $A$
\[ \Ocal=\{\del_p(A),\del_{\Phi(p)}(A),\ldots,\del_{\Phi^{q-1}(p)}(A) \} \]
of type $(q,l)$, put $c(\Ocal)=\gcd(q,k)$. Remark that we have $q=|\Ocal|$ by definition. Each such $\Ocal$ yields AG-orbits $\widetilde{\Ocal}_{1},\widetilde{\Ocal}_{2},\ldots,\widetilde{\Ocal}_{c(\Ocal)}$ in $\Ak$ given by
\begin{eqnarray*}
\widetilde{\Ocal}_{r}&=&\{ \del(\Ak)_{\Phi^r(p)},\del(\Ak)_{\Psi(\Phi^r(p))},\ldots,\del(\Ak)_{\Psi^{\frac{L}{k}-1}(\Phi^r(p))}\}\\
&=&\{ \del(\Ak)_{\Phi^r(p)},\del(\Ak)_{\Phi^{k+r}(p)},\ldots,\del(\Ak)_{\Phi^{L-k+r}(p)}\}
\end{eqnarray*}
for any $1\le r\le c(\Ocal)$, where $L=\lcm(q,k)$. In particular we have $|\widetilde{\Ocal}_r|=\frac{L}{k}$.
By $(\ref{Eq_LLL})$, for each $1\le r\le\gcd(q,k)$, we have
\begin{eqnarray*}
\sum_{0\le i\le\frac{L}{k}-1}\length(\del_{\Psi^i(\Phi^r(p))}(\Ak))
&=&\frac{L}{k}(k-1)+\sum_{0\le i\le\frac{L}{k}-1}\bigg(\sum_{0\le m\le k-1}\length(\del_{\Phi^{ik+r+m}(p)}(A))\bigg)\\
&=&\frac{L}{k}(k-1)+\sum_{0\le n\le L-1}\length(\del_{\Phi^{n}(p)}(A))\\
&=&\frac{L}{k}(k-1)+\frac{L}{q}\sum_{0\le n\le q-1}\length(\del_{\Phi^{n}(p)}(A))\\
&=&\frac{L}{k}(k-1)+\frac{L}{q}l,
\end{eqnarray*}
which means that the type of $\widetilde{\Ocal}_r$ is $(\frac{L}{k},\frac{L}{k}(k-1)+\frac{L}{q}l)$.

Let $\mathrm{Orb}(A)$ denote the set of all AG-orbits in $A$. We obtain
\[ \sum_{\Ocal\in\mathrm{Orb}(A)}\sum_{1\le r\le c(\Ocal)}|\widetilde{\Ocal}_r|
=\sum_{\Ocal\in\mathrm{Orb}(A)}c(\Ocal)\frac{\lcm(|\Ocal|,k)}{k}=\sum_{\Ocal\in\mathrm{Orb}(A)}|\Ocal|=\dfr_A. \]
Since $\dfr_A=\dfr_{A^{(k)}}$, this shows $\mathrm{Orb}(A^{(k)})=\{\widetilde{\Ocal}_r\mid \Ocal\in\mathrm{Orb}(A),\ 1\le r\le c(\Ocal) \}$. Thus any AG-orbit in $A^{(k)}$ arises as one of these $\widetilde{\mathcal{O}_r}$.
\end{proof}

\begin{prop}\label{PropAGRepet1}
Let $k$ be any positive integer.
\begin{enumerate}
\item For any $(n,m)\in\Np\ti\Nn$, we have
\[ \phi_{\Ak}(n,m)=\sum_{(q,l)\in S_{(k)}(n,m)}\gcd(q,k)\phi_A(q,l), \]
where we put
\[ S_{(k)}(n,m)=\Set{(q,l)\in\Np\ti\Nn | \begin{array}{l}(n,m)=(\frac{L}{k},\frac{L}{q}l+\frac{L}{k}(k-1)),\\
L=\lcm(q,k)
\end{array}}. \]
\item For any $l\in\Np$, we have
\[ \phi_{\Ak}(0,l)=k\phi_A(0,l). \]
\end{enumerate}
\end{prop}
\begin{proof}
This is immediate from Lemma \ref{LemAGRepet}.
\end{proof}

\begin{prop}\label{PropAGRepet2}
Let $(q,l)\in\Np\ti\Nn$ be any element, and put $d=\gcd(q,l)$. We have
\begin{equation}\label{Eq_AGRepet}
\phi_A(q,l)=\frac{1}{q}\sum_{c|d}\mu(c)\phi_{A^{(\frac{q}{c})}}\big(1,\frac{q+l}{c}-1\big).
\end{equation}
Here $\mu$ is the M\"{o}bius function, which is
\[ \mu(c)=\begin{cases}
0&\text{if $c$ has a square factor},\\
(-1)^r&\text{if $c$ is a product of different $r$ primes}.\end{cases}
\]
\end{prop}
\begin{proof}
Put $(q_0,l_0)=(\frac{q}{d},\frac{l}{d})$.
By Proposition \ref{PropAGRepet1}, for any $n\in\Np$ we have
\[ \phi_{A^{(nq_0)}}(1,n(q_0+l_0)-1)=\sum_{(q\ppr,l\ppr)\in S_{(nq_0)}(1,n(q_0+l_0)-1)}\gcd(q\ppr,nq_0)\phi_A(q\ppr,l\ppr). \]
Since
\begin{eqnarray*}
&&\hspace{-1cm}S_{(nq_0)}(1,n(q_0+l_0)-1)\\
&=&\Set{(q\ppr,l\ppr)\in\Np\ti\Nn |
\begin{array}{l}
\lcm(q\ppr,nq_0)=nq_0,\\
n(q_0+l_0)-1=nq_0(\frac{l\ppr}{q\ppr}+\frac{nq_0-1}{nq_0})
\end{array}
}\\
&=&\{(q\ppr,l\ppr)\in\Np\ti\Nn \mid q\ppr\ \text{is a divisor of}\ nq_0,\ \text{satisfying}\ q\ppr l_0=q_0l\ppr \}\\
&=&\{(q\ppr,l\ppr)\in\Np\ti\Nn \mid nq_0=uq\ppr\ \text{and}\ nl_0=ul\ppr\ \text{for some}\ u\in\Np \}\\
&=&\Set{(\frac{n}{u}q_0,\frac{n}{u} l_0)\in\Np\ti\Nn | u\in\Np,\ u\ \text{is a divisor of}\ n},
\end{eqnarray*}
we obtain
\[ \phi_{A^{(nq_0)}}(1,n(q_0+l_0)-1)=\sum_{u|n}\frac{n}{u}q_0\phi_A(\frac{n}{u}q_0,\frac{n}{u}l_0). \]
Thus the M\"{o}bius inversion formula gives
\[ nq_0\phi_A(nq_0,nl_0)=\sum_{c|n}\mu(c)\phi_{A^{(\frac{nq_0}{c})}}(1,\frac{n}{c}(q_0+l_0)-1) \]
for any $n\in\Np$. Especially for $n=d$, we obtain $(\ref{Eq_AGRepet})$.
\end{proof}

\begin{cor}\label{CorAGRepetHoch}
Let $(q,l)\in\Np\ti\Nn$ be any element, and put $d=\gcd(q,l)$. If $(Q,I)$ has no anticycle, then $\phi_A(q,l)$ can be written as follows.
\begin{enumerate}
\item If $q\ne l$ and $l>0$, then
\[\phi_A(q,l)=\frac{1}{q}\sum_{c|d}\mu(c)\dim_K\HH^{\frac{q+l}{c}-1}(A^{(\frac{q}{c})}). \]
\item If $q=l>0$, then
\[\phi_A(q,q)=\frac{\mu(q)}{q}(\chi(Q)-1)+\frac{1}{q}\sum_{c|q}\mu(c)\dim_K\HH^{\frac{2q}{c}-1}(A^{(\frac{q}{c})}). \]
\item If $l=0$ and $q$ is odd, then
\[\phi_A(q,0)=-\frac{\mu(q)}{q}+\frac{1}{q}\sum_{c|q}\mu(c)\dim_K\HH^{\frac{q}{c}-1}(A^{(\frac{q}{c})}). \]
\item If $l=0$ and $q$ is even, then
\[ \phi_A(q,0)
=-\frac{\mu(q)}{q}-\frac{\mu(q/2)}{q}(|Q_1|+1)+\frac{1}{q}\sum_{c|q}\mu(c)\dim_K\HH^{\frac{q}{c}-1}(A^{(\frac{q}{c})}).
\]
\end{enumerate}
\end{cor}
\begin{proof}
For any $k\in\Np$, algebra $\Ak$ has no anticycle by Lemma \ref{LemBasisSheets} {\rm (2)}. Thus Fact \ref{FactLadkani}, Proposition \ref{PropAGRepet2} and the equality $\chi(Q^{(2)})=|Q_0^{(2)}|-|Q_1^{(2)}|=2|Q_0|-(2|Q_1|+\dfr_A)=-|Q_1|$ show the desired equalities.
\end{proof}

\begin{cor}\label{CorCoprime}
Assume that $(Q,I)$ has no anticycle.
Let $(q,l)\in(\Np\ti\Np)\setminus\{(1,1)\}$ be any element. If $q$ and $l$ are coprime, then we have
\[ \phi_A(q,l)=\frac{1}{q}\dim\HH^{q+l-1}(A^{(q)}). \]
\end{cor}
\begin{proof}
This immediately follows from Corollary \ref{CorAGRepetHoch} {\rm (1)}.
\end{proof}

\section{Relation with generalized Auslander-Platzeck-Reiten reflection}\label{section_APR}

In this section, we assume that $A=KQ/I$ is gentle. We recall the definition of generalized Auslander-Platzeck-Reiten reflection for gentle bound quivers from \cite{BM}.
\begin{dfn}\label{DefAPRreflection}
Let $(Q,I)$ be a gentle bound quiver. Let $x\in Q_0$ be a vertex which satisfies the following conditions.
\begin{itemize}
\item[{\rm (r1)}] There is no $\al\in Q_1$ with $s(\al)=t(\al)=x$.
\item[{\rm (r2)}] For any $\al\in Q_1$ with $s(\al)=x$, there exists $\be_{\al}\in Q_1$ such that $t(\be_{\al})=x$ and $\be_{\al}\al\notin I$.
\end{itemize}
The bound quiver $(Q\ppr,I\ppr)$ obtained from $(Q,I)$ by applying the {\it generalized Auslander-Platzeck-Reiten reflection} ({\it generalized APR-reflection} for short) at $x$ is defined as follows.
\begin{enumerate}
\item[{\rm (0)}] $Q\ppr_0=Q_0$.
\item[{\rm (1)}] $Q\ppr_1=\{\lam\ppr\mid \lam\in Q_1\}$, where $\lam\ppr\in Q\ppr_1$ is an arrow with the following $s(\lam\ppr)$ and $t(\lam\ppr)$.
\[ s(\lam\ppr)=\begin{cases}
x&\ \text{if}\ t(\lam)=x\\
s(\be_{\lam})&\ \text{if}\ s(\lam)=x\\
s(\lam)&\ \text{otherwise}
\end{cases}\ ,\qquad
t(\lam\ppr)=\begin{cases}
s(\lam)&\ \text{if}\ t(\lam)=x\\
x&\ \text{if}\ \lam\ \text{satisfies}\ (\ref{cond_refl1})\ \text{for some}\ \be\in Q_1\\
t(\lam)&\ \text{otherwise}
\end{cases}\ .
\]
Here, condition $(\ref{cond_refl1})$ is the following.
\begin{equation}\label{cond_refl1}
t(\be)=x,\ s(\be)=t(\lam)\ \text{and}\ \lam\be\in I.
\end{equation}
\item[{\rm (2)}] $I\ppr$ is generated by the following set $R\ppr$ of paths of length $2$.
\begin{eqnarray*}
R\ppr&=&\{\lam\ppr\kap\ppr\mid\lam,\kap\in Q_1,\ t(\kap)\ne x,\ t(\lam)=s(\kap)\ne x\ \text{and}\ \lam\kap\in I\}\\
&\ &\!\!\amalg\ \{(\be_{\al})\ppr\al\ppr\mid\al\in Q_1,\ s(\al)=x\}\\
&\ &\!\!\amalg\ \{\lam\ppr\kap\ppr\mid \lam,\kap\in Q_1,\ t(\kap)=x,\ \text{and}\ \lam\ \text{satisfies}\ (\ref{cond_refl1})\ \text{for some}\ \be\ne\kap \}.
\end{eqnarray*}
\end{enumerate}
\end{dfn}

Let us consider the following, slightly a bit stronger set of conditions for a vertex $x\in Q_0$.
\begin{cond}\label{Cond_APR_preserve}
\ \ 
\begin{enumerate}
\item[{\rm (c1)}] ($=${\rm (r1)}) There is no arrow with $s(\al)=t(\al)=x$.
\item[{\rm (c2)}] $\din(x)=\dout(x)=2$.
\item[{\rm (c3)}] For any $\be\in Q_1$ with $t(\be)=x$, there exists $\gam$ which satisfies $t(\gam)=s(\be)$ and $\gam\be\in I$.
\end{enumerate}
\end{cond}

In the rest of this section, let $x$ be a vertex satisfying Condition \ref{Cond_APR_preserve}. Such $x$ fulfills the requirements {\rm (r1),(r2)} in Definition \ref{DefAPRreflection}. In addition, we use the following notations.
\begin{dfn}\label{DefNotaAPRatx}
Let $(Q,I)$ and $x$ be as above.
\begin{enumerate}
\item Let $(Q\bls,I\bls)$ be a blossoming of $(Q,I)$ as before. We continue to use the symbols $s_p,t_p\in Q_0$ and $\sig_p,\tau_p\in Q_1$ for $1\le p\le\dfr_A$ as in Definition \ref{DefBlossom}. Remark that $x$ also satisfies Condition \ref{Cond_APR_preserve} as a vertex in $(Q\bls,I\bls)$. Let $(Q\ppr,I\ppr)$ (respectively $(Q^{\blossom\prime},I^{\blossom\prime})$) denote the gentle bound quiver obtained from $(Q,I)$ (resp. $(Q\bls,I\bls)$) by applying generalized APR-reflection at $x$.
\item By {\rm (c2)} and {\rm (c3)}, there are the following arrows.
\begin{itemize}
\item[{\rm (i)}] $\al_1,\al_2\in Q_1$ with $\al_1\ne\al_2$ and $s(\al_1)=s(\al_2)=x$.
\item[{\rm (ii)}] $\be_1,\be_2\in Q_1$ with $\be_1\ne\be_2$ and $t(\be_1)=t(\be_2)=x$, satisfying $\be_1\al_2\in I$ and $\be_2\al_1\in I$.
\item[{\rm (iii)}] $\gam_1,\gam_2\in Q_1$ with $t(\gam_i)=s(\be_i)$ and $\gam_i\be_i\in I$ for $i=1,2$.
\end{itemize}
We put $\Scal=\{\al_1,\al_2,\be_1,\be_2,\gam_1,\gam_2\}$ and
\[ t(\al_i)=a_i,\ s(\be_i)=t(\gam_i)=b_i,\ s(\gam_i)=c_i \]
for $i=1,2$.
\end{enumerate}
\end{dfn}

\begin{rem}\label{RemAPRatx}
The following holds.
\begin{enumerate}
\item By {\rm (c1)}, we have $x\notin\{ a_1,a_2,b_1,b_2\}$. In particular we have $\{\al_1,\al_2\}\cap\{\be_1,\be_2\}=\emptyset$ and $\{\be_1,\be_2\}\cap\{\gam_1,\gam_2\}=\emptyset$. By $\be_1\ne\be_2$ and $\gam_i\be_i\in I\ (i=1,2)$, we also have $\gam_1\ne\gam_2$. On the contrary, $\{\al_1,\al_2\}\cap\{\gam_1,\gam_2\}$ may not be empty.
\item For any $\lam\in Q_1$, we have
\[ s(\lam\ppr)=\begin{cases}
x&\ \text{if}\ \lam\in\{\be_1,\be_2\}\\
b_i&\ \text{if}\ \lam=\al_i\ (i=1,2)\\
s(\lam)&\ \text{if}\ \lam\in Q_1\setminus\{\al_1,\al_2,\be_1,\be_2\}
\end{cases}\ ,\qquad
t(\lam\ppr)=\begin{cases}
b_i&\ \text{if}\ \lam=\be_i\ (i=1,2)\\
x&\ \text{if}\ \lam\in\{\gam_1,\gam_2\}\\
t(\lam)&\ \text{if}\ \lam\in Q_1\setminus\{\be_1,\be_2,\gam_1,\gam_2\}
\end{cases}
\]
in $Q\ppr_1$. Similarly for arrows in $Q^{\blossom\prime}$. Especially, we have the following.
\begin{itemize}
\item $Q\ppr$ is a subquiver of $Q^{\blossom\prime}$.
\item For any $v\in Q\bls_0$, its in-degree in $Q\bls$ is the same as that in $Q^{\blossom\prime}$. Namely, the in-degrees of vertices are unchanged by the generalized APR-reflection at $x$. Similarly for the out-degrees.
\item For any arrow $\lam\in Q_1\setminus\Scal$, we have $s(\lam\ppr)=s(\lam)$ and $t(\lam\ppr)=t(\lam)$. In this sense, arrows outside of $\Scal$ are \lq\lq unchanged" by the generalized APR-reflection at $x$.
\end{itemize}
\item $I\ppr$ is generated by
\begin{eqnarray*}
&&\{\lam\ppr\kap\ppr\mid\lam,\kap\in Q_1,\, \kap\notin\{\al_1,\al_2,\be_1,\be_2\},\, t(\lam)=s(\kap),\, \lam\kap\in I\}\\
&&\!\!\amalg\,\{\be\ppr_i\al\ppr_i\mid i=1,2\}\amalg\{\gam\ppr_1\be\ppr_2,\gam\ppr_2\be\ppr_1 \}.
\end{eqnarray*}
Similarly, $I^{\blossom\prime}$ is generated by
\begin{eqnarray*}
&&\{\lam\ppr\kap\ppr\mid\lam,\kap\in Q\bls_1,\, \kap\notin\{\al_1,\al_2,\be_1,\be_2\},\, t(\lam)=s(\kap),\, \lam\kap\in I\}\\
&&\!\!\amalg\,\{\be\ppr_i\al\ppr_i\mid i=1,2\}\amalg\{\gam\ppr_1\be\ppr_2,\gam\ppr_2\be\ppr_1 \}.
\end{eqnarray*}
Thus if $\lam,\kap\in Q_1$, then we have $\lam\ppr\kap\ppr\in I\ppr$ if and only if $\lam\ppr\kap\ppr\in I^{\blossom\prime}$, 
which means that $(Q\ppr,I\ppr)$ is a bound subquiver of $(Q^{\blossom\prime},I^{\blossom\prime})$.
\item By {\rm (2)}, in particular
\begin{itemize}
\item $s_p\ (1\le p\le d)$ are the source vertices with out-degree $1$,
\item $t_p\ (1\le p\le d)$ are the sink vertices with in-degree $1$,
\item any vertex $v\in Q\ppr_0$ has in-degree $2$ and out-degree $2$ in $Q^{\blossom\prime}_0$
\end{itemize}
in the gentle bound quiver $(Q^{\blossom\prime},I^{\blossom\prime})$. This means that $(Q^{\blossom\prime},I^{\blossom\prime})$ is a blossoming of $(Q\ppr,I\ppr)$, as in Remark \ref{RemGentleUnique}.

\end{enumerate}
\end{rem}


\begin{dfn}\label{Def_mu}
We continue to use the notation in Definition \ref{DefNotaAPRatx}. Put $B=KQ\ppr/I\ppr$, and $B\bls=KQ^{\blossom\prime}/I^{\blossom\prime}$.
\begin{enumerate}
\item Let $\Acal_x$ and $\Acal\bls_x$ be the following sets of paths in $A=KQ/I$ and $A\bls=KQ\bls/I\bls$, respectively. We have $\Acal_x\se\Acal\bls_x$.
\begin{eqnarray*}
&\Acal_x=(Q_1\setminus\Scal)\amalg\{\be_i\al_i\mid i=1,2\}\amalg(\{\gam_1,\gam_2\}\setminus\{\al_1,\al_2\}),&\\
&\Acal\bls_x=(Q\bls_1\setminus\Scal)\amalg\{\be_i\al_i\mid i=1,2\}\amalg(\{\gam_1,\gam_2\}\setminus\{\al_1,\al_2\}).&
\end{eqnarray*}
\item Let $\cfr_x\co\Acal\bls_x\to\{\text{paths in}\ B\bls\}$ be a map defined by the following.
\begin{itemize}
\item[{\rm (i)}] For any arrow $\lam\in Q\bls_1\setminus\Scal$, put $\cfr_x(\lam)=\lam\ppr$.
\item[{\rm (ii)}] For $\be_1\al_1$ and $\be_2\al_2$, define as
\[
\cfr_x(\be_1\al_1)=\begin{cases}
\al_1\ppr&\text{if}\ \al_1\notin\{\gam_1,\gam_2\}\\
\al_1\ppr\be_1\ppr=\gam_1\ppr\be_1\ppr&\text{if}\ \al_1=\gam_1\\
\al_1\ppr\be_2\ppr=\gam_2\ppr\be_2\ppr&\text{if}\ \al_1=\gam_2
\end{cases}
\ ,\quad
\cfr_x(\be_2\al_2)=\begin{cases}
\al_2\ppr&\text{if}\ \al_1\notin\{\gam_1,\gam_2\}\\
\al_2\ppr\be_2\ppr=\gam_2\ppr\be_2\ppr&\text{if}\ \al_2=\gam_2\\
\al_2\ppr\be_1\ppr=\gam_1\ppr\be_1\ppr&\text{if}\ \al_2=\gam_1
\end{cases}\ .
\]
\item[{\rm (iii)}] For $i=1,2$, if $\gam_i\in\Acal\bls_x$ (i.e., if $\gam_i\notin\{\al_1,\al_2\}$), then $\cfr_x(\gam_i)=\gam_i\ppr\be_i\ppr$.
\end{itemize}
\end{enumerate}
It is obvious that $\cfr_x(\xi)$ is indeed a path in $B\bls$ for any $\xi\in\Acal\bls_x$, since $\gam_i\ppr\be_i\ppr$ is always a path for $i=1,2$ by definition. Also, we can check $s(\xi)=s(\cfr_x(\xi))$ and $t(\xi)=t(\cfr_x(\xi))$ easily.
\end{dfn}

\begin{prop}\label{PropDecompAx}
Let $\Pfr_p=\sig_p\wp_p\tau_p$ be any maximal path in $A\bls$, as in Definition \ref{DefBlossom}. Then $\wp_p$ can be written uniquely as a product of paths
\begin{equation}\label{ExpDecompAx}
\wp_p=\xi_1\cdots\xi_r
\end{equation}
with $r\in\Nn$ and pairwise distinct $\xi_1,\ldots,\xi_r\in\Acal_x$. If $\length(\wp_p)=0$, we just put $\wp_p=e_{t(\sig_p)}$.
\end{prop}
\begin{proof}
Put $l=\length(\wp_p)$. We may assume $l>0$. In this case we can uniquely write as $\wp_p=\lam_1\lam_2\cdots\lam_l$ using some arrows $\lam_1,\ldots,\lam_l\in Q_1$. The following holds for $1\le j\le l$.
\begin{itemize}
\item If $\lam_j\in Q_1\setminus\Scal$, then $\lam_j\in\Acal\bls_x$.
\item If $\lam_j=\al_i$ for $i\in\{1,2\}$, then we should have $j>1$ and $\lam_{j-1}\lam_j=\be_i\al_i\in\Acal\bls_x$.
\item Similarly, if $\lam_j\in\{\be_1,\be_2\}$, then we have $j<l$ and $\lam_j\lam_{j+1}\in\Acal\bls_x$.
\item If $\lam_j\notin\{\al_1,\al_2,\be_1,\be_2\}$ and $\lam_j=\gam_i$ for $i\in\{1,2\}$, then $\lam_j\in \{\gam_1,\gam_2\}\setminus\{\al_1,\al_2\}\se\Acal\bls_x$.
\end{itemize}
Thus $\lam_1\cdots\lam_l$ can be written as a product of elements in $\Acal\bls_x$. Uniqueness of such an expression follows from Remark \ref{RemAPRatx} {\rm (1)} and the definition of $\Acal\bls_x$. Since $A=KQ/I$ has no oriented cycle because it is gentle by assumption, those $\xi_1,\ldots,\xi_r$ are necessarily pairwise distinct.
\end{proof}

\begin{lem}\label{LemDecompAx1}
Let $\lam\in Q\bls_1$ be any arrow.
\begin{enumerate}
\item Assume $s(\lam\ppr)=b_1$. Then $\be_1\ppr\lam\ppr$ is a path in $B\bls$ if and only if $\lam\ne\al_1$.
\item Assume $t(\lam\ppr)=b_2$. Then $\lam\ppr\al_2\ppr$ is a path in $B\bls$ if and only if $\lam\ne\be_2$.
\item If $\lam\gam_2$ is a path in $A\bls$ and if $\gam_2\notin\{\al_1,\al_2\}$ and $\lam\ne\gam_1$, then $\lam\ppr\gam_2\ppr$ is a path in $B\bls$.
\item If $\gam_1\gam_2$ is a path in $A$ and if $\gam_2\notin\{\al_1,\al_2\}$, then $\be_1\ppr\gam_2\ppr$ is a path in $B$.
\end{enumerate}
\end{lem}
\begin{proof}
This is straightforward from Remark \ref{RemAPRatx} {\rm (3)} and the gentleness.
\end{proof}

\begin{lem}\label{LemDecompAx2}
Let $\lam\in Q\bls_1$ be any arrow.
\begin{enumerate}
\item If $\gam_1\lam$ is a path in $A\bls$ and if $\lam\ne\be_2$, then $\be_1\ppr\lam\ppr$ is a path in $B\bls$.
\item If $\lam\be_2$ is a path in $A\bls$ and if $\lam\ne\gam_1$, then $\lam\ppr\al_2\ppr$ is a path in $B\bls$.
\item If $\al_1\lam$ is a path in $A\bls$ and if $\lam\notin\{\be_1,\be_2\}$, then $\cfr_x(\be_1\al_1)\lam\ppr$ is a path in $B\bls$.
\end{enumerate}
\end{lem}
\begin{proof}
{\rm (1)} Since $\gam_1\lam$ is a path in $A\bls$, in particular we have $\lam\ne\be_1$. Also, $s(\lam)=t(\gam_1)=b_1\ne x$ implies $\lam\notin\{\al_1,\al_2\}$. 
Thus we have $\lam\notin\{\al_1,\al_2,\be_1,\be_2\}$, which implies $s(\lam\ppr)=s(\lam)=b_1$. Lemma \ref{LemDecompAx1} {\rm (1)} shows that $\be_1\ppr\lam\ppr$ is a path in $B\bls$.

{\rm (2)} Since $\lam\be_2$ is a path in $A\bls$, in particular we have $\lam\ne \gam_2$. Also, $t(\lam)=s(\be_2)=b_2\ne x$ implies $\lam\notin\{\be_1,\be_2\}$.Thus we have $\lam\notin\{\be_1,\be_2,\gam_1,\gam_2\}$, which implies $t(\lam\ppr)=t(\lam)=b_2$. Lemma \ref{LemDecompAx1} {\rm (2)} shows that $\lam\ppr\al_2\ppr$ is a path in $B\bls$.

{\rm (3)} Since $\al_1\lam$ is a path, in particular we have $s(\lam)=t(\al_1)=a_1\ne x$. Especially we have $\lam\notin\{\al_1,\al_2\}$.
Thus we have $s(\lam\ppr)=s(\lam)=a_1=t(\be_1\al_1)=t(\cfr_x(\be_1\al_1))$. 
By Remark \ref{RemAPRatx} {\rm (3)} and by $\lam\notin\{\al_1,\al_2,\be_1,\be_2 \}$, we see that if $\al_1=\gam_1$, then $t(\be_1)\ne s(\lam)$ implies that $\cfr_x(\be_1\al_1)\gam\ppr_1=\al_1\ppr\be\ppr_1\gam\ppr_1$ is a path in $B\bls$.
Similarly if $\al_1=\gam_2$, then $t(\be_2)\ne s(\lam)$ implies that $\cfr_x(\be_1\al_1)\gam\ppr_1=\al_1\ppr\be\ppr_2\gam\ppr_1$ is a path in $B\bls$.
If $\al\notin\{\gam_1,\gam_2\}$, then $\al_1\lam\notin I\bls$ implies that $\cfr_x(\be_1\al_1)\gam\ppr_1=\al_1\ppr\gam_1\ppr$ is a path in $B\bls$.
\end{proof}

\begin{prop}\label{PropMutAx}
Let $\xi,\ze\in\Acal\bls_x$ be any pair of elements. If $\xi\ze$ is a path in $A\bls$, then $\cfr_x(\xi)\cfr_x(\ze)$ is a path in $B\bls$.
\end{prop}
\begin{proof}
Since $(Q\bls,I\bls)$ has no oriented cycle, we have $\xi\ne\ze$. We remark that each $\cfr_x(\xi)$ and $\cfr_x(\ze)$ is a path in $B\bls$. Accordingly to the 3 components $Q\bls_1\setminus\Scal$, $\{\be_i\al_i\mid i=1,2\}$ and $\{\gam_1,\gam_2\}\setminus\{\al_1,\al_2\}$ of $\Acal\bls_x$, the arguments below in the following 9 cases show the proposition.
\begin{enumerate}
\item If $\xi,\ze\in Q\bls_1\setminus\Scal$, this is obvious by Remark \ref{RemAPRatx} {\rm (3)}.
\item If $\xi,\ze\in\{\gam_1,\gam_2\}\setminus\{\al_1,\al_2\}$, we may assume $\xi=\gam_1$ and $\ze=\gam_2$ without loss of generality. Then $\be\ppr_1\gam\ppr_2$ is a path in $B\bls$ by Lemma \ref{LemDecompAx2} {\rm (1)}, hence so is $\cfr_x(\xi)\cfr_x(\ze)=\gam_1\ppr\be_1\ppr\gam_2\ppr\be_2\ppr$.
\item If $\xi\in\{\gam_1,\gam_2\}\setminus\{\al_1,\al_2\}$ and $\ze\in Q\bls_1\setminus\Scal$, we may assume $\xi=\gam_1$. Then $\cfr_x(\xi)\cfr_x(\ze)=\gam_1\ppr\be_1\ppr\ze\ppr$ is a path in $B\bls$ by Lemma \ref{LemDecompAx2} {\rm (1)}, similarly as above.
\item If $\xi\in Q\bls_1\setminus\Scal$ and $\ze\in\{\be_i\al_i\mid i=1,2\}$, we may assume $\ze=\be_2\al_2$. Then $\xi\ppr\al_2\ppr=\cfr_x(\xi)\al_2\ppr$ is a path in $B\bls$ by Lemma \ref{LemDecompAx2} {\rm (2)}. Since $\cfr_x(\ze)$ always starts with $\al_2\ppr$, this means that $\cfr_x(\xi)\cfr_x(\ze)$ is a path in $B\bls$.
\item If $\xi,\ze\in\{\be_i\al_i\mid i=1,2\}$, we may assume $\xi=\be_1\al_1$ and $\ze=\be_2\al_2$. Similarly as above, it suffices to show that $\cfr_x(\xi)\al_2\ppr$ is a path in $B\bls$.

Since $\al_1\be_2$ is a path, we have $a_1=b_2$ and $\al_1\ne\gam_2$. If $\al_1=\gam_1$, then $t(\be_1\ppr)=b_1=a_1=b_2$ implies that $\be_1\ppr\al_2\ppr$ is a path in $B\bls$ by Lemma \ref{LemDecompAx1} {\rm (2)}. Thus $\cfr_x(\be_1\al_1)\al_2\ppr=\al_1\ppr\be_1\ppr\al_2\ppr$ is a path in $B\bls$. If $\al_1\ne\gam_1$, then $\cfr_x(\be_1\al_1)=\al_1\ppr$. Thus $\cfr_x(\be_1\al_1)\al_2\ppr=\al_1\ppr\al_2\ppr$ is a path in $B\bls$ by Lemma \ref{LemDecompAx2} {\rm (2)}. 
\item If $\xi\in\{\gam_1,\gam_2\}\setminus\{\al_1,\al_2\}$ and $\ze\in\{\be_1\al_1,\be_2\al_2\}$, we may assume $\xi=\gam_1$. Since $\gam_1\be_1\in I$, necessarily we have $\ze=\be_2\al_2$. Then $t(\xi)=s(\ze)$ means $b_1=b_2$, hence $t(\be_1\ppr)=b_2$. Thus $\be_1\ppr\al_2\ppr$ is a path in $B\bls$ by Lemma \ref{LemDecompAx1} {\rm (2)}, hence so is $\cfr_x(\xi)\cfr_x(\ze)$.
\item If $\xi\in Q\bls_1\setminus\Scal$ and $\ze\in\{\gam_1,\gam_2\}\setminus\{\al_1,\al_2\}$, we may assume $\ze=\gam_2\notin\{\al_1,\al_2\}$. If $\xi\ne\gam_1$, then $\xi\ppr\gam_2\ppr$ is a path in $B\bls$ by Lemma \ref{LemDecompAx1} {\rm (3)}, hence so is $\cfr_x(\xi)\cfr_x(\ze)=\xi\ppr\gam_2\ppr\be_2\ppr$. If $\xi=\gam_1$, then $\cfr_x(\xi)\cfr_x(\ze)=\gam_1\ppr\be_1\ppr\gam_2\ppr\be_2\ppr$ is a path in $B$ by Lemma \ref{LemDecompAx1} {\rm (4)}.
\item If $\xi\in\{\be_i\al_i\mid i=1,2\}$ and $\ze\in Q\bls_1\setminus\Scal$, we may assume $\xi=\be_1\al_1$. Then $\cfr_x(\xi)\cfr_x(\ze)=\cfr_x(\xi)\ze\ppr$ is a path in $B\bls$ by Lemma \ref{LemDecompAx2} {\rm (3)}.
\item If $\xi\in\{\be_i\al_i\mid i=1,2\}$ and $\ze\in\{\gam_1,\gam_2\}\setminus\{\al_1,\al_2\}$, we may assume $\xi=\be_1\al_1$. Then $\cfr_x(\xi)\ze\ppr$ is a path in $B\bls$ by Lemma \ref{LemDecompAx2} {\rm (3)}. Since $\cfr_x(\ze)$ starts with $\ze\ppr$, this implies that $\cfr_x(\xi)\cfr_x(\ze)$ is a path in $B\bls$.
\end{enumerate}
\end{proof}

\begin{cor}\label{CorDecompAx}
Let $\sig_p\wp_p\tau_p$ be any maximal path in $A\bls$, as in Definition \ref{DefBlossom}. Take the expression
\[ \sig_p\wp_p\tau_p=\sig_p\xi_1\cdots\xi_r\tau_p \]
with $r\in\Nn$ and $\xi_1,\ldots,\xi_r\in\Acal\bls_x$, as in Proposition \ref{PropDecompAx}. Then 
\begin{equation}\label{MaxinB}
\sig_p\ppr\cfr_x(\xi_1)\cdots\cfr_x(\xi_r)\tau_p\ppr
\end{equation}
gives a maximal path in $B\bls$.
\end{cor}
\begin{proof}
Since $\cfr_x(\sig_p)=\sig_p\ppr$ and $\cfr_x(\tau_p)=\tau_p\ppr$, Proposition \ref{PropMutAx} shows that $(\ref{MaxinB})$ is indeed a path in $B\bls$. Since $s_p=s(\sig_p\ppr)$ is a source and $t_p=t(\tau_p\ppr)$ is a sink in $B\bls$, it is maximal.
\end{proof}

\begin{prop}\label{PropAPRPresRepet}
Let $A=KQ/I$, $x\in Q_0$ and $B=KQ\ppr/I\ppr$ be as above. We continue to use the above notations. Especially, we assume that blossom vertices in $A\bls$ and $B\bls$ have the same labeling $\{s_p\mid 1\le p\le \dfr_A=\dfr_B\}$ and $\{t_p\mid 1\le p\le \dfr_A=\dfr_B\}$, as shown in Corollary \ref{CorDecompAx}.

Let $C=KQ\pprr/I\pprr$ be another arbitrary gentle algebra satisfying $\dfr_C=\dfr_A=d$. For any permutation $\sfr\in\Sfr_d$, we have the following.
\begin{enumerate}
\item $B\un{\sfr}{\ci}C$ is obtained from $A\un{\sfr}{\ci}C$ by the generalized APR-reflection at $x\in Q_0\se(Q\un{\sfr}{\ci}Q\pprr)_0$.
\item $C\un{\sfr}{\ci}B$ is obtained from $C\un{\sfr}{\ci}A$ by the generalized APR-reflection at $x\in Q_0\se(Q\pprr\un{\sfr}{\ci}Q)_0$.
\end{enumerate}
\end{prop}
\begin{proof}
Since {\rm (2)} can be shown in a similar way, we only show {\rm (1)}. Since $\Scal=\{\al_1,\al_2,\be_1,\be_2,\gam_1,\gam_2\}$ is contained in $\Scal\se Q_1\se(Q\un{\sfr}{\ci}Q\pprr)_1$, it is obvious that $x$ also satisfies Condition \ref{Cond_APR_preserve} as a vertex of $Q\un{\sfr}{\ci}Q\pprr$. Since the arrows and relations are unchanged outside of $\Scal$, it is obvious that $(Q\ppr\un{\sfr}{\ci}Q\pprr,I\ppr\un{\sfr}{\ci}I\pprr)$ is obtained from $(Q\un{\sfr}{\ci}Q\pprr,I\un{\sfr}{\ci}I\pprr)$ by the generalized APR-reflection at $x\in(Q\un{\sfr}{\ci}Q\pprr)_0$.
\end{proof}

\begin{cor}\label{CorAPRRepet}
Let $A=KQ/I$, $x\in Q_0$ and $B=KQ\ppr/I\ppr$ be as before. Let $k$ be any positive integer. Then $(Q^{\prime(k)},I^{\prime(k)})$ is obtained from $(Q^{(k)},I^{(k)})$ by a $k$-times iteration of generalized APR-reflections.
\end{cor}
\begin{proof}
By Proposition \ref{PropAPRPresRepet} {\rm (2)}, we see that $(Q^{(k-1)},I^{(k-1)})\un{\id}{\ci}(Q\ppr,I\ppr)$ is obtained from $(Q^{(k)},I^{(k)})=(Q^{(k-1)},I^{(k-1)})\un{\id}{\ci}(Q,I)$ by the generalized APR-reflection at $x^{[k]}\in Q^{(k)}_0$. Iteratively, for any $2\le j\le k$, Proposition \ref{PropAPRPresRepet} shows that
\[ (Q^{(k-j)},I^{(k-j)})\un{\id}{\ci}(Q^{\prime(j)},I^{\prime(j)})\qquad ((Q^{\prime(k)},I^{\prime(k)})\ \text{when}\ j=k) \]
is obtained from the generalized APR-reflection at $x^{[k-j+1]}\in Q^{(k-j+1)}_0\se(Q^{(k-j+1)}\un{\id}{\ci}Q^{\prime(j-1)})_0$.
\end{proof}

\section{Some remarks for graded case}\label{section_Graded}

In this section, assume that $A=KQ/I$ is a locally gentle algebra, equipped with a function $\deg\co Q_1\to\Zbb$.
\begin{dfn}\label{DefAGrading}
By using the function $\deg\co Q_1\to\Zbb$, we can endow $A$ a $\Zbb$-grading, by assigning
\[ \deg\thh=\begin{cases}\displaystyle\sum_{1\le u\le l}\deg\al_u&\ \text{if}\ \thh=\al_1\cdots\al_{l}\\[12pt]
\ 0&\ \text{if}\ \thh=e_a\ \text{for some}\ a\in Q_0 \end{cases} \]
to each path $\thh\in A$, and by extending it $K$-linearly. Let $A_n$ denote the homogeneous part of degree $n$. Then $A=\un{n\in\Zbb}{\oplus}A_n$ becomes a $\Zbb$-graded $K$-algebra.
\end{dfn}

Graded version of AG-invariants for such $A$ have been introduced by Lekili and Polishchuk in \cite{LP}. In our terminology, it can be stated as follows (Definition \ref{DefAGInvGraded}). For the underlying ungraded locally gentle algebra, we continue to use the same notations as in Section \ref{section_Review}. 
\begin{itemize}
\item $d=\dfr_A$, and $\Phi=\Phi_A\in\Sfr_d$,
\item $\Mfr_{A\bls}=\{\Pfr_p(A)=\sig_p\wp_p(A)\tau_p\mid1\le p\le d \}$ and $\wp_p=\wp_p(A)$,
\item $\Afr_{A\bls}=\{\Del_p(A)=\sig_{\Phi(p)}\del_p(A)\tau_p\mid1\le p\le d \}$ and $\del_p=\del_p(A)$.
\end{itemize}
Remark that each AG-orbit of $A$ is of the form
\[ \Ocal=\{\del_{\Phi^j(p)}\mid0\le j<q\} \]
for some $1\le p\le d$, where $q=|\Ocal|\in\Np$ is the smallest positive integer satisfying $\del_p=\del_{\Phi^q(p)}$.

\begin{dfn}\label{DefDegBar}
Let $l>0$ be a positive integer. For any antipath $\om=\al_1\cdots\al_l$ in $A$, put
\[ \ovl{\deg}(\om)=l-\sum_{1\le i\le l}\deg(\al_i). \]
This also applies when $\om$ is an anticycle.
We put $\ovl{\deg}(e_a)=0$ for any $a\in Q_0$.
\end{dfn}

\begin{rem}\label{RemDegBar1}
The following holds.
\begin{enumerate}
\item If $\om=\al_1\cdots\al_l$ and $\ze=\be_1\cdots\be_m$ are anticycles in $A$ such that $\om\ze=\al_1\cdots\al_l\be_1\cdots\be_m$ is again an anticycle in $A$, then we have
\[ \ovl{\deg}(\om\ze)=\ovl{\deg}(\om)+\ovl{\deg}(\ze). \]
\item For any arrow $\al\in Q_1$, we have
$\ovl{\deg}(\al)=1-\deg(\al)$
by definition.
\end{enumerate}
\end{rem}

\begin{rem}\label{RemDegBar2}
If $\deg_0\co Q_1\to\Zbb$ is the constant function to $0\in\Zbb$, then we have
$\ovl{\deg_0}(\om)=l$
for any antipath $\om=\al_1\cdots\al_l$ in $A$.
\end{rem}

\begin{dfn}\label{DefAGInvGraded}
Let $A=KQ/I$ and $\deg\co Q_1\to\Zbb$ as before.
\begin{enumerate}
\item For each AG-orbit $\Ocal=\{\del_{\Phi^j(p)}(A)\mid 0\le j<q\}$ with $q=|\Ocal|$, considering a sequence
\[ \wp_p(A),\del_p(A),\wp_{\Phi(p)}(A),\del_{\Phi(p)}(A),\ldots,\wp_{\Phi^{q-1}(p)}(A),\del_{\Phi^{q-1}(p)}(A), \]
put $\type_{\deg}(\Ocal)=(q,l)\in\Np\ti\Nn$, where
\[ l=\sum_{0\le j<q-1}\big(\deg(\wp_{\Phi^j(p)}(A))+\ovl{\deg}(\del_{\Phi^j(p)}(A)\big). \]
\item For any $(q,l)\in\Nn\ti\Nn$, define $\phi_{A,\deg}(q,l)\in\Nn$ by the following.
\[
\phi_{A,\deg}(q,l)=\begin{cases}
\begin{array}{l}
|\{\rho\mid \rho\ \text{is an oriented cycle in}\ A\ \text{with}\ \deg(\rho)=l \}|\\
+|\{\om\mid \om\ \text{is an anticycle in}\ A\ \text{with}\ \ovl{\deg}(\om)=l \}|
\end{array}&\ \text{if}\ q=0,\\[16pt]
|\{\Ocal\mid \Ocal\ \text{is an AG-orbit in}\ A\ \text{with}\ \type_{\deg}(\Ocal)=(q,l)\}|&\ \text{if}\ q>0.
\end{cases}
\]
\end{enumerate}
\end{dfn}

\begin{rem}\label{RemPhiBar}
Suppose $A=KQ/I$ is gentle.
Let $\deg_0$ be the constant function to $0\in\Zbb$, as in Remark \ref{RemDegBar2}. We have $\type_{\deg_0}(\Ocal)=\type(\Ocal)$ for any AG-orbit $\Ocal$ in $A$, and $\phi_{A,\deg_0}(q,l)=\phi_A(q,l)$ for any $(q,l)\in\Nn\ti\Nn$. 
\end{rem}

\begin{dfn}\label{DefDegRepet}
As in Definition \ref{DefQkIk}, let $(Q,I)$ be a connected locally gentle bound quiver 
which satisfies $\Mfr_A\amalg\Tfr_A\ne\emptyset$. Let $k$ be any positive integer. Remark that the $k$-times gentle repetition $(Q^{(k)},I^{(k)})$ of $(Q,I)$ defined in Definition \ref{DefQkIk}, has the following set of arrows.
\[ Q_1^{(k)}=\{\al^{[i]}\mid \al\in Q_1,\ 1\le i\le k\}\amalg\{\varpi_p^{[i]}\mid 1\le p\le d,\ 1\le i\le k-1 \}. \]
We extend the function $\deg\co Q_1\to\Zbb$ to $Q^{(k)}_1$, which we denote by the same symbol $\deg\co Q_1^{(k)}\to\Zbb$, by the following.
\begin{itemize}
\item For any $\al\in Q_1$, put $\deg(\al^{[i]})=\deg(\al)$ for any $1\le i\le k$.
\item For each $1\le p\le d$, choose some $w_p\in\Zbb$, and put $\deg(\varpi_p^{[i]})=w_p$ for any $1\le i\le k-1$.
\end{itemize}
\end{dfn}

\begin{rem}
Although there is a natural choice of $w_p\in\Zbb$ as below,
\[ w_p=\begin{cases}\displaystyle-\sum_{1\le u\le l}\deg(\al_u)&\ \text{if}\ \varpi_p^{[i]}=\rho^{\ast [i]}\ \text{for}\ \rho=\al_1\cdots\al_l\in\Mfr_A\\
0&\ \text{otherwise}
\end{cases} \]
the proofs of Propositions \ref{PropAGInvGraded1} and \ref{PropAGGraded2} work for arbitrary $w_p$.
\end{rem}

\begin{prop}\label{PropAGInvGraded1}
Let $(Q,I)$, $k\in\Np$, and $\deg\co Q_1^{(k)}\to\Zbb$ be as in Definition \ref{DefDegRepet}. Then we have
\[
\phi_{A^{(k)},\deg}(0,l)=k\phi_{A,\deg}(0,l)
\]
for any $l\in\Nn$.
\end{prop}
\begin{proof}
This is obvious by Lemma \ref{LemBasisSheets}.
\end{proof}

\begin{prop}\label{PropAGGraded2}
Let $(Q,I)$, $k\in\Np$, and $\deg\co Q_1^{(k)}\to\Zbb$ be as in Definition \ref{DefDegRepet}. Assume that $k$ is coprime to $|\Ocal|$ for any AG-orbit $\Ocal$ in $A$. Then for any $(q,l)\in\Np\ti\Nn$, we have
\begin{equation}\label{***}
\phi_{A^{(k)},\deg}(q,l)=\begin{cases}
\phi_{A,\deg}(q,\frac{q+l}{k}-q)&\ \text{if}\ \frac{q+l}{k}-q\in\Nn,\\
0&\ \text{otherwise}.
\end{cases}
\end{equation}
\end{prop}
\begin{proof}
Put $\Psi=\Phi_{A^{(k)}}$. Remark that we have $\Psi=\Phi^k$ for $\Phi=\Phi_A$. Similarly as in the proof of Lemma \ref{LemAGRepet}, any AG-orbit $\Ocal=\{\del_{\Phi^j(p)}\mid0\le j<q\}$ of $\type_{\deg}(\Ocal)=(q,n)$ in $A$ yields one AG-orbit
\[ \widetilde{\Ocal}=\{\del_{\Psi^j(p)}(A^{(k)})\mid 0\le j<q\} \]
in $A^{(k)}$, since $k$ and $q=|\Ocal|$ are coprime. For any $0\le j<q$, we have
\begin{eqnarray*}
\wp_{\Psi^j(p)}(A^{(k)})&=&(\wp_{\Psi^j(p)}(A))^{[1]}(\varpi_{\Psi^j(p)})^{[1]}(\wp_{\Psi^j(p)}(A))^{[2]}\cdots(\varpi_{\Psi^j(p)})^{[k-1]}(\wp_{\Psi^j(p)}(A))^{[k]},\\
\del_{\Psi^j(p)}(A^{(k)})&=&(\del_{\Phi^{k-1}\Psi^j(p)}(A))^{[1]}(\varpi_{\Phi^{k-1}\Psi^j(p)})^{[1]}(\del_{\Phi^{k-2}\Psi^j(p)}(A))^{[2]}\cdots(\varpi_{\Phi\Psi^j(p)})^{[k-1]}(\del_{\Psi^j(p)}(A))^{[k]}
\end{eqnarray*}
as in $(\ref{EqP})$ and $(\ref{EqA})$.
They satisfy
\begin{eqnarray}
\deg(\wp_{\Psi^j(p)}(A^{(k)}))&=&\sum_{1\le i\le k}\deg((\wp_{\Psi^j(p)}(A))^{[i]})+\sum_{1\le i\le k-1}\deg((\varpi_{\Psi^j(p)})^{[i]})\nonumber\\
&=&k\deg(\wp_{\Psi^j(p)}(A))+(k-1)w_{\Psi^j(p)}\label{EqdegP}
\end{eqnarray}
and
\begin{eqnarray}
\ovl{\deg}(\del_{\Psi^j(p)}(A^{(k)}))&=&\sum_{1\le i\le k}\ovl{\deg}((\del_{\Phi^{k-i}\Psi^j(p)}(A))^{[i]})+\sum_{1\le i\le k-1}\ovl{\deg}((\varpi_{\Phi^{k-i}\Psi^j(p)})^{[i]})\nonumber\\
&=&(k-1)+\sum_{1\le i\le k}\ovl{\deg}(\del_{\Phi^{k-i}\Psi^j(p)}(A))-\sum_{1\le i\le k-1}w_{\Phi^{k-i}\Psi^j(p)}\label{EqdegA}
\end{eqnarray}
by Remark \ref{RemDegBar1}. By definition, we have
$\type_{\deg}(\widetilde{\Ocal})=(q,l)$, with
\begin{equation}\label{EqLengthOtilde}
l=\sum_{0\le j\le q-1}\Big(\deg\big(\wp_{\Psi^j(p)}(A^{(k)})\big)+\ovl{\deg}\big(\del_{\Psi^j(p)}(A^{(k)})\big)\Big).
\end{equation}
Let us calculate $l$. By $(\ref{EqdegP})$, we have
\begin{eqnarray*}
\sum_{0\le j\le q-1}\deg(\wp_{\Psi^j(p)}(A^{(k)}))
&=&\sum_{0\le j\le q-1}\big(k\deg(\wp_{\Phi^{kj}(p)}(A)+(k-1)w_{\Psi^{kj}(p)}\big)\\
&=&k\sum_{0\le x\le q-1}\deg(\wp_{\Phi^x(p)(A)})+(k-1)\sum_{0\le x\le q-1}w_{\Phi^x(p)}
\end{eqnarray*}
since $k$ and $q$ are coprime. Also, by $(\ref{EqdegA})$ and $\Phi^{q+m}(p)=\Phi^m(p)\ (\fa m\in\Nn)$, we have
\begin{eqnarray*}
&&\hspace{-1cm}\sum_{0\le j\le q-1}\ovl{\deg}(\del_{\Psi^j(p)}(A^{(k)}))\\
&=&\sum_{0\le j\le q-1}(k-1)+\sum_{0\le j\le q-1}\sum_{1\le i\le k}\ovl{\deg}(\del_{\Phi^{kj+k-i}(p)}(A))
-\sum_{0\le j\le q-1}\sum_{1\le i\le k-1}w_{\Phi^{kj+k-i}(p)}\\
&=&q(k-1)+\sum_{0\le y\le qk-1}\ovl{\deg}(\del_{\Phi^y(p)}(A))
-\Big(\sum_{0\le y\le qk-1}w_{\Phi^y(p)}-\sum_{0\le j\le q-1}w_{\Phi^{kj}(p)}\Big)\\
&=&q(k-1)+k\sum_{0\le y\le q-1}\ovl{\deg}(\del_{\Phi^y(p)}(A))
-(k-1)\sum_{0\le y\le q-1}w_{\Phi^y(p)}.
\end{eqnarray*}
Thus by $(\ref{EqLengthOtilde})$, we have
\begin{equation}\label{DegTypeEq}
l=q(k-1)+k\Big(\sum_{0\le x\le q-1}\deg\wp_{\Phi^x(p)}(A)+\sum_{0\le y\le q-1}\ovl{\deg}(\del_{\Phi^y(p)}(A)\Big)=q(k-1)+kn.
\end{equation}
Similarly as in the proof of Lemma \ref{LemAGRepet} and Proposition \ref{PropAGRepet1}, from $(\ref{DegTypeEq})$ we obtain
\[ \phi_{A^{(k)},\deg}(q,l)=\sum_{(q\ppr,n)\in S_{(k)}(q,l)}\phi_{A,\deg}(q\ppr,n) \]
where
\begin{eqnarray*}
S_{(k)}(q,l)&=&\{(q\ppr,n)\in\Np\ti\Nn\mid (q,l)=(q\ppr,q\ppr(k-1)+kn) \}\\
&=&\begin{cases}
\{(q,\frac{l-q\ppr(k-1)}{k})\}&\ \text{if}\ \frac{l-q\ppr(k-1)}{k}\in\Nn,\\
\emptyset&\ \text{otherwise}.
\end{cases}
\end{eqnarray*}
This shows $(\ref{***})$.
\end{proof}

\section{Associated semisimple algebra}\label{section_Associated}


In the proceeding sections, for any pair of paths $\thh,\xi$ in $A$, we write $\thh\xi=0$ as an equation in $A$, which means $\thh\xi\in I$, including the case $t(\thh)\ne s(\xi)$. Similarly for $\thh\xi\ne0$.

\begin{dfn}\label{DefSA}
Let $A$ be a locally gentle algebra.
\begin{enumerate}
\item For any maximal path $\rho=\al_1\cdots\al_{l}\in A$ with
\[ s(\al_u)=a_{u-1},\ \ t(\al_u)=a_u\quad(1\le u\le l), \]
let $\Mcal_{\rho}=M_{l+1}(K)$ be the full $(l+1)\times(l+1)$-matrix algebra, whose $(u,v)$-matrix unit is denoted by $\rho_{u,v}$. Thus $\Mcal_{\rho}$ has a $K$-basis $\{\rho_{u,v}\mid 0\le u,v\le l \}$, which we may abbreviately write as follows.
\begin{equation}\label{Mrho1}
\left[\begin{array}{cccc}
K\rho_{0,0}&K\rho_{0,1}&\cdots&K\rho_{0,l}\\
K\rho_{1,0}&K\rho_{1,1}&\cdots&K\rho_{1,l}\\
\vdots&\ddots&\ddots&\vdots\\
K\rho_{l,0}&\cdots&\cdots&K\rho_{l,l}\\
\end{array}\right]
\end{equation}
Remark that a vertex $a\in Q_0$ can appear at most twice in $\{a_0,\ldots,a_{l}\}$.
Define a $K$-algebra homomorphism $\eta_A^{\rho}\co A\to\Mcal_{\rho}$
by $K$-linearly extending the following correspondence for any path $\thh\in A$.
\begin{itemize}
\item[{\rm (i)}] If $\thh=e_a$ for some $a\in Q_0$, then
\[ \eta_A^{\rho}(e_a)=\displaystyle\sum_{\bsm a_u=e_a\\ 0\le u\le l\esm}\rho_{u,u}=%
\begin{cases}
\ 0&\ \text{if}\ a\notin\{a_0,\ldots,a_{l} \}\\
\rho_{u,u}&\ \text{if}\ a=a_u,\ \text{and}\ a\ne a_v\ \text{whenever}\ v\ne u\\
\rho_{u,u}+\rho_{v,v}&\ \text{if}\ a=a_u=a_v\ \text{for}\ u\ne v
\end{cases}. \]
\item[{\rm (ii)}] If $\thh$ is of positive length, then
\[ \eta_A^{\rho}(\thh)=\begin{cases}\rho_{u-1, v}&\ \text{if}\ \thh=\al_u\cdots\al_v\ \text{for some}\ 1\le u\le v\le l\\
\ 0&\ \text{otherwise} \end{cases}. \]
\end{itemize}
\item For any $a\in\Tfr_A$, let $K\efr_a$ be the $1$-dimensional $K$-algebra with the idempotent basis element $\efr_a$. Define a $K$-algebra homomorphism $\eta_A^a\co A\to K\efr_a$ by
\[ \eta_A^a(\thh)=\begin{cases}\efr_a&\ \text{if}\ \thh=e_a\\ \ 0&\ \text{otherwise}\end{cases}. \]
\item Define a semisimple $K$-algebra $V(A)$ by
\[ V(A)=\big(\prod_{\rho\in\Mfr_A}\Mcal_{\rho}\big)\times\big(\prod_{a\in\Tfr_A}K\efr_a\big). \]
By the universality of the product, we obtain a $K$-algebra homomorphism $\eta_A\co A\to V(A)$, whose components are $\eta_A^{\rho}$ and $\eta_A^a$ defined in {\rm (1)} and {\rm (2)}.
\end{enumerate}
\end{dfn}

\begin{rem}\label{RemBasisVA}

\begin{enumerate}
\item Remark that each component $\Mcal_{\rho}$ or $K\efr_a$ has an inclusion to $V(A)$
\[ \Mcal_{\rho}\hookrightarrow V(A),\ \ K\efr_a\hookrightarrow V(A) \]
as a $K$-module. We identify elements in $\Mcal_{\rho}$ or in $K\efr_a$ with their images in $V(A)$ by these inclusions. Namely, an element $\xi\in\Mcal_{\rho}$ is identified with an element $\upsilon=(\{\upsilon_{\pi}\}_{\pi\in\Mcal_{\rho}},\{\upsilon_a\}_{a\in\Tfr_A})\in V(A)$ such that
\[ \upsilon_{\pi}=\begin{cases}\xi&\ \text{if}\ \pi=\rho\\
0&\ \text{if}\ \pi\ne\rho
\end{cases}
\ \ ,\quad \upsilon_a=0\ \ (\fa a\in\Tfr_A).
\]
Similarly for elements in $K\efr_a$.
\item With the above identification, a $K$-basis of $V(A)$ is given by the following subset of $V(A)$.
\begin{equation}\label{BasisForVA}
\big(\coprod_{\rho\in\Mfr_A}\{\rho_{u,v}\mid 0\le u,v\le\length(\rho) \}\big)\amalg\{\efr_a\mid a\in\Tfr_A\}.
\end{equation}
\end{enumerate}
\end{rem}

\begin{rem}\label{RemGrading}
If a function $\deg\co Q_1\to\Zbb$ is given, we may endow $A$ a $\Zbb$-grading as seen in Definition \ref{DefAGrading}. Also for $V(A)$, we have the following.
\begin{enumerate}
\item For any $\rho=\al_1\cdots\al_{l}\in \Mfr_A$, we may endow $\Mcal_{\rho}$ a $\Zbb$-grading by assigning
\[ \deg(\rho_{u,v})=\begin{cases}
\displaystyle\sum_{u-1\le w\le v}\deg\al_w&\ \text{if}\ u<v\\
\ 0&\ \text{if}\ u=v\\
\displaystyle\sum_{v-1\le w\le u}(-\deg\al_w)&\ \text{if}\ u>v
 \end{cases} \]
to each basis element $\rho_{u,v}$, and extending it $K$-linearly. Then $\eta_A^{\rho}\co A\to\Mcal_{\rho}$ preserves the grading.

\item For any $a\in\Tfr_A$ we assign $\deg\efr_a=0$, to regard $K\efr_a$ as a graded $K$-algebra concentrated in degree $0$. Then $\eta_A^a\co A\to K\efr_a$ preserves the grading.

\item $V(A)$ becomes a $\Zbb$-graded $K$-algebra $V(A)=\un{d\in\Zbb}{\bigoplus}V_d(A)$ by the gradings given in {\rm (1)} and {\rm (2)}. Then $\eta_A\co A\to V(A)$ preserves the grading.
\end{enumerate}

We mainly deal with $A$ and $V(A)$ without grading. However, also in this ungraded case, considering a temporary grading given by the constant function $\deg_1\co Q_1\to\Zbb$ with value $1\in\Zbb$ makes some arguments easier. For any path $0\ne\xi\in A$, we have $\deg_1\xi=\length(\xi)$ with this grading.
\end{rem}

\begin{prop}\label{PropApK}
Let $A=KQ/I$ be a locally gentle algebra satisfying $\dfr_A>0$. For any $k\in\Np$, let $\Akp$ to be the following $k\ti k$-upper-triangular matrix algebra.
\[ \Akp=\left[
\begin{array}{ccccc}
A&V(A)&V(A)&\cdots&V(A)\\
0&A&V(A)&\cdots&V(A)\\
0&0&A&\ddots&\vdots\\
0&\cdots&\ddots&\ddots&V(A)\\
0&\cdots&\cdots&0&A\\
\end{array}
\right] \]
Namely, $\Akp=[\Akp_{ij}]_{ij}$ is given by
\[ \Akp_{ij}=
\begin{cases}
A& \text{if}\ i=j,\\
V(A)& \text{if}\ i<j,\\
0& \text{if}\ i>j.\\
\end{cases}
\]
As for the multiplication, we use the multiplications in $A$ and $V(A)$, and the $A$-$A$-bimodule structure on $V(A)$ induced from $\eta_A$.
More precisely, for elements $x=[x_{ij}]_{ij},y=[y_{ij}]_{ij}\in\Akp$ where $x_{ij},y_{ij}\in\Akp_{ij}$ (in particular $x_{ij}=y_{ij}=0$ if $i>j$), their product $xy=z=[z_{ij}]_{ij}$ is defined by the following.
\[ z_{ij}=\begin{cases}
x_{ii}y_{ii}&\ \text{if}\ i=j\\[11.5pt]
\eta_A(x_{ii})y_{ij}+\big(\displaystyle\sum_{i< m<j}x_{im}y_{mj}\big)+x_{ij}\eta_A(y_{jj})&\ \text{if}\ i<j\\
0&\ \text{if}\ i>j
\end{cases} \]
Then we have an isomorphism $\psi\co \Ak\ov{\cong}{\lra}\Akp$ of $K$-algebras.
\end{prop}

\begin{rem}\label{RemUTA}
As we will show later in Proposition \ref{PropVAdim}, the homomorphism $\eta_A$ becomes injective whenever $A$ is gentle. Thus if $A$ is gentle, identifying $A$ with the image of $\eta_A$, we may regard $\Akp$ as a $K$-subalgebra of $M_k(V(A))$.
\end{rem}

\begin{proof}[Proof of Proposition \ref{PropApK}.]
For each $1\le i,j\le k$, put $\Ak_{ij}=1^{[i]}\Ak1^{[j]}$ for simplicity. By using Proposition \ref{PropBasisSheets}, we give an isomorphism of $K$-modules
\begin{equation}\label{psiij}
\psi_{ij}\co \Ak_{ij}\ov{\cong}{\lra}\Akp_{ij}
\end{equation}
in the following way.
\begin{enumerate}
\item If $i>j$, we have $\Ak_{ij}=\Akp_{ij}=0$.
\item If $i=j$, naturally we have $\Ak_{ii}\cong A^{[i]}$. Define $\psi_{ii}$ to be the inverse of the composition of
\[ \Akp_{ii}=A\un{\cong}{\ov{\ (-)^{[i]}}{\lra}}A^{[i]}\cong \Ak_{ii}. \]
This respects the multiplication.
\item If $i<j$, we have a $K$-basis $(\ref{BasisSheetsB1})$ of $\Ak_{ij}$, and a $K$-basis $(\ref{BasisForVA})$ of $\Akp_{ij}$. Define $\psi_{ij}$ by extending $K$-linearly the following bijective correspondence of the basis elements.
\begin{itemize}
\item[{\rm (i)}] For any $\rho=\al_1\cdots\al_l\in\Mfr_A$ and $0\le u,v\le l$, put
\[ \psi_{ij}\big((\al_{u+1}\cdots\al_l)^{[i]}\rho^{[i,j]}(\al_1\cdots\al_v)^{[j]} \big)=\rho_{u,v}\in V(A)=\Akp_{ij}. \]
\item[{\rm (ii)}] For any $a\in\Tfr_A$, put
\[ \psi_{ij}(\iota_a^{[i,j]})=\efr_a\in V(A)=\Akp_{ij}. \]
\end{itemize}
\end{enumerate}

Summing up $(\ref{psiij})$ for $1\le i,j\le k$, we obtain an isomorphism $\psi\co \Ak\ov{\cong}{\lra}\Akp$ of $K$-modules, which satisfies $\psi|_{A_{ij}^{(k)}}=\psi_{ij}$. It remains to show that $\psi$ respects multiplications of basis elements. Let us show
\begin{equation}\label{EqToShow_ab}
\psi(\afr\bfr)=\psi(\afr)\psi(\bfr)
\end{equation}
for any pair of basis elements $\afr\in \Ak_{hi}$ and $\bfr\in \Ak_{i\ppr j}$ in Proposition \ref{PropBasisSheets}.
Since $\Ak_{hi}=0$ for $h>i$, we may assume $1\le h\le i\le k$. We may also assume $1\le i\ppr\le j\le k$ by the same reason. Besides, if $i\ne i\ppr$ we have $\afr\bfr=0$, and similarly $\psi(\afr)\in \Akp_{hi},\psi(\bfr)\in \Akp_{i\ppr j}$ implies $\psi(\afr)\psi(\bfr)=0$, hence $(\ref{EqToShow_ab})$ is trivially satisfied. Thus it suffices to consider the case $i=i\ppr$, so we may assume
\[ \afr\in \Ak_{hi},\ \bfr\in\Ak_{ij}\quad(1\le h\le i\le j\le k) \]
from the beginning. We show $(\ref{EqToShow_ab})$ in the following 4 cases.
\medskip

\noindent\underline{Case 1}: $h=i=j$.

This case is obvious, since $\psi_{ii}$ respects multiplications.

\medskip

\noindent\underline{Case 2}: $h=i<j$.

In this case, we may assume $\afr=\thh^{[i]}$ for some path $\thh$ in $A$.
Let us show that $\psi_{ij}(\thh^{[i]}\bfr)=\eta_A(\psi_{ii}(\thh^{[i]}))\psi_{ij}(\bfr)$ holds in $V(A)=\Akp_{ij}$, for any element $\bfr\in\Ak_{ij}$ which is of one of the following forms.
\begin{itemize}
\item[{\rm (i)}] $\bfr=(\al_{u+1}\cdots\al_l)^{[i]}\rho^{[i,j]}(\al_1\cdots\al_v)^{[j]}$ for some $\rho=\al_1\cdots\al_l\in\Mfr_A$ and $0\le u,v\le l$.
\item[{\rm (ii)}] $\bfr=\iota_a^{[i,j]}$ for some $a\in\Tfr_A$.
\end{itemize}

For {\rm (i)}, we see that $\thh^{[i]}\bfr=(\thh\al_{u+1}\cdots\al_l)^{[i]}\rho^{[i,j]}(\al_1\cdots\al_v)^{[j]}$ becomes $0$ except for the case
\begin{equation}\label{casexiw}
\thh=\al_{w+1}\cdots\al_u\ \text{for some}\ 0\le w\le u,
\end{equation}
where we put $\al_{w+1}\cdots\al_u=e_{t(\al_u)}$ for convenience when $w=u$. 
Thus we have
\[ \thh^{[i]}\bfr=\begin{cases}
(\al_{w+1}\cdots\al_l)^{[i]}\rho^{[i,j]}(\al_1\cdots\al_v)^{[j]}&\ \text{if}\ (\ref{casexiw})\ \text{holds}\\
0&\ \text{otherwise}
\end{cases} \]
which shows
\[
\psi_{ij}(\thh^{[i]}\bfr)=
\begin{cases}
\rho_{w,u}&\ \text{if}\ (\ref{casexiw})\ \text{holds}\\
0&\ \text{otherwise}
\end{cases}
=\eta_A(\thh)\rho_{u,v}\ =\ \eta_A(\psi_{ii}(\thh^{[i]}))\psi_{ij}(\bfr).
\]

For {\rm (ii)}, similarly we have
\[ \psi_{ij}(\thh^{[i]}\bfr)=\begin{cases}
\efr_a&\ \text{if}\ \thh=e_a\\
0&\ \text{otherwise}
\end{cases}
=\eta_A(\thh)\efr_a=\psi_{ii}(\thh^{[i]})\psi_{ij}(\bfr). \]

\medskip

\noindent\underline{Case 3}: $h<i=j$.

This case can be shown in a similar way as in Case 2.

\medskip

\noindent\underline{Case 4}: $h<i<j$.

If $\afr=\iota_a^{[h,i]}$ for some $a\in\Tfr_A$, then we have
\[
\psi_{hj}(\afr\bfr)=
\begin{cases}
\efr_a&\ \text{if}\ \bfr=\iota_a^{[i,j]}\\
0&\ \text{otherwise}
\end{cases}
=\efr_a\psi_{ij}(\bfr)=\psi_{hi}(\afr)\psi_{ij}(\bfr).
\]
Similarly for the case $\bfr=\iota_b^{[i,j]}$ for some $b\in\Tfr_A$.
Thus it remains to deal with the case where
\[
\afr=(\al_{u+1}\cdots\al_l)^{[h]}\rho^{[h,i]}(\al_1\cdots\al_v)^{[i]},\quad
\bfr=(\be_{p+1}\cdots\be_m)^{[i]}\om^{[i,j]}(\be_1\cdots\be_q)^{[j]}
\]
hold for some
\[
\rho=\al_1\cdots\al_l\in\Mfr_A,\ 0\le u,v\le l,\ \ \text{and}\ \ 
\om=\be_1\cdots\be_m\in\Mfr_A,\ 0\le p,q\le m.
\]
In this case, we have
\begin{eqnarray*}
\afr\bfr&=&
\begin{cases}
(\al_{u+1}\cdots\al_l)^{[h]}\rho^{[h,i]}\rho\rho^{[i,j]}(\al_1\cdots\al_q)&\ \text{if}\ \rho=\om\ \text{and}\ v=p\\
0&\ \text{otherwise}
\end{cases}\\
&=&
\begin{cases}
(\al_{u+1}\cdots\al_l)^{[h]}\rho^{[h,j]}(\al_1\cdots\al_q)&\ \text{if}\ \rho=\om\ \text{and}\ v=p\\
0&\ \text{otherwise}
\end{cases}\ ,
\end{eqnarray*}
hence
\[ \psi_{hj}(\afr\bfr)=\begin{cases}
\rho_{u,q}&\ \text{if}\ \rho=\om\ \text{and}\ v=p\\
0&\ \text{otherwise}
\end{cases} \]
holds in $V(A)$. On the other hand, we have
\[ \psi_{hi}(\afr)\psi_{ij}(\bfr)=\rho_{u,v}\om_{p,q}=
\begin{cases}
\rho_{u,q}&\ \text{if}\ \rho=\om\ \text{and}\ v=p\\
0&\ \text{otherwise}
\end{cases}\ ,
\]
which shows $\psi_{hj}(\afr\bfr)=\psi_{hi}(\afr)\psi_{ij}(\bfr)$.
\end{proof}

By Proposition \ref{PropApK}, we have an isomorphism $\psi\co A^{(k)}\ov{\cong}{\lra}\Akp$ for any locally gentle algebra $A$. In the rest of this manuscript, we identify $\Akp$ with $A^{(k)}$ through this isomorphism.
Let us look at the structure of $V(A)$ more closely, when $A$ is gentle.\
\begin{prop}\label{PropVAdim}
If $A$ is gentle, then the following holds.
\begin{enumerate}
\item For any $a\in Q_0$, we have $\eta_A(e_a)=\ufr_a+\vfr_a$ for a pair of distinct non-zero orthogonal idempotent elements $\ufr_a,\vfr_a\in V(A)$. In particular, $\eta_A(e_a)\ne0$.
\item The idempotent elements $\{\ufr_a\mid a\in Q_0\}\amalg\{\vfr_a\mid a\in Q_0\}$ obtained in {\rm (1)} are pairwise orthogonal, and satisfy
\begin{equation}\label{Decomp1}
1=\displaystyle\sum_{a\in Q_0}(\ufr_a+\vfr_a)
\end{equation}
for the unit $1\in V(A)$.
\item $\eta_A$ is injective.
\item $\dim_KV(A)=2\dim_KA$.
\end{enumerate}
\end{prop}
\begin{proof}
{\rm (1)} Since $A$ is gentle, at least one and at most two maximal paths pass through $a$.

\smallskip

\noindent\underline{Case 1} Suppose that two different maximal paths $\rho,\om\in\Mfr_A$ pass through $a$. In this case, each of $\rho,\om$ passes through $a$ only once, and
\[ \eta_A^{\rho}(e_a)\in\Mcal_{\rho},\quad \eta_A^{\om}(e_a)\in\Mcal_{\om} \]
give the desired idempotents in $V(A)$. As for the other components of $\eta_A$, we have $\eta_A^{\pi}(e_a)=0$ for any $\pi\in\Mfr_A\setminus\{\rho,\om\}$, and $\eta_A^b(e_a)=0$ for any $b\in\Tfr_A$.

\smallskip

\noindent\underline{Case 2} Suppose that only one maximal path $\rho\in\Mfr_A$ passes through $a$. Put $\rho=\al_1\cdots\al_l$ and
\[ s(\al_u)=a_{u-1},\ \ t(\al_u)=a_u\quad(1\le u\le l) \]
as in Definition \ref{DefSA}.
If $\rho$ passes through $a$ twice, then there are $0\le u<v\le l$ such that $a=a_u=a_v$. In this case,
\[ \rho_{u,u}=\ep_u^{\rho},\ \rho_{v,v}=\ep_v^{\rho}\in\Mcal_{\rho} \]
give the desired idempotents in $V(A)$. We have $\eta_A^{\pi}(e_a)=0$ for any $\pi\in\Mfr_A\setminus\{\rho\}$, and $\eta_A^b(e_a)=0$ for any $b\in\Tfr_A$.

If $\rho$ passes through $a$ only once, then we have $a\in\Tfr_A$. There is a unique $0\le u\le l$ such that $a=a_u$. In this case,
\[ \eta_A^{\rho}(e_a)=\rho_{u,u}=\ep_u^{\rho}\in\Mcal_{\rho}\ \ \text{and}\ \ \eta_A^a(e_a)=\efr_a\in K\efr_a \]
give the desired idempotents in $V(A)$. We have $\eta_A^{\pi}(e_a)=0$ for any $\pi\in\Mfr_A\setminus\{\rho\}$, and $\eta_A^b(e_a)=0$ for any $b\in\Tfr_A\setminus\{ a\}$.

{\rm (2)} The proof of {\rm (1)} shows that each $\ufr_a,\vfr_a$ is given by
\begin{itemize}
\item[{\rm (i)}] $\ep_u^{\rho}\in\Mcal_{\rho}$ for some $\rho\in\Mfr_A$ and $0\le u\le\length(\rho)$, or
\item[{\rm (ii)}] $\efr_a\in K\efr_a$ for some $a\in\Tfr_A$.
\end{itemize}
Moreover, we see easily that $\{\ufr_a,\vfr_a\}\cap\{\ufr_b,\vfr_b\}$ for $a\ne b$. This implies that these $2|Q_0|$ elements $\{\ufr_a,\vfr_a\}_{a\in Q_0}$ are distinct. Different elements {\rm (i),(ii)} are obviously mutually orthogonal. Equation $(\ref{Decomp1})$ follows from
\[ 1=\eta_A(1)=\eta_A\big(\displaystyle\sum_{a\in Q_0}e_a\big)=\displaystyle\sum_{a\in Q_0}\eta_A(e_a)=\sum_{a\in Q_0}(\ufr_a+\vfr_a). \]

{\rm (3)} As $K$-modules, we have
\[ A=\displaystyle\bigoplus_{a,b\in Q_0}e_aAe_b\quad \text{and}\quad V(A)=\displaystyle\bigoplus_{a,b\in Q_0}\eta_A(e_a)V(A)\eta_A(e_b). \]
It suffices to show that
\[ \eta_A|_{e_aAe_b}\co e_aAe_b\to \eta_A(e_a)V(A)\eta_A(e_b)\se V(A) \]
is monomorphic for any $a,b\in Q_0$. Let us assign a grading by $\deg_1$ as in Remark \ref{RemGrading}. Since $e_aAe_b=\displaystyle\bigoplus_{d\in\Zbb_{\ge0}}(e_aA_de_b)$ and since $\eta_A$ preserves the grading, it is enough to show that
\begin{equation}\label{ToShowMonom}
\eta_A|_{e_AA_de_b}\co e_aA_de_b\to V_d(A)\se V(A)
\end{equation}
is monomorphic for any $a,b\in Q_0$ and any $d\in\Zbb_{\ge0}$.
If $d=0$, then we have $e_aA_0e_b=0$ unless $a=b$. If moreover $a=b$, then $e_aA_0e_a=Ke_a$ is $1$-dimensional. Thus the monomorphicity of $(\ref{ToShowMonom})$ follows from $\eta_A(e_a)\ne0$ shown in {\rm (1)}.

Thus it suffices to deal with the case $d>0$.
For simplicity, let $\Bfr_d^{a,b}$ denote the set of paths in $A$ from $a$ to $b$ of length $d$.
Since $A$ is gentle, we have $|\Bfr_d^{a,b}|\le 2$. 
For any $\xi\in \Bfr_d^{a,b}$, there is a unique $\rho=\al_1\cdots\al_l\in\Mfr_A$ with $l\ge d$ and $1\le u\le l-d$ such that $\xi=\al_u\cdots\al_{u+d}$. By definition we have $\eta_A^{\rho}(\xi)=\rho_{u-1, u+d}\ne0$. Thus if $|\Bfr_d^{a,b}|\le 1$, then $(\ref{ToShowMonom})$ is monomorphic.

It remains to show in the case $|\Bfr_d^{a,b}|=2$. Put $\Bfr_d^{a,b}=\{\xi,\ze\}$. Then we have $e_aA_de_b=K\xi\oplus K\ze$ and $\dim_K(e_aA_de_b)=2$.
Take the maximal paths $\rho=\al_1\cdots\al_l,\om=\be_1\cdots\be_m\in\Mfr_A$ with $l,m\ge d$ which give
\[ \xi=\al_u\cdots\al_{u+d}\quad\text{and}\quad \ze=\be_v\cdots\be_{v+d} \]
for some $1\le u\le l-d$ and $1\le v\le m-d$.
To show $(\ref{ToShowMonom})$ is monomorphic, it suffices to show that $\eta_A(\xi),\eta_A(\ze)\in V(A)$ are linearly independent over $K$.

If $\rho\ne\om$, then we have
\[
\eta_A^{\pi}(\xi)=
\begin{cases}
\rho_{u-1, u+d}\ne0&\ \text{if}\ \pi=\rho\\
0&\ \text{otherwise}
\end{cases}
\ \ ,\ \ 
\eta_A^{\pi}(\ze)=
\begin{cases}
\om_{v-1\, v+d}\ne0&\ \text{if}\ \pi=\om\\
0&\ \text{otherwise}
\end{cases}
\]
for any $\pi\in\Mfr_A$, which obviously implies that $\eta_A(\xi),\eta_A(\ze)\in V(A)$ are linearly independent over $K$. 

If $\rho=\om$, then $u\ne v$ follows from $\xi\ne\ze$. Thus
\[ \eta_A^{\rho}(\xi)=\rho_{u-1, u+d}\quad\text{and}\quad \eta_A^{\rho}(\ze)=\rho_{v-1, v+d} \]
are $K$-linearly independent in $\Mcal_{\rho}$, and thus so are $\eta_A(\xi),\eta_A(\ze)\in V(A)$.

{\rm (4)} Let $\rho_1,\ldots,\rho_r$ denote the all maximal paths in $A$, and put $l_p=\length(\rho_p)$ for any $1\le p\le r$. As in Remark \ref{RemCountTA}, we have
\[ |\Tfr_A|=2|Q_0|-\displaystyle\sum_{1\le p\le r}(l_p+1).  \]
On the other hand, since any path of positive length is a part of a (unique) maximal path because $A$ is gentle, we have
\[ |\{\text{paths in $A$ of positive lengths}\}|=\displaystyle\sum_{1\le p\le r}\frac{l_p(l_p+1)}{2}. \]
Thus we obtain
\begin{eqnarray*}
\dim_KV(A)&=&\displaystyle\sum_{1\le p\le r}\dim_K(\Mcal_{\rho})+\displaystyle\sum_{a\in\Tfr_A}\dim_K(K\efr_a)\\
&=&\displaystyle\sum_{1\le p\le r}(l_p+1)^2+2|Q_0|-\displaystyle\sum_{1\le p\le r}(l_p+1)\\
&=&2(|\{\text{paths in $A$ of positive lengths}\}|+|Q_0|)\\
&=&2\dim_KA.
\end{eqnarray*}
\end{proof}

By Proposition \ref{PropVAdim} {\rm (3)}, in particular $\eta_A$ induces a short exact sequence
\[ 0\to A\ov{\eta_A}{\lra}V(A)\to\Cok\eta_A\to0 \]
of $A$-$A$-bimodules if $A$ is gentle. Let us show some property of $\Cok\eta_A$. In the following, let $DA=\Hom_K(A,K)$ be the standard $K$-dual of $A$, and let $\{\xi^{\bullet}\mid \xi\ \text{is a path in}\ A\}$ denote the dual basis.

\begin{dfn}\label{DefAlmostStandard}
An $A$-$A$-bimodule $E$ is called an {\it almost standard dual of} $A$ if it has a $K$-basis
\[ \{ \xi^{\dag}\mid \xi\in A\ \text{is a path} \} \]
which satisfies the following conditions.
\begin{enumerate}
\item $A_0=KQ_0$ acts on $E$ in the same way as on $DA$. Namely,
\[ e_a\xi^{\dag}=\begin{cases}\xi^{\dag}&\ \text{if}\ t(\xi)=a\\0&\ \text{otherwise}\end{cases} \]
holds for any $a\in Q_0$. Similarly for $\xi^{\dag}e_a$. In other words, we have $\xi^{\dag}\in e_{t(\xi)}Ee_{s(\xi)}$.
\item For any path $\thh\in A$ of positive length and any path $\xi\in A$, the following holds for $\thh\xi^{\dag}\in E$. Similarly for $\xi^{\dag}\thh$.
\begin{itemize}
\item $\thh\xi^{\dag}=0$ holds in $E$ if and only if $\thh\xi^{\bullet}=0$ holds in $DA$.

In particular, we have $\thh\xi^{\dag}=0$ whenever $\length(\thh)>\length(\xi)$.
\item If $\length(\thh)<\length(\xi)$, then $\thh\xi^{\dag}=\om^{\dag}$ holds in $E$ for a path $\om\in A$ if and only if $\thh\xi^{\bullet}=\om^{\bullet}$ holds in $DA$ for the same path.
\item If $\length(\thh)=\length(\xi)$, then
\begin{equation}\label{Eq_C}
\thh\xi^{\dag}=Ce_a^{\dag}
\end{equation}
hold for some $C\in K\setminus\{0\}$ if and only if $\thh\xi^{\bullet}=e_a^{\bullet}$ holds in $DA$ (if and only if $\thh=\xi$ and $s(\xi)=a$).
\end{itemize}
Those coefficients $C$ can differ accordingly to $\xi$. We do not require either $C=C\ppr$ for $\xi\xi^{\dag}=Ce_{s(\xi)}^{\dag}$ and $\xi^{\dag}\xi=C\ppr e_{t(\xi)}^{\dag}$.
\end{enumerate}
\end{dfn}

\begin{rem}
Let $E$ be an almost standard dual of $A$.
\begin{enumerate}
\item By putting $\deg(\xi^{\dag})=-\length(\xi)$, we may also regard $E$ as a $\Zbb_{\le0}$-graded module.
\item There is a short exact sequence of $A$-$A$-bimodules
\[ 0\to D_0A\to E\to DA/D_0A\to0, \]
where $D_0A=\mathrm{soc}DA$ is the degree zero part of $DA$.
\end{enumerate}
\end{rem}

\begin{prop}\label{PropCoketaSD}
$\Cok\eta_A$ is an almost standard dual of $A$.
\end{prop}
\begin{proof}
Take a temporary grading $V(A)=\un{d\in\Zbb}{\bigoplus}V_d(A)$ by $\deg_1\co Q_1\to\Zbb$ as in Remark \ref{RemGrading}, for the convenience in the argument. Then the $K$-basis $(\ref{BasisForVA})$ can be divided into the following $K$-basis of the homogeneous parts $V_d(A)$ for $d\in\Zbb$.
\begin{itemize}
\item $\un{\rho\in\Mfr_A}{\coprod}\{\rho_{u,v}\mid 0\le u<v\le \length(\rho),\, v-u=d\}$ \ if $d>0$.
\item $\big(\un{\rho\in\Mfr_A}{\coprod}\{\rho_{u,u}=\ep^{\rho}_u\mid 0\le u\le \length(\rho)\}\big)\amalg\{\efr_a\mid a\in\Tfr_A\}=\un{a\in Q_0}{\coprod}\{\ufr_a,\vfr_a\}$ \ if $d=0$.
\item $\un{\rho\in\Mfr_A}{\coprod}\{\rho_{v,u}\mid 0\le u<v\le \length(\rho),\, u-v=d\}$ \ if $d<0$.
\end{itemize}
For any $\rho=\al_1\cdots\al_l\in\Mfr_A$ and for any $0\le u<v\le\length(\rho)$, the path $\al_{u+1}\cdots\al_v\in A$ satisfies $\eta_A(\al_{u+1}\cdots\al_v)=\rho_{u,v}$ by definition. This means that $\eta_A$ surjects $\un{d>0}{\bigoplus}A_d$ onto $\un{d>0}{\bigoplus}V_d(A)$. Thus the above $K$-basis gives the following $K$-basis of $\Cok\eta_A=V(A)/\eta_A(A)$.
\begin{equation}\label{BBCC}
\big(\un{\rho\in\Mfr_A}{\coprod}\{\ovl{\rho_{v,u}}\mid 0\le u<v\le\length(\rho)\}\big)\amalg\big(\un{a\in Q_0}{\coprod}\{\ovl{\ufr_a}=-\ovl{\vfr_a}\}\big).
\end{equation}
Here, $\ovl{\ze}$ denotes the image of $\ze\in V(A)$ in $\Cok\eta_A$. The equation $\ovl{\ufr_a}=-\ovl{\vfr_a}$ is a consequence of $\eta_A(e_a)=\ufr_a+\vfr_a$. Let us define a subset $\{\xi^{\dag}\mid \xi\in A\ \text{is a path}\}$ of $\Cok\eta_A$ by the following.
\begin{itemize}
\item For any path $\xi\in A$ with $\length(\xi)>0$, there uniquely exists
\begin{equation}\label{Forrhoexists}
\rho=\al_1\cdots\al_l\in\Mfr_A\ \text{and}\ 0\le u<v\le l\ \text{such that}\ \xi=\al_{u+1}\cdots\al_v.
\end{equation}
Using this, put $\xi^{\dag}=\ovl{\rho_{v,u}}$.
\item For any $a\in Q_0$, choose one of $\ovl{\ufr_a}$ or $\ovl{\vfr_a}$, and denote it by $e_a^{\dag}$. Since $\ovl{\ufr_a}=-\ovl{\vfr_a}$ holds in $\Cok\eta_A$, anyway we have $e_a^{\dag}=\pm\ovl{\ufr_a}$.
\end{itemize}
The above argument shows that this set $\{\xi^{\dag}\mid \xi\in A\ \text{is a path}\}$ indeed forms a $K$-basis of $\Cok\eta_A$. Let us check the conditions {\rm (1),(2)} in Definition \ref{DefAlmostStandard}.

\medskip

{\rm (1)} Let $a\in Q_0$ be any vertex, and let $\xi\in A$ be any path. It suffices to show $\xi^{\dag}\in e_{t(\xi)}(\Cok\eta_A)e_{s(\xi)}$. If $\length(\xi)>0$, then we have $\xi^{\dag}=\ovl{\rho_{v,u}}$ with $(\ref{Forrhoexists})$ taken as above. Then the fact $\rho_{v,u}\in\ep^{\rho}_vV(A)\ep^{\rho}_u\se e_{t(\xi)}V(A)e_{s(\xi)}$ implies $\ovl{\rho_{v,u}}\in e_{t(\xi)}(\Cok\eta_A)e_{s(\xi)}$. If $\length(\xi)=0$, then $\xi$ satisfies $\xi=e_{t(\xi)}=e_{s(\xi)}$, and thus $\xi^{\dag}=\pm\ovl{\ufr_{t(\xi)}}\in e_{t(\xi)}(\Cok\eta_A)e_{s(\xi)}$ follows from Proposition \ref{PropVAdim} {\rm (1),(2)}.

\medskip

{\rm (2)} Let $\thh,\xi\in A$ be any pair of paths, with $\length(\thh)>0$. If $\length(\xi)=0$, obviously we have $\thh\xi^{\bullet}=0$ and $\thh\xi^{\dag}=0$. Thus we may assume $\length(\xi)>0$. Take $(\ref{Forrhoexists})$ as above. Then $\thh\xi^{\dag}\ne0$ holds if and only if
$\thh=\al_{w+1}\cdots\al_v$
holds for some $u\le w<v$. Moreover, the following holds for such $\thh$.
\begin{itemize}
\item If $u<w$, then $\thh\xi^{\dag}=\eta_A(\thh)\ovl{\rho_{v,u}}=\ovl{\rho_{w,v}}\ovl{\rho_{v,u}}=\ovl{\rho_{w,u}}=(\al_{u+1}\cdots\al_w)^{\dag}$.
\item If $u=w$, then $\thh\xi^{\dag}=\ovl{\rho_{u,v}}\ovl{\rho_{v,u}}=\ovl{\ep^{\rho}_u}=\pm e_{s(\xi)}^{\dag}$.
\end{itemize}
This shows
\[ \thh\xi^{\dag}=\begin{cases}
(\al_{u+1}\cdots\al_w)^{\dag}&\ \text{if}\ \thh=\al_{w+1}\cdots\al_v\ \text{and}\ u<w<v,\\
\pm e_{s(\xi)}^{\dag}&\ \text{if}\ \thh=\xi,\\
0&\ \text{otherwise}.
\end{cases}
\]
On the other hand, from the $A$-$A$-bimodule structure on $DA$, we have
\[  \thh\xi^{\bullet}=\begin{cases}
(\al_{u+1}\cdots\al_w)^{\bullet}&\ \text{if}\ \thh=\al_{w+1}\cdots\al_v\ \text{and}\ u<w<v,\\
e_{s(\xi)}^{\bullet}&\ \text{if}\ \thh=\xi,\\
0&\ \text{otherwise}.
\end{cases}
\]
Thus {\rm (2)} is satisfied.
\end{proof}

\begin{rem}\label{RemCpm1}
The above proof of Proposition \ref{PropCoketaSD} shows that $\xi\xi^{\dag}=\pm e_{s(\xi)}^{\dag}$ holds in $\Cok\eta_A$, for any path $\xi\in A$ of positive length. Similarly for $\xi^{\dag}\xi$.

Namely, the coefficients $C$ in $(\ref{Eq_C})$ can be taken to satisfy $C=\pm1$. However, we do not require $C$ to be $\pm1$ a priori in Definition \ref{DefAlmostStandard} in general, since $C\ne0$ is enough for our purpose. In fact, in the situation we are considering in Section \ref{section_Characterize}, the arguments there implicitly shows that we may replace those $\xi^{\dag}$ appropriately, to adjust $C$ to become $\pm 1$ (by using $c_{\rho}$ in Lemma \ref{LemQuasiInverse} and Proposition \ref{PropQuasiInverse}).
\end{rem}

\begin{rem}\label{RemQuestionI}
We remark that there also exists an $A$-$A$-submodule $N\se V(A)$ which induces an isomorphism $V(A)/N\cong DA$ of $A$-$A$-bimodules. In fact,
$N=\big(\bigoplus_{d>0}V_d(A)\big)\oplus\big(\bigoplus_{a\in Q_0}K(\ufr_a-\vfr_a)\big)$
satisfies this property, where $V(A)=\un{d\in\Zbb}{\bigoplus}V_d(A)$ is the grading by $\deg_1$, similarly as in the proof of Proposition \ref{PropCoketaSD}. Difference between $N$ and $\eta_A(A)$ lies only in the degree $0$ part.

Indeed, if we denote the image of $\xi\in V(A)$ in $V(A)/N$ by $\und{\xi}$, then $V(A)/N$ has a $K$-basis
$\big(\un{\rho\in\Mfr_A}{\coprod}\{\und{\rho_{u,v}}\mid 0\le u<v\le\length(\rho)\}\big)\amalg\big(\un{a\in Q_0}{\coprod}\{\und{\ufr_a}=\und{\vfr_a}\}\big)$,
which corresponds bijectively to the dual basis in $DA$, respecting the $A$-$A$-bimodule structure.

Consequently, we obtain an epimorphism of $A$-$A$-bimodules $V(A)\to DA$. This induces the following sequence of surjective homomorphisms of $K$-algebras.
\[ \Ak\cong\Akp
\to
\left[
\begin{array}{cccccc}
A&V(A)&0&\cdots&\cdots&0\\
0&A&V(A)&0&\cdots&0\\
0&0&A&V(A)&\ddots&\vdots\\
\vdots&\vdots&\ddots&\ddots&\ddots&0\\
0&\cdots&\cdots&0&A&V(A)\\
0&\cdots&\cdots&\cdots&0&A\\
\end{array}
\right]
\to
\left[
\begin{array}{cccccc}
A&DA&0&\cdots&\cdots&0\\
0&A&DA&0&\cdots&0\\
0&0&A&DA&\ddots&\vdots\\
\vdots&\vdots&\ddots&\ddots&\ddots&0\\
0&\cdots&\cdots&0&A&DA\\
0&\cdots&\cdots&\cdots&0&A\\
\end{array}
\right]
\]
We remark that the rightmost $K$-algebra is the one obtained by the usual (but finitely many times) repetition of $A$.
\end{rem}

\section{Characterization of $V(A)$}\label{section_Characterize}

In this section, we always assume that $A$ is gentle,  
and moreover satisfies the following condition. 
\begin{cond}\label{CondMaxLength1}
For any maximal path $\rho\in A$, there is no arrow $\al\in Q_1$ satisfying $s(\al)=s(\rho)$, $t(\al)=t(\rho)$ and $\al\ne\rho$.
\end{cond}
Throughout this section, let
\begin{equation}\label{ExBimodV}
0\to A\ov{\eta\ppr}{\lra}V\ov{\pi}{\lra}E\to 0
\end{equation}
be a short exact sequence of $A$-$A$-bimodules with the following properties.
\begin{itemize}
\item[{\rm (v1)}] $V$ is a semisimple $K$-algebra.
\item[{\rm (v2)}] $\eta\ppr$ is a morphism of $K$-algebras.
\item[{\rm (v3)}] $E$ is an almost standard dual of $A$.
\end{itemize}
We remark that $\dim_KV=2$ also follows from the existence of $(\ref{ExBimodV})$.

As we have seen, $\eta_A\co A\hookrightarrow V(A)$ satisfies them. 
In this section, we will see that $V(A)$ is characterized by these properties uniquely up to isomorphism. 
Indeed, we show the following.
\begin{thm}\label{ThmCharacterizeVA}
Assume $\Char(K)\ne 2$, and that $A$ satisfies Condition \ref{CondMaxLength1}. Then for any $(\ref{ExBimodV})$ satisfying {\rm (v1),(v2),(v3)}, there exists an isomorphism of $K$-algebras $\vp\co V\ov{\cong}{\lra}V(A)$ with $\vp\ci\eta\ppr=\eta_A$.
\end{thm}
Its proof will be given at the end of this section.
In the rest, we regard $A$ as a $K$-subalgebra of $V$ through $\eta\ppr$, and abbreviate $\eta\ppr(\xi)\in V$ simply to $\xi$, for any $\xi\in A$.
We will often use the following observation.
\begin{rem}\label{RemUseOften}
For any pair of vertices $a,b\in Q_0$, the sequence $(\ref{ExBimodV})$ induces a short exact sequence of $K$-modules
\[ 0\to e_aAe_b\to e_aVe_b\to e_aEe_b\to0. \]
In particular, for any $K$-basis $\{ \vartheta_j\mid j\in J \}$ of $e_aAe_b$ and any $K$-basis $\{ \chi_s\mid s\in S \}$ of $e_aEe_b$, any choice of lifts $\wt{\chi}_s$ of $\chi_s$ to $e_aVe_b$ for $s\in S$ gives a $K$-basis $\{ \vartheta_j\mid j\in J\}\amalg\{ \wt{\chi}_s\mid s\in S\}$ of $e_aVe_b$.
\end{rem}

\begin{lem}\label{LemTildeVanish}
Let $\rho=\al_1\cdots\al_{l}\in e_aAe_b$ be any maximal path, and let $\thh\in e_aAe_c$ be any path which does not start with $\al_1$. Then the following holds for any lift $\wt{\thh}\in V$ of $\thh^{\dag}\in E$.
\begin{enumerate}
\item If $l=1$, then $\wt{\thh}\rho=0$ holds.
\item If $l\ge2$, then we have $\wt{\thh}\al_1\al_2=0$, and thus $\wt{\thh}\rho=0$ holds also in this case.
\end{enumerate}
Similarly, if a path $\om\in e_dAe_c$ does not end with $\al_{l}$, we have $\rho\wt{\om}=0$.
\end{lem}
\begin{proof}
By the assumption on $\rho$ and $\thh$, we can easily deduce $a\ne c$. If both $\wt{\thh}$ and $\wt{\thh}\ppr$ are lifts of $\thh^{\dag}$, then $\wt{\thh}-\wt{\thh}\ppr\in\Ker\pi=A$ should satisfy $(\wt{\thh}-\wt{\thh}\ppr)\rho=0$ since $\rho$ is maximal. Thus the element $\wt{\thh}\rho\in V$ does not depend on the choice of lift $\wt{\thh}$. We may also assume $\wt{\thh}$ is in $e_cAe_a$.

First let us reduce to the case where $\thh$ is maximal. If there is any arrow $\be\in A$ of positive length such that $\be\thh\ne0$, then for any lift $\wt{(\be\thh)}$ of $(\be\thh)^{\dag}$, we have
\[ \wt{\thh}\rho=\wt{(\be\thh)}\be\rho=0. \]
Thus it is enough to consider the case with no such $\be$. In this case, there exists a path $\kappa\in A$ (possibly $\kappa=e_b$) such that $\thh\kappa$ is a maximal path. This gives $\wt{\thh}\rho=\kappa\wt{(\thh\kappa)}\rho$,
and $\wt{\thh}\rho=0$ should automatically follow if $\wt{(\thh\kappa)}\rho=0$ is shown.

Thus we may assume that $\thh$ is maximal, from the beginning. Then since $\rho,\thh$ are different maximal paths, it follows that $a$ is a source vertex. In particular we have $a\ne b$, in addition to $a\ne c$.
It remains to show {\rm (1)} and {\rm (2)} under these assumptions.

\medskip

{\rm (1)} By Condition \ref{CondMaxLength1}, we have $b\ne c$ in this case. Since $\pi(\wt{\thh}\rho)=\thh^{\dag}\rho=0$, we have $\wt{\thh}\rho\in e_cAe_b$. Thus if $\wt{\thh}\rho\ne0$, 
then it should be a $K$-linear sum of paths $\ze\in e_cAe_b$. Since $\thh$ is maximal, we have $\thh\ze=0$ for any path $\ze\in e_cAe_b$. Moreover by $l=1$ and $a\ne c$, we see that $\ze$ does not end with $\al_1$. By these properties, such path $\ze$ should be unique, thus we obtain 
\begin{equation}\label{Eq_wtrxz}
\wt{\thh}\rho=x\ze
\end{equation}
for some $x\in K$ and the path $\ze\in e_cAe_b$.
It is enough to show that $(\ref{Eq_wtrxz})$ forces $x=0$. 

If there is any arrow $\gam\in A$ with $\ze\gam\ne0$, then $(\ref{Eq_wtrxz})$ implies $x\ze\gam=\wt{\thh}\rho\gam=0$ and hence $x=0$.
Thus we may assume that there is no such $\gam$. In this case $\ze$ cannot be a part of $\thh$, since they have different targets $b\ne c$. It also implies that $\thh$ is the unique path in $A$ from $a$ to $c$. Indeed, otherwise the existence of $\rho$ tells us that any other path from $a$ to $c$ is a part of $\thh$, and thus
\[ \thh=\thh_1\thh_2 \quad(\thh_1\in e_aAe_c,\thh_2\in e_cAe_c) \]
should hold for some pair $\thh_1,\thh_2$ of paths of positive lengths. Then $\thh\ze=0$ induces $\thh_1\ze\ne0$, which in turn should imply that $\ze$ is a part of $\thh_2$, which is a contradiction.

Take a quasi-inverse $\ze\dia\in e_bVe_c$ of $\ze$. Namely, $\ze\dia$ is an element satisfying $\ze\ze\dia\ze=\ze$. Such $\ze\dia$ always exists since $V$ is semisimple. By the argument so far, there is no path in $e_cAe_a$ nor in $e_aAe_c$ other than $\thh$, thus
\[ \rho\ze\dia=y\thh \]
holds for some $y\in K$. Then
\[ x\ze=x\ze\ze\dia\ze=\wt{\thh}\rho\ze\dia\ze=y\wt{\thh}\thh\ze=0 \]
follows from $(\ref{Eq_wtrxz})$, and hence $x=0$ as desired.

{\rm (2)} Put $a_2=t(\al_1)$ for simplicity. By $\pi(\wt{\thh}\al_1)=\thh^{\dag}\al_1=0$, we have $\wt{\thh}\al_1\in e_cAe_{a_2}$.

%
%
By Condition \ref{CondMaxLength1}, we have $c\ne a_2$.
If $\wt{\thh}\al_1=0$, there is nothing to show. Otherwise, there exists a unique path $\ze\in e_cAe_{a_2}$ 
which gives
$\wt{\thh}\al_1=x\ze$
for some $x\in K$, similarly as in {\rm (1)}. Since $\rho$ is maximal and $a\ne c$, this $\ze$ does not end with $\al_1$, and hence $\ze\al_2=0$. This shows $\wt{\thh}\al_1\al_2=x\ze\al_2=0$.
%
%
%
%
\end{proof}

We will use the following properties of $A$. In the following, for any pair of vertices $a,b\in Q_0$, let $\Bp^{a,b}$ denote the set of paths in $e_aAe_b$ of positive lengths. Remark that we always have $|\Bp^{a,b}|\le 4$ and $|\Bp^{a,a}|\le 1$, since $A$ is gentle.
\begin{lem}\label{LemPositivePathCount}
Assume that there is a maximal path $\rho\in e_aAe_b$ with $a\ne b$. Then the following holds.
\begin{enumerate}
\item $|\Bp^{b,a}|\le1$.
\item If $|\Bp^{a,a}|=1$, then $\rho=\lam\om$ holds for some $\om\in\Bp^{a,b}$ and $\lam\in\Bp^{a,a}$. In particular the path $\lam$ is a subpath of $\rho$, and we have $|\Bp^{a,b}|\ge 2$.
\item If $|\Bp^{b,b}|=1$, then $\rho=\om\kap$ holds for some $\om\in\Bp^{a,b}$ and $\kap\in\Bp^{b,b}$. In particular $\kap$ is a subpath of $\rho$, and we have $|\Bp^{a,b}|\ge 2$.
\item If $|\Bp^{b,a}|=1$, than $|\Bp^{a,a}|=|\Bp^{b,b}|$.
\item Assume that there is $\om\in\Bp^{a,b}$ which does not start with a common arrow with $\rho$ nor end with a common arrow with $\rho$. Then $|\Bp^{b,a}|=0$ holds.
\end{enumerate}
\end{lem}
\begin{proof}
{\rm (1)} If there is $\ze\in\Bp^{b,a}$, it should satisfy $\ze\rho=\rho\ze=0$. Since $A$ is gentle, such $\ze$ is unique. This means $|\Bp^{b,a}|\le1$.

{\rm (2)} If there exists $\lam\in\Bp^{a,a}$, it should satisfy $\lam^2=\lam\rho=0$ since $A$ is gentle and $\rho$ is maximal. This in particular means that $\lam$ and $\rho$ start with a common arrow. Since $\rho$ is maximal, there is a path $\om$ such that $\rho=\lam\om$. Especially we have $\Bp^{a,b}\ni\om\ne\rho$, hence $|\Bp^{a,b}|\ge2$.

{\rm (3)} This is shown in a similar way as {\rm (2)}.

{\rm (4)} Assume $|\Bp^{b,a}|=1$. By {\rm (1)}, it suffices to show the equivalence of $|\Bp^{a,a}|=1$ and $|\Bp^{b,b}|=1$. We only show that $|\Bp^{a,a}|=1$ implies $|\Bp^{b,b}|=1$, since the converse can be shown in a similar way.
Suppose that $|\Bp^{a,a}|=1$ holds, and put $\Bp^{a,a}=\{\lam\}$. By {\rm (2)}, we have $\rho=\lam\om$ for some $\om\in\Bp^{a,b}$. Since $\ze\rho=0$, the gentleness of $A$ implies $\ze\om\ne0$, which means $\ze\om\in\Bp^{b,b}$ and thus $|\Bp^{b,b}|=1$.

{\rm (5)} If there is any $\ze\in\Bp^{b,a}$, then $\ze\rho=\rho\ze=0$ and the gentleness of $A$ imply $\ze\om\ne0$ and $\om\ze\ne0$. Then $\ze\om$ forms an oriented cycle in $A$, which is a contradiction. Thus we have $\Bp^{b,a}=\emptyset$.\end{proof}

\begin{lem}\label{LemMaxAtMost3}
Assume that there is a maximal path $\rho\in e_aAe_b$ with $a\ne b$. Then the value of $\nfr=(|\Bp^{a,b}|,|\Bp^{b,a}|,|\Bp^{a,a}|,|\Bp^{b,b}|)$ becomes one of the following.
\[ (1,0,0,0),(1,1,0,0),(2,0,0,0),(2,0,1,0),(2,0,0,1),(4,0,1,1),(3,1,1,1). \]
In each case, the following holds.
\begin{itemize}
\item[{\rm (i)}] If $\nfr=(1,0,0,0)$, then obviously $\Bp^{a,b}=\{\rho\}$.
\item[{\rm (ii)}] If $\nfr=(1,1,0,0)$, then $\Bp^{a,b}=\{\rho\}$ and $\Bp^{b,a}=\{\ze\}$ for some path $\ze$. 
\item[{\rm (iii)}] If $\nfr=(2,0,0,0)$, then $\Bp^{a,b}=\{\rho,\om\}$ for some path $\om$. These $\rho$ and $\om$ do not start with a common path, nor end with a common path.
\item[{\rm (iv)}] If $\nfr=(2,0,1,0)$, then $\Bp^{a,b}=\{\rho,\om\}$ and $\Bp^{a,a}=\{\lam\}$ hold for some paths $\om$ and $\lam$ satisfying $\rho=\lam\om$.
\item[{\rm (v)}] If $\nfr=(2,0,0,1)$, then $\Bp^{a,b}=\{\rho,\om\}$ and $\Bp^{b,b}=\{\kap\}$ hold for some paths $\om$ and $\kap$ satisfying $\rho=\om\kap$.
\item[{\rm (vi)}] If $\nfr=(4,0,1,1)$, then $\Bp^{a,b}=\{\rho,\lam\om,\om\kap,\om\}$, $\Bp^{a,a}=\{\lam\}$ and $\Bp^{b,b}=\{\kap\}$ hold for some paths $\lam,\om,\kap$ satisfying $\rho=\lam\om\kap$.
\item[{\rm (vii)}] If $\nfr=(3,1,1,1)$, then $\Bp^{a,b}=\{\rho,\om,\tau\}$, $\Bp^{b,a}=\{\ze\}$, $\Bp^{a,a}=\{\tau\ze\}$ and $\Bp^{b,b}=\{\ze\om\}$ hold for some paths $\tau,\ze,\om$ satisfying $\rho=\tau\ze\om$.
\end{itemize}
We also remark that whenever there exists a path $\om\in e_aAe_b$ different from $\rho$, then Condition \ref{CondMaxLength1} implies $\length(\om)\ge2$.
\end{lem}
\begin{proof}
Put $\rho=\al_1\cdots\al_l$, and put $a_i=t(\al_i)$ for $1\le i\le l-1$.

\medskip

\noindent\und{Case 1} Suppose $\{a,b\}\cap\{a_1,\ldots,a_{l-1}\}=\emptyset$ holds. Then there is no subpath of $\rho$ which lies in $\Bp^{a,a}$ or $\Bp^{b,b}$. By Lemma \ref{LemPositivePathCount} {\rm (2)} and {\rm (3)}, this show $|\Bp^{a,a}|=|\Bp^{b,b}|=0$.

If there exists a path $\om\in\Bp^{a,b}$ different from $\rho$, then $\Bp^{a,a}=\Bp^{b,b}=\emptyset$ implies that $\rho$ and $\om$ do not start with a common arrow nor end with a common arrow. Since $A$ is gentle, this means that such $\om$ is unique. Besides, $|\Bp^{b,a}|=0$ follows from Lemma \ref{LemPositivePathCount} {\rm (5)}. Thus if $\{a,b\}\cap\{a_1,\ldots,a_{l-1}\}=\emptyset$, then we have
\[ \nfr=(1,0,0,0),(1,1,0,0),(2,0,0,0), \]
which correspond to the cases {\rm (i),(ii),(iii)}.

\medskip

\noindent\und{Case 2} Suppose $\{a,b\}\cap\{a_1,\ldots,a_{l-1}\}=\{a\}$ holds. Similarly as in Case 1, we have $|\Bfr^{b,b}|=0$. There is a unique $1\le u\le l-1$ such that $a_u=a$. If we put $\lam=\al_1\cdots\al_u$ and $\om=\al_{u+1}\cdots\al_l$, they give $\rho=\lam\om$, $\Bp^{a,a}\ni\lam$ and $\Bp^{a,b}\supseteq\{\rho,\om\}$. In particular we have $|\Bp^{a,a}|=1\ne0=|\Bp^{b,b}|$, which shows $|\Bp^{b,a}|=0$ by Lemma \ref{LemPositivePathCount} {\rm (4)}. In this case $\al_1$ and $\al_{u+1}$ are the only outgoing arrows at $a$, which shows that any path in $\Bp^{a,b}$ should be a subpath of $\rho$. Since $b\notin\{a_1,\ldots,a_{l-1}\}$, this means $\Bp^{a,b}=\{\rho,\om\}$. Thus if $\{a,b\}\cap\{a_1,\ldots,a_{l-1}\}=\{a\}$, this corresponds to the case {\rm (iv)}.

\medskip

\noindent\und{Case 3} Suppose $\{a,b\}\cap\{a_1,\ldots,a_{l-1}\}=\{b\}$ holds. Similarly as in Case 2, this corresponds to the case {\rm (v)}.

\medskip

\noindent\und{Case 4} Suppose $\{a,b\}\cap\{a_1,\ldots,a_{l-1}\}=\{a,b\}$ holds. There exist $1\le u,v\le l-1$ satisfying $a_u=a$ and $a_v=b$.
If $u<v$, then
\[ \lam=\al_1\cdots\al_u\in\Bp^{a,a},\ \ \om=\al_{u+1}\cdots\al_v\in\Bp^{a,b},\ \ \kap=\al_{v+1}\cdots\al_l\in\Bp^{b,b} \]
satisfy $\rho=\lam\om\kap$, and thus $\Bp^{a,b}=\{\rho,\lam\om,\om\kap,\om \}$ follows. By Lemma \ref{LemPositivePathCount} {\rm (5)}, we also have $|\Bp^{b,a}|=0$. This corresponds to the case {\rm (vi)}.

If $u>v$, then
\[ \tau=\al_1\cdots\al_v\in\Bp^{a,b},\ \ \ze=\al_{v+1}\cdots\al_u\in\Bp^{b,a},\ \ \om=\al_{u+1}\cdots\al_l\in\Bp^{a,b} \]
satisfy $\rho=\tau\ze\om$, and we have $\Bp^{a,b}\supseteq\{\rho,\om,\tau\}$, $\Bp^{a,a}=\{\tau\ze\}$, $\Bp^{b,b}=\{\ze\om\}$. 
In this case, $\{\al_1,\al_{u+1}\}$ is the set of outgoing arrows at $a$, and $\{\al_v,\al_l\}$ is the set of incoming arrows at $b$. Since there cannot exist a path which starts with $\al_{u+1}$ and ends with $\al_v$, it follows $\Bp^{a,b}=\{\rho,\om,\tau\}$. 
This corresponds to the case {\rm (vii)}.
\end{proof}

\begin{lem}\label{LemNonMaxVanish}
Assume that there is a maximal path $\rho\in e_aAe_b$ with $a\ne b$. Let $\wt{\rho}\in e_bVe_a$ be any lift of $\rho^{\dag}$. If $\Bp^{b,a}=\emptyset$, then the following holds for any path $\om\in e_aAe_b$.
\begin{enumerate}
\item If $\rho$ and $\om$ do not start with a common arrow, then $\wt{\rho}\om=0$.
\item If $\rho$ and $\om$ do not end with a common arrow, then $\om\wt{\rho}=0$.
\end{enumerate}
\end{lem}
\begin{proof}
We only show {\rm (1)}, in a similar way as in the proof of Lemma \ref{LemTildeVanish} {\rm (2)}. Put $\om=\be_1\be_2\cdots\be_r$ and $t(\be_1)=c$. By Condition \ref{CondMaxLength1}, we have $r\ge 2$ and $b\ne c$. Since $\pi(\wt{\rho}\be_1)=\rho^{\dag}\be_1=0$, we have $\wt{\rho}\be_1\in e_bAe_c$. Thus $\wt{\rho}\be_1$ is a $K$-linear sum of paths $\ze\in e_bAe_c$. Since $\Bp^{b,a}=\emptyset$, any such $\ze$ does not pass through $a$. In particular, $\ze$ does not end with $\be_1$. This shows $\ze\be_2=0$, which implies $\wt{\rho}\om=(\wt{\rho}\be_1)\be_2\cdots\be_r=0$.
\end{proof}

\begin{rem}\label{RemSandDeterm}
Let $\thh\in A$ be any path of positive length. Since $A$ is assumed to be gentle, there is no oriented cycle in $A$. This implies that the element $\thh\wt{\thh}\thh\in V$ is uniquely determined from $\thh$, independently from the choice of lift $\wt{\thh}$ of $\thh^{\dag}$.
\end{rem}

\begin{lem}\label{LemQuasiInverse}
For each maximal path $\rho\in A$, there exists $0\ne c_{\rho}\in K$ such that $\rho\wt{\rho}\rho=c_{\rho}\rho$ holds for any lift $\wt{\rho}\in V$ of $\rho^{\dag}$. In other words, $c_{\rho}\iv\wt{\rho}$ gives a quasi-inverse of $\rho$. By Remark \ref{RemSandDeterm}, this $c_{\rho}$ does not depend on the choice of lift $\wt{\rho}$.
\end{lem}
\begin{proof}
Put $s(\rho)=a$ and $t(\rho)=b$, for simplicity. We may assume $\wt{\rho}$ is in $e_bVe_a$. Take any quasi-inverse $\rho\dia\in e_bVe_a$ of $\rho$.

\medskip

\noindent\underline{Case 1}: $a\ne b$.

Assume $a\ne b$. 
As in Remark \ref{RemUseOften}, $K$-module $e_bVe_a$ has a $K$-basis
\[ \{\wt{\rho}\}\amalg\{\wt{\om}\mid\rho\ne\om\in\Bp^{a,b}\}\amalg\Bp^{b,a}, \]
where $\wt{\om}\in e_bVe_a$ is a lift of $\om^{\dag}$. Thus $\rho\dia$ can be written as
\begin{equation}\label{rhodiasum}
\rho\dia=x\wt{\rho}+\sum_{\rho\ne\om\in\Bp^{a,b}}y_{\om}\wt{\om}+\sum_{\ze\in\Bp^{b,a}}z_{\ze}\ze
\end{equation}
for some $x,y_{\om},z_{\ze}\in K$. Since $\rho$ is maximal, we have $\rho\ze\rho=0$ for any $\ze\in\Bp^{b,a}$.
Also if $\rho\ne\om\in\Bp^{a,b}$, then $\rho$ and $\om$ do not start with a common arrow or not end with a common arrow. Thus by Lemma \ref{LemTildeVanish}, we have $\rho\wt{\om}\rho=0$ for such $\om$. Therefore we obtain $\rho=\rho\rho\dia\rho=x\rho\wt{\rho}\rho$ from $(\ref{rhodiasum})$. This in particular means $x\ne0$, and $\rho\wt{\rho}\rho=x\iv\rho$.




\medskip

\noindent\underline{Case 2}: $a=b$.

Assume $a=b$. Let $f_a\in e_aVe_a$ be any lift of $e_a^{\dag}\in E$. Since $A$ is gentle, there does not exist a path $\thh$ of positive length satisfying $s(\thh)=t(\thh)=a$ other than $\rho$. This means that $\{e_a,\rho,f_a,\wt{\rho}\}$ is a $K$-basis of $e_aVe_a$, as in Remark \ref{RemUseOften}. Also, since $\pi(f_a\rho)=e_a^{\dag}\rho=0$, we have $f_a\rho=p\rho+qe_a$ for some $p,q\in K$. This implies $0=f_a\rho^2=p\rho^2+q\rho=q\rho$ and hence $q=0$, which means $f_a\rho=p\rho$. This shows $\rho f_a\rho=p\rho^2=0$. By using the above basis, we have
\[ \rho\dia=x\wt{\rho}+y\rho+ze_a+wf_a \]
for some $x,y,z,w\in K$. This gives $\rho=\rho\rho\dia\rho=x\rho\wt{\rho}\rho$, hence $\rho\wt{\rho}\rho=x\iv\rho$.

\end{proof}

\begin{prop}\label{PropQuasiInverse}
Assume $\Char(K)\ne 2$. For any maximal path $\rho\in e_aAe_b$, there exists a unique element $\rho\ust\in e_bVe_a$ which satisfies the following.
\begin{itemize}
\item[{\rm (q1)}] $\rho\ust$ is a quasi-inverse of $\rho$.
\item[{\rm (q2)}] $\pi(c_{\rho}\rho\ust)=\rho^{\dag}$ for $c_{\rho}\in K\setminus\{0\}$ obtained in Lemma \ref{LemQuasiInverse}.
\item[{\rm (q3)}] $\rho$ is a quasi-inverse of $\rho\ust$.
\item[{\rm (q4)}] $\rho\ust\rho\ust=0$ and $e_a=\rho\rho\ust+\rho\ust\rho$ hold if $a=b$.
\end{itemize}
\end{prop}
\begin{proof}
Let $\wt{\rho}\in e_bVe_a$ be any lift of $\rho^{\dag}$. By Lemma \ref{LemQuasiInverse}, there exists a unique $c_{\rho}\in K\setminus\{0\}$ with which $\rho\dia:=c_{\rho}\iv\wt{\rho}$ gives a quasi-inverse of $\rho$. 

\medskip

\noindent\underline{Case 1}: $a\ne b$.

Assume $a\ne b$. 
First we show the uniqueness. If both $\rho\ust$ and $\rho^{\star\prime}$ satisfy the stated properties, then their difference satisfies $\rho\ust-\rho^{\star\prime}\in e_bAe_a$ because of {\rm (q2)}, and thus can be written as a $K$-linear sum of paths in $A$ of positive length. Then we obtain $\rho^{\star\prime}=\rho^{\star\prime}\rho\rho^{\star\prime}=\rho\ust\rho\rho\ust=\rho\ust$ by {\rm (q3)}.

Let us show the existence.
Since $\rho\dia$ already satisfies {\rm (q1)} and {\rm (q2)}, it suffices to modify $\rho\dia$ to satisfy {\rm (q3)}.
As in Lemma \ref{LemMaxAtMost3}, we have seven cases {\rm (i),\ldots,(vii)} corresponding to the values of $\nfr=(|\Bp^{a,b}|,|\Bp^{b,a}|,|\Bp^{a,a}|,|\Bp^{b,b}|)$. We use the symbols given in the statement of Lemma \ref{LemMaxAtMost3}.

\medskip

{\rm (i)} In this case, $\{\rho\dia\}$ gives a $K$-basis of $e_bVe_a$. Thus $\rho\dia\rho\rho\dia=x\rho\dia$ holds for some $x\in K$. Multiplying $\rho$ from the both sides, we obtain
\[ \rho=\rho(\rho\dia\rho\rho\dia)\rho=x\rho\rho\dia\rho=x\rho, \]
hence $x=1$. Thus $\rho\ust=\rho\dia$ satisfies the desired properties.

\medskip

{\rm (ii)} In this case, since $e_bVe_a$ has a $K$-basis $\{\rho\dia,\ze\}$, we have $\rho\dia\rho\rho\dia=x\rho\dia+y\ze$ for some $x,y\in K$. Similarly as in {\rm (i)}, multiplying $\rho$ from the both sides, we obtain $x=1$. If we put $\rho\ust=\rho\dia+y\ze$, it satisfies the desired properties.

\medskip

{\rm (iii),(iv),(v)} These cases can be shown in a similar way as the following {\rm (vi)}, using Lemma \ref{LemNonMaxVanish}.

\medskip

{\rm (vi)} In this case, we have $\Bp^{b,a}=\emptyset$ and $\Bp^{a,b}=\{\rho,\lam\om,\om\kap,\om\}$. Since $\kap\wt{\rho},\wt{\rho}\lam,\kap\wt{\rho}\lam$ give lifts of $(\lam\om)^{\dag},(\om\kap)^{\dag},\om^{\dag}$ respectively, we have a $K$-basis $\{\rho\dia,\kap\wt{\rho},\wt{\rho}\lam,\kap\wt{\rho}\lam\}$ of $e_bVe_a$. Thus
\begin{equation}\label{RMO}
\rho\dia\rho\rho\dia=x\rho\dia+y\kap\wt{\rho}+z\wt{\rho}\lam+w\kap\wt{\rho}\lam
\end{equation}
holds for some $x,y,z,w\in K$. Multiplying $\rho$ to $(\ref{RMO})$ from the both sides, we obtain $x=1$. Then if we multiply $\rho$ from the left, we obtain
$\rho\rho\dia=\rho\rho\dia+z\rho\wt{\rho}\lam$, hence $z\rho\wt{\rho}\lam=0$.
Since $(\rho\wt{\rho}\lam)\om\kap=\rho\wt{\rho}\rho=c_{\rho}\rho\ne0$, we have $\rho\wt{\rho}\lam\ne0$, which shows $z=0$. Similarly, multiplying $\rho$ to $(\ref{RMO})$ from the right, we obtain $y=0$.

Now $(\ref{RMO})$ becomes $\rho\dia\rho\rho\dia=\rho\dia+w\kap\wt{\rho}\lam$. If we multiply $\om$ to this equation from the both sides, we obtain
\[ w\om\kap\wt{\rho}\lam\om=0 \]
by Lemma \ref{LemNonMaxVanish}. Since $\lam(\om\kap\wt{\rho}\lam\om)\kap=c_{\rho}\rho\ne0$, it follows $w=0$. Thus $(\ref{RMO})$ becomes $\rho\dia\rho\rho\dia=\rho\dia$, which means that $\rho\ust=\rho\dia$ satisfies the desired properties.

\medskip

{\rm (vii)} In this case, similarly we see that $\{\rho\dia,\wt{\rho}\tau\ze,\ze\om\wt{\rho},\ze\}$ gives a $K$-basis of $e_bVe_a$. Thus
\begin{equation}\label{LHO}
\rho\dia\rho\rho\dia=x\rho\dia+y\wt{\rho}\tau\ze+z\ze\om\wt{\rho}+w\ze
\end{equation}
holds for some $x,y,z,w\in K$. Similarly as {\rm (vi)}, multiplying $\rho$ to $(\ref{LHO})$ from the both sides/left/right, we obtain $x=1$ and $y=z=0$. If we put $\rho\ust=\rho\dia+w\ze$, then $\rho\ust$ satisfies the desired properties.

%

\medskip

\noindent\underline{Case 2}: $a=b$.

Assume $a=b$. We use the assumption of $\Char(K)\ne2$ only in this case. Let $f_a\in e_aVe_a$ be a lift of $e_a^{\dag}\in E$. As in the proof of Lemma \ref{LemQuasiInverse}, we have $f_a\rho=p\rho$ for some $p\in K$. Since $\{e_a,\rho,f_a,\rho\dia\}$ is a $K$-basis of $e_aVe_a$ by Remark \ref{RemUseOften},
\[ \rho\dia\rho\dia=x\rho\dia+y\rho+ze_a+wf_a \]
holds for some $x,y,z,w\in K$. If we put $\rho\ust=\rho\dia-\frac{1}{2}xe_a-(z+wp+\frac{1}{4}x^2)\rho$, it satisfies {\rm (q1),(q2)} and moreover $\rho\ust\rho\ust\rho=0$, since
\begin{eqnarray*}
\rho\ust\rho\ust\rho&=&\big(\rho\dia-\frac{1}{2}xe_a-(z+wp+\frac{1}{4}x^2)\rho\big)\big(\rho\dia-\frac{1}{2}xe_a-(z+wp+\frac{1}{4}x^2)\rho\big)\rho\\
&=&\big(\rho\dia\rho\dia-x\rho\dia+\frac{1}{4}x^2e_a-(z+wp+\frac{1}{4}x^2)\rho\rho\dia\big)\rho\\
&=&\big((x\rho\dia+y\rho+ze_a+wf_a)-x\rho\dia\big)\rho+\frac{1}{4}x^2\rho-(z+wp+\frac{1}{4}x^2)\rho\\
&=&(y\rho+ze_a+wf_a)\rho+\frac{1}{4}x^2\rho-(z+wp+\frac{1}{4}x^2)\rho\\
&=&(z\rho+wp\rho+\frac{1}{4}x^2\rho)-(z+wp+\frac{1}{4}x^2)\rho\ =\ 0.
\end{eqnarray*}
We need the following claim, to show {\rm (q3)} and {\rm (q4)}.
\begin{claim}\label{Claim_rhorho_basis}
If a maximal path $\rho$ satisfies $s(\rho)=t(\rho)=a$ as in Case 2, then $\{ \rho\rho\ust,\,\rho\ust\rho,\,\rho,\,\rho\ust \}$ is a $K$-basis of $e_aVe_a$ for the above $\rho\ust$. Moreover, we have
\begin{equation}\label{eq_cycle_rho}
e_a=\rho\rho\ust+\rho\ust\rho.
\end{equation}
\end{claim}
\begin{proof}
Since $\dim_K(e_aVe_a)=4$, to show that $\{ \rho\rho\ust,\,\rho\ust\rho,\,\rho,\,\rho\ust \}$ is a $K$-basis, it suffices to show that these elements are linearly independent over $K$. However this can be shown easily, using the equations
 $\rho(\rho\rho\ust)=(\rho\ust\rho)\rho=\rho^2=0$ and $\rho\rho\ust\rho=\rho\ne0$. 

In particular, the element $e_a\in e_aVe_a$ can be written by using this basis as
\[ e_a=q\rho\rho\ust+r\rho\ust\rho+s\rho+t\rho\ust, \]
for some $q,r,s,t\in K$. Multiplying $\rho$ from the both sides, we obtain $t=0$. Then, multiplying $\rho$ from the left, we obtain $r=1$. Similarly for $q=1$. Now we have $e_a=\rho\rho\ust+\rho\ust\rho+s\rho$. If we multiply $\rho\rho\ust$ from the left, then $\rho\ust\rho\ust\rho=0$ shows
\[ \rho\rho\ust=(\rho\rho\ust\rho)\rho\ust+\rho(\rho\ust\rho\ust\rho)+s\rho\rho\ust\rho=\rho\rho\ust+s\rho, \]
which means $s=0$. Thus we have $e_a=\rho\rho\ust+\rho\ust\rho$.
\end{proof}
By Claim \ref{Claim_rhorho_basis} and $\rho\ust\rho\ust\rho=0$, we obtain
\[ \rho\ust=\rho\ust(\rho\rho\ust+\rho\ust\rho)=\rho\ust\rho\rho\ust, \]
which shows {\rm (q3)}. Then multiplying $\rho\ust$ to $(\ref{eq_cycle_rho})$ from the right, we obtain $\rho\rho\ust\rho\ust=0$. Multiplying $\rho\ust$ to $(\ref{eq_cycle_rho})$ from the both sides, at last we obtain {\rm (q4)}.
Thus the existence of $\rho\ust$ in Case 2 is shown.

\smallskip

Let us show the uniqueness. Suppose that both $\rho\ust$ and $\rho^{\star\prime}$ satisfies the conditions. Then $\rho\ust-\rho^{\star\prime}\in e_aAe_a$ follows from {\rm (q2)}, and thus
\[ \rho\ust-\rho^{\star\prime}=m\rho+ne_a \]
holds for some $m,n\in K$. Then we have
\begin{eqnarray*}
\rho\ust&=&\rho\ust\rho\rho\ust\ =\ (\rho^{\star\prime}+m\rho+ne_a)\rho(\rho^{\star\prime}+m\rho+ne_a)\\
&=&\rho^{\star\prime}\rho\rho^{\star\prime}+n\rho^{\star\prime}\rho+n\rho\rho^{\star\prime}+n^2\rho\ =\ \rho^{\star\prime}+ne_a+n^2\rho
\end{eqnarray*}
by {\rm (q3)} and {\rm (q4)}.
Interchanging the roles of $\rho\ust$ and $\rho^{\star\prime}$, from $\rho^{\star\prime}-\rho\ust=-m\rho-ne_a$ we obtain
\begin{equation}\label{eq_mn}
\rho^{\prime\star}=\rho\ust-ne_a+(-n)^2\rho.
\end{equation}
These two equations show $2n^2\rho=0$, and thus we have $n=0$ since $\Char(K)\ne2$. Thus $(\ref{eq_mn})$ implies $\rho\ust=\rho^{\prime\star}$.
\end{proof}

\medskip

In the rest of this section, we assume $\Char(K)\ne2$.
\begin{dfn}\label{DefForBasis}
For each maximal path $\rho\in A$, we have a unique quasi-inverse $\rho\ust\in e_bVe_a$ satisfying the properties in Proposition \ref{PropQuasiInverse}. Using them, we define as follows.
\begin{enumerate}
\item Let $a\in Q_0$ be any vertex. If a pair of paths $\mu,\nu$ (possibly $\mu=e_a$ or $\nu=e_a$, but not at the same time) satisfies $t(\mu)=a=s(\nu)$ and if $\mu\nu$ is maximal, we denote this pair by $\langle\mu,\nu\rangle_a$ or simply by $\langle\mu,\nu\rangle$. For any such pair $\mn_a=\mn$, define an idempotent element $\emn\in e_aVe_a$ by
\[ \emn=\nu(\mu\nu)\ust\mu. \]
Since $\pi(\emn)\ne0$, in particular we have $\emn\ne0$.
We denote the set of such pairs by
\[ E_a=\{\mn_a\mid t(\mu)=s(\nu)=a,\ \mu\nu\in\Mfr_A\}. \]
By definition, $\mn=\mnp$ means that $\mu=\mu\ppr$ and $\nu=\nu\ppr$ hold. Otherwise we write $\mn\ne\mnp$.
\item Let $\xi\in e_aAe_b$ be any path of positive length. Take unique paths $\lam,\kappa\in A$ which makes $\lam\xi\kappa$ maximal (possibly $\lam=e_a$ or $\kappa=e_b$), and define $\xi\ust\in e_bVe_a$ by
\[ \xi\ust=\kappa(\lam\xi\kappa)\ust\lam. \]
\end{enumerate}
\end{dfn}

\begin{rem}\label{Rem_DefForBasis}
The following is immediate from the definition.
\begin{enumerate}
\item For any $\mn_a\in E_a$, we have
\begin{equation}\label{rem_e1}
\mu\emn\nu=\mu\nu(\mu\nu)\ust\mu\nu=\mu\nu.
\end{equation}
\item Let $\xi\in e_aAe_b$ be any path of positive length, and let $\lam,\kappa$ be paths such that $\lam\xi\kappa$ is maximal. We have
\begin{equation}\label{rem_e2}
\xi\xi\ust=\ep_{\langle\lam,\xi\kappa\rangle_a}\quad\text{and}\quad \xi\ust\xi=\ep_{\langle\lam\xi,\kappa\rangle_b}
\end{equation}
by definition.
\item By {\rm (2)}, any $\mn_a$ satisfies
\[ \emn=\begin{cases}
\mu\ust\mu&\text{if}\ \mu\ne e_a,\\
\nu\nu\ust&\text{if}\ \nu\ne e_a.
\end{cases} \]
\end{enumerate}
\end{rem}

\begin{rem}\label{Rem2munu}
For each $a\in Q_0$, we have $a\in\Tfr_A$ if and only if $|E_a|=1$.
Moreover, the following {\rm (1)} and {\rm (2)} are equivalent.
\begin{enumerate}
\item There exist $\mn_a\ne\mnp_a$ in $E_a$ satisfying $\mu\nu=\mu\ppr\nu\ppr$.
\item There exists a non-maximal path $\thh\in e_aAe_a$ of positive length, which is necessarily unique and satisfies $\thh^2=0$.
\end{enumerate}
Indeed, in this case $\mn,\mnp\in E_a$ are related by
\[ \mu\ppr\thh=\mu\quad\text{and}\quad \nu\ppr=\thh\nu \]
or
\[ \mu\thh=\mu\ppr\quad\text{and}\quad \nu=\thh\nu\ppr. \]
\end{rem}

\begin{lem}\label{Lem6_1}
Let $\al\in Q_1$ be any arrow. Let $\rho\in e_aAe_b$ be any maximal path which does not start with $\al$. Then we have $\rho\ust\al=0$ in $V$.
\end{lem}
\begin{proof}
Since $\rho\ust\in e_bVe_a$, this is obvious if $s(\al)\ne a$. Thus we may assume $s(\al)=a$. Put $t(\al)=c$. By Condition \ref{CondMaxLength1}, it satisfies $b\ne c$. Since $\pi(\rho\ust\al)=c_{\rho}\iv\rho^{\dag}\al=0$, we have $\rho\ust\al\in e_bAe_c$. By $b\ne c$, any path $\ze\in e_bAe_c$ should satisfy $\rho\ze=0$, because of the maximality of $\rho$. This shows $\rho\rho\ust\al=0$, which implies
$\rho\ust\al=(\rho\ust\rho\rho\ust)\al=\rho\ust(\rho\rho\ust\al)=0$.
\end{proof}

\begin{lem}\label{LemMaxOrtho}
Let $a\in Q_0$ be any vertex. For any $\mn_a\ne\mnp_a$ in $E_a$, the associated idempotents $\emn,\emnp\in e_aVe_a$ are orthogonal to each other.
\end{lem}
\begin{proof}
Put $\rho=\mu\nu$ and $\rho\ppr=\mu\ppr\nu\ppr$ for simplicity.
By symmetry, it suffices to show $\emn\emnp=0$. There are the following 4 cases.
\[ \mathrm{(i)}\ \mu\ne e_a\ \text{and}\ \nu\ppr\ne e_a,%
\ \ \mathrm{(ii)}\ \mu=e_a\ \text{and}\ \nu\ppr\ne e_a,%
\ \ \mathrm{(iii)}\ \mu\ne e_a\ \text{and}\ \nu\ppr=e_a,%
\ \ \mathrm{(iv)}\ \mu=\nu\ppr=e_a. \]
Let us show $\emn\emnp=0$ in each of the above cases.

\medskip

{\rm (i)} In this case, $\emn\emnp=\nu\rho\ust\mu\nu\ppr\rho^{\prime\star}\nu^{\prime}=0$ immediately follows from $\mu\nu\ppr=0$. 

\medskip

{\rm (ii)} In this case, $\nu=\rho$ is maximal. Let $\al\in Q_1$ be the arrow with which $\nu\ppr$ starts. Then $\rho$ does not start with $\al$. By Lemma \ref{Lem6_1} we obtain $\rho\ust\al=0$, which implies
$\emn\emnp=\rho(\rho\ust\nu\ppr)\rho^{\prime\star}\mu\ppr=0$.

%

\medskip

{\rm (iii)} This case can be shown in a similar way as {\rm (ii)}, using the fact that $\mu$ and $\mu\ppr$ do not end with a common arrow.

\medskip

{\rm (iv)} In this case, both $\nu=\rho$ and $\mu\ppr=\rho\ppr$ are maximal. We have $\emn\emnp=\rho\rho\ust\rho^{\prime\star}\rho\ppr$ by definition. Put $s(\rho\ppr)=b$ and $t(\rho)=c$ for simplicity.
If $b\ne c$, we see that
\begin{itemize}
\item any path $\thh\in e_cAe_b$ satisfies $\rho\thh=0$ and $\thh\rho\ppr=0$,
\item for any path $\om\in e_bAe_c$, it does not start with a common arrow with $\rho\ppr$ or does not end with a common arrow with $\rho$, which means that either $\wt{\om}\rho\ppr=0$ or $\rho\wt{\om}=0$ holds by Lemma \ref{LemTildeVanish}.
\end{itemize}
Since $\rho\ust\rho^{\prime\star}$ is a $K$-linear combination of such $\thh$ and $\wt{\om}$, we have $\rho(\rho\ust\rho^{\prime\star})\rho\ppr=0$.

If $b=c$, since $\rho,\rho\ppr$ are maximal we have $\rho\rho\ppr=\rho\ppr\rho=0$. If moreover $a=b$, we have $\rho=\rho\ppr$ since $|\Bp^{a,a}|\le 1$, and the equality $\emn\emnp=0$ follows from $\rho\ust\rho\ust=0$. Suppose $a\ne b$. In this case we have $\rho\ne\rho\ppr$, and
\begin{itemize}
\item $\rho$ is the unique path in $e_aAe_b$,
\item $\rho\ppr$ is the unique path in $e_bAe_a$,
\item $\Bp^{a,a}=\Bp^{b,b}=\emptyset$.
\end{itemize}
Let $f_b\in e_bVe_b$ be a lift of $e_b^{\dag}$. Since $\pi(\rho f_b)=0$, we have $\rho f_b\in e_aAe_b$. Hence
\begin{equation}\label{Eq_rhofxrho}
\rho f_b=x\rho
\end{equation}
holds for some $x\in K$. Since $e_bAe_b=Ke_b$ and $e_bEe_b=Ke_b^{\dag}$, we have a $K$-basis $\{ e_b,f_b\}$ of $e_bVe_b$ by Remark \ref{RemUseOften}. Thus $\rho\ust\rho^{\prime\star}=ye_b+zf_b$ holds for some $y,z\in K$. Then $(\ref{Eq_rhofxrho})$ shows
\[ \emn\emnp=\rho(\rho\ust\rho^{\prime\star})\rho\ppr=\rho(ye_b+zf_b)\rho\ppr=(y+xz)\rho\rho\ppr=0. \]
\end{proof}

\begin{lem}\label{LemIdempBasis}
Let $a\in Q_0$ be any vertex. Suppose there exist two $\mn\ne\mnp$ in $E_a$.
The following holds.
\begin{enumerate}
\item $\emn,\emnp\in e_aVe_a$ are linearly independent over $K$.
\item If they moreover satisfies
\begin{equation}\label{Eq_exeye}
e_a=x\emn+y\emnp
\end{equation}
for some $x,y\in K$, then it follows $x=y=1$.
\end{enumerate}
In particular if $\dim_K(e_aVe_a)=2$, then $\{\emn,\emnp\}$ gives a $K$-basis.
\end{lem}
\begin{proof}
By Lemma \ref{LemMaxOrtho}, these $\emn,\emnp$ are orthogonal idempotents. Both {\rm (1)} and {\rm (2)} are immediate from this fact. 
\end{proof}

\begin{prop}\label{PropIdempBasis}
Let $a\in Q_0$ be any vertex. The following gives a $K$-basis of $e_aVe_a$.
\begin{enumerate}
\item If there is only one $\mn_a$ in $E_a$, then
$\{ \emn,\ e_a-\emn\}$
gives a $K$-basis consisting of orthogonal idempotents.
\item If there are $\mn_a\ne\mnp_a$ in $E_a$ and if $\mu\nu\ne\mu\ppr\nu\ppr$, then
$\{ \emn,\ \emnp\}$
gives a $K$-basis consisting of orthogonal idempotents, which satisfies $e_a=\emn+\emnp$.
\item Suppose there are $\mn_a\ne\mnp_a$ in $E_a$ satisfying $\mu\nu=\mu\ppr\nu\ppr$. By symmetry we may assume there is a unique non-maximal path $\thh\in e_aAe_a$ satisfying $\thh^2=0$ such that $\mu=\mu\ppr\thh$ and $\nu\ppr=\thh\nu$ hold, as in Remark \ref{Rem2munu}. In this case,
$\{ \emn,\emnp,\thh,\thh\ust \}$
gives a $K$-basis of $e_aVe_a$, which satisfies the following conditions.
\begin{itemize}
\item[{\rm (i)}] $\emn=\thh\ust\thh$ and $\emnp=\thh\thh\ust$ hold, as seen in Remark \ref{Rem2munu}.
\item[{\rm (ii)}] $e_a=\emn+\emnp$.
\item[{\rm (iii)}] $\thh\thh\ust\thh\ne0$. (The equality $\thh\thh\ust\thh=\thh$ for any (not necessarily maximal) path $\thh$ will be shown later, in Lemma \ref{LemForThm1}.)
\end{itemize}
\end{enumerate}
\end{prop}
\begin{proof}
Let $f_a\in e_aVe_a$ be any lift of $e_a^{\dag}$. In the cases  {\rm (1)} and {\rm (2)}, there is no path $\thh\in e_aAe_a$ of positive length, as seen in Remark \ref{Rem2munu}. Thus $\{e_a,f_a\}$ gives a $K$-basis of $e_aVe_a$. In particular we have $\dim_K(e_aVe_a)=2$, hence {\rm (2)} follows immediately from Lemma \ref{LemIdempBasis}.

\medskip

{\rm (1)} We have
\begin{equation}\label{Eq_exeyf}
\emn=xe_a+yf_a
\end{equation}
for some $x,y\in K$. Since $0\ne\pi(\emn)=ye_a^{\dag}$, it satisfies $y\ne0$. Thus $\emn$ and $e_a$ are linearly independent, which shows that $\{ \emn,e_a-\emn\}$ gives a $K$-basis. 

\bigskip

{\rm (3)} In this case, we have a $K$-basis $\{ e_a,\thh,f_a,\thh\ust\}$ of $e_aVe_a$. {\rm (i)} is already seen in Remark \ref{Rem2munu}. {\rm (iii)} follows from
\[ \mu\ppr(\thh\thh\ust\thh)\nu=\mu\emn\nu=\mu\nu\ne0. \]
Thus {\rm (i)} and {\rm (iii)} hold. From these, we have
$\emn\thh=\thh\emnp=\thh^2=0$ and $\thh\emn=\emnp\thh=\thh\thh\ust\thh\ne0$. This implies that $\{\emn,\emnp,\thh,\thh\ust\}$ is linearly independent over $K$, and thus forms a $K$-basis of $e_aVe_a$.

Let us show {\rm (ii)}. This is a slight generalization of Claim \ref{Claim_rhorho_basis}. Since $\{\emn,\emnp,\thh,\thh\ust\}$ is a $K$-basis, in particular
\[ e_a=x\emn+y\emnp+z\thh+w\thh\ust \]
holds for some $x,y,z,w\in K$. This implies
\[ 0=\thh^2=\thh(x\emn+y\emnp+z\thh+w\thh\ust)\thh=w\thh\thh\ust\thh, \]
hence $w=0$. 
Then $e_a=x\emn+y\emnp+z\thh$ holds. From
\[ e_a-2z\thh=(e_a-z\thh)^2=(x\emn+y\emnp)^2=x^2\emn+y^2\emnp, \]
we also have
\[ e_a=x^2\emn+y^2\emnp+2z\thh. \]
Comparing the coefficients we obtain $z=2z$, which means $z=0$.
Thus we have $e_a=x\emn+y\emnp$, and Lemma \ref{LemIdempBasis} {\rm (2)} shows {\rm (ii)}.
\end{proof}

\begin{dfn}\label{DefUpa}
Accordingly to the cases given in Proposition \ref{PropIdempBasis}, we define as follows, for each vertex $a\in Q_0$.
\begin{enumerate}
\item If there is only one $\mn_a$ in $E_a$, put $B_a=\{ \emn,\ e_a-\emn\}$. Also, put $e_a-\emn=\ifr_a$.
\item If there are $\mn_a\ne\mnp_a$ in $E_a$, put $B_a=\{ \emn,\emnp\}$.
\end{enumerate}
They are orthogonal idempotents.
\end{dfn}

\begin{rem}\label{RemUandUp}
Remark that if $a\ne b$ in $Q_0$, then $e_aVe_b$ has a $K$-basis
\[ \{\thh\in V\mid \thh\in e_aAe_b\ \text{is a path}\} \amalg \{\ze\ust\in V\mid \ze\in e_bAe_a\ \text{is a path}\}. \]
Thus by Proposition \ref{PropIdempBasis}, we have the following $K$-basis of the whole $V$.
\begin{equation}\label{BasisForV}
(\underset{a\in Q_0}{\amalg}B_a)
\amalg \{ \thh\in V\mid \thh\in A\ \text{is a path},\, l(\thh)>0 \}
\amalg \{ \thh\ust\in V\mid \thh\in A\ \text{is a path},\, l(\thh)>0 \}.
\end{equation}
\end{rem}

We prepare some lemmas to prove Theorem \ref{ThmCharacterizeVA}.
\begin{lem}\label{LemForThm1}
For any paths $\thh$ and $\tau$ in $A$ of positive lengths, the following holds in $V$.
\begin{enumerate}
\item If $\thh\tau\ne0$, then $\tau\ust\thh\ust=(\thh\tau)\ust$. Otherwise $\tau\ust\thh\ust=0$.
\item $\thh\ust\thh\thh\ust=\thh\ust$.
\item $\thh\thh\ust\thh=\thh$.
\item If $\thh\tau\ne0$, then $\tau\ust\thh\ust\thh=\tau\ust$ and $\tau\tau\ust\thh\ust=\thh\ust$.
\end{enumerate}
\end{lem}
\begin{proof}
{\rm (1)} If $\thh\tau\ne0$, there exist paths $\lam$ and $\kappa$ with which $\lam\thh\tau\kappa=\rho$ is maximal. By definition, we have
\[ (\thh\tau)\ust=\kappa\rho\ust\lam,\ \ \thh\ust=\tau\kappa\rho\ust\lam,\ \ \tau\ust=\kappa\rho\ust\lam\thh. \]
Thus we obtain $\tau\ust\thh\ust=(\kappa\rho\ust\lam\thh)(\tau\kappa\rho\ust\lam)=\kappa\rho\ust\rho\rho\ust\lam=\kappa\rho\ust\lam=(\thh\tau)\ust$.

Suppose $\thh\tau=0$. If $t(\thh)\ne s(\tau)$, then $\tau\ust\thh\ust=0$ is obvious. We may assume $t(\thh)=s(\tau)=a$. Take paths $\lam,\kappa,\xi,\ze$ so that $\lam\thh\kappa=\rho$ and $\xi\tau\ze=\om$ are maximal. By definition, we have $\thh\ust=\kappa\rho\ust\lam$ and $\tau\ust=\ze\om\ust\xi$. By assumption we have $\langle\lam\thh,\kappa\rangle_a\ne\langle\xi,\tau\ze\rangle_a$, and thus
\[ e_a=\ep_{\langle\lam\thh,\kappa\rangle}+\ep_{\langle\xi,\tau\ze\rangle}=\kappa\rho\ust\lam\thh+\tau\ze\om\ust\xi \]
follows from Proposition \ref{PropIdempBasis}. Multiplying $\om\ust\xi$ from the left and $\kappa\rho\ust$ from the right, we obtain
\begin{eqnarray*}
\om\ust\xi\kappa\rho\ust&=&\om\ust\xi(\kappa\rho\ust\lam\thh+\tau\ze\om\ust\xi)\kappa\rho\ust\\
&=&\om\ust\xi\kappa\rho\ust\rho\rho\ust+\om\ust\om\om\ust\xi\kappa\rho\ust\ =\ 2\om\ust\xi\kappa\rho\ust,
\end{eqnarray*}
which shows $\om\ust\xi\kappa\rho\ust=0$. Thus $\tau\ust\thh\ust=\ze(\om\ust\xi\kappa\rho\ust)\lam=0$ follows.

{\rm (2)} Take paths $\lam$ and $\kappa$ so that $\rho=\lam\thh\kappa$ becomes maximal. Then we have
\[ \thh\ust\thh\thh\ust=(\kappa\rho\ust\lam)\thh(\kappa\rho\ust\lam)=\kappa\rho\ust\rho\rho\ust\lam=\kappa\rho\ust\lam=\thh\ust. \]

{\rm (3)} Take paths $\lam$ and $\kappa$ so that $\rho=\lam\thh\kappa$ becomes maximal. If $\thh=\al\thh\ppr$ for some arrow $\al$, then we have
$\thh\thh\ust\thh=\thh\kappa\rho\ust\lam\thh=(\al\al\ust\al)\thh\ppr$.
Thus it suffices to show $\al\al\ust\al=\al$ for any arrow $\al$. Put $s(\al)=a$ and $t(\al)=b$.

If there is an arrow $\gam$ with $t(\gam)=a$ and $\gam\al=0$, then we have $e_a=\al\al\ust+\gam\ust\gam$ by Proposition \ref{PropIdempBasis}, and thus $\al=\al\al\ust\al+\gam\ust\gam\al=\al\al\ust\al$ follows. Similarly, if there is an arrow $\upsilon$ with $s(\upsilon)=b$ and $\al\upsilon=0$, we obtain $\al\al\ust\al=\al$.
Thus we may assume that
\begin{itemize}
\item[{\rm (i)}] there is no arrow $\gam$ with $t(\gam)=a$ and $\gam\al=0$,
\item[{\rm (ii)}]  there is no arrow $\upsilon$ with $s(\upsilon)=b$ and $\al\upsilon=0$.
\end{itemize}
Suppose that there is an arrow $\be\ne\al$ with $s(\be)=a$. Then by {\rm (i)}, there is a maximal path $\om\in e_aAe_c$ starting with $\be$. In this case we have $e_a=\om\om\ust+\al\al\ust$ by Proposition \ref{PropIdempBasis}. By Condition \ref{CondMaxLength1}, we have $b\ne c$. Since $\pi(\om\ust\al)=c_{\om}\om^{\dag}\al=0$, we have $\om\ust\al\in A$. By $b\ne c$ and the maximality of $\om$, we obtain $\om\om\ust\al=0$, which shows $\al\al\ust\al=\al$. Thus we may also assume that
\begin{itemize}
\item[{\rm (iii)}] there is no arrow $\be$ with $s(\be)=a$ other than $\al$.
\end{itemize}
By {\rm (i)} and {\rm (ii)}, there is no path in $e_bAe_a$. By {\rm (iii)}, there is no path in $e_aAe_b$ either, other than $\al$. Then, since $\al\al\ust\al\in e_aVe_b$, we have $\al\al\ust\al=x\al$ for some $x\in K$. Multiplying $\al\ust$ from the left, we obtain $\al\ust\al=(\al\ust\al\al\ust)\al=x\al\ust\al$ by {\rm (2)}, which means $x=1$.

{\rm (4)} Take paths $\lam$ and $\kappa$ so that $\lam\thh\tau\kappa=\rho$ becomes maximal. Then we have $\tau\ust\thh\ust\thh=(\thh\tau)\ust\thh=(\kappa\rho\ust\lam)\thh=\tau\ust$. Similarly for $\tau\tau\ust\thh\ust=\thh\ust$.
\end{proof}

\begin{lem}\label{LemForThm2}
Let $\thh,\tau\in A$ be any pair of paths of positive lengths. The following holds.
\begin{enumerate}
\item If $\thh\ust\tau\ne0$, then $\thh$ and $\tau$ start with a common arrow.
\item Conversely if $\thh$ and $\tau$ start with a common arrow, there is a (unique) maximal path $\rho=\al_1\cdots\al_l\in A$ and $1\le u\le l$ with which
\[ \thh=\al_u\cdots\al_v\quad\text{and}\quad \tau=\al_u\cdots\al_w \]
hold for some $u\le v,w\le l$. With this expression, we have
\begin{equation}\label{Eq_thhtau}
\thh\ust\tau=\begin{cases}
(\al_{w+1}\cdots\al_v)\ust&\ \text{if}\ v>w,\\
\ep_{\langle\al_1\cdots\al_v,\al_{v+1}\cdots\al_l\rangle}&\ \text{if}\ v=w,\\
\al_{v+1}\cdots\al_w&\ \text{if}\ v<w.
\end{cases}
\end{equation}
In particular we have $\thh\ust\tau\ne0$.
\end{enumerate}
\end{lem}
\begin{proof}
{\rm (1)} Suppose that $\thh$ and $\tau$ start with different arrows, and let us show $\thh\ust\tau=0$. Take paths $\lam,\kap\in A$ so that $\rho=\lam\thh\kap$ is maximal. By assumption, we have $\lam\tau=0$ in $A$. By definition, we have
$\thh\ust\tau=\kap\rho\ust\lam\tau$.

If $\lam\ne e_a$, then $\thh\ust\tau=0$ follows from $\lam\tau=0$. If $\lam=e_a$, then Lemma \ref{Lem6_1} shows $\rho\ust\tau=0$. Thus $\thh\ust\tau=\kap\rho\ust\tau=0$ follows also in this case.

{\rm (2)} Existence of $\rho,u,v,w$ with the mentioned properties, is obvious. Let us show $(\ref{Eq_thhtau})$. The last assertion $\thh\ust\tau\ne0$ is an immediate consequence.

If $v>w$, then
\[ \thh\ust\tau=(\tau(\al_{w+1}\cdots\al_v))\ust\tau=(\al_{w+1}\cdots\al_v)\ust\tau\ust\tau=(\al_{w+1}\cdots\al_v)\ust \]
follows from Lemma \ref{LemForThm1} {\rm (1),(4)}.
If $v=w$, then we have
\[ \thh\ust\tau=\thh\ust\thh=(\al_{v+1}\cdots\al_l)\rho\ust(\al_1\cdots\al_v)=\ep_{\langle\al_1\cdots\al_v,\al_{v+1}\cdots\al_l\rangle}. \]
If $v<w$, then for $\tau\ppr=\al_{v+1}\cdots\al_w$, we have
\begin{eqnarray*}
\thh\ust\tau&=&(\al_{v+1}\cdots\al_l)\rho\ust(\al_1\cdots\al_{u-1})(\al_u\cdots\al_v)(\al_{v+1}\cdots\al_w)\\
&=&\tau\ppr(\al_{w+1}\cdots\al_l)\rho\ust(\al_1\cdots\al_v)\tau\ppr\ =\ \tau\ppr\tau^{\prime\star}\tau\ppr\ =\ \tau\ppr
\end{eqnarray*}
by Lemma \ref{LemForThm1} {\rm (3)}.
\end{proof}

\begin{dfn}\label{Defvarphi}
Define a $K$-linear map $\vp\co V\to V(A)$, by extending $K$-linearly the following map from the basis $(\ref{BasisForV})$ of $V$ to the basis $(\ref{BasisForVA})$ of $V(A)$ in Remark \ref{RemBasisVA}.
\begin{enumerate}
\item For any path $\thh\in A$ of positive length, we may find a unique maximal path $\rho=\al_1\cdots\al_l\in A$ and $1\le u\le v\le l$ satisfying $\thh=\al_u\cdots\al_v$. Put
\[ \vp(\thh)=\rho_{u-1, v}=\eta_A^{\rho}(\thh)\in\Mcal_{\rho}\se V(A) \]
and put
\[ \vp(\thh\ust)=\rho_{v, u-1}\in\Mcal_{\rho}\se V(A). \]
\item Let $a\in Q_0$ be any vertex.
\begin{itemize}
\item[{\rm (i)}] For any $\mn_a\in E_a$, by definition there uniquely exist a maximal path $\rho=\al_1\cdots\al_l$ and $0\le u\le l$ such that
$\mu=\al_1\cdots\al_u$ and $\nu=\al_{u+1}\cdots\al_l$ hold. Put
\[ \vp(\emn)=\rho_{u,u}=\ep_u^{\rho}\in\Mcal_{\rho}\se V(A). \]
\item[{\rm (ii)}] If $|E_a|=1$, equivalently if $a\in\Tfr_A$, put
\[ \vp(\ifr_a)=\efr_a\in K\efr_a\se V(A). \]
\end{itemize}
\end{enumerate}
\end{dfn}

\begin{lem}\label{Lemphieta}
$\vp\ci\eta\ppr=\eta_A$ holds.
\end{lem}
\begin{proof}
For any path $\thh\in A$ of positive length, it should be a subpath of a unique maximal path $\rho\in A$, and $\vp\ci\eta\ppr(\thh)=\eta_A^{\rho}(\thh)\in\Mcal_{\rho}\se V(A)$ by definition. Since the other components of $\eta_A(\thh)$ are
\[ \eta_A^{\pi}(\thh)=0\ (\rho\ne\pi\in\Mfr_A)\ \ \text{and}\ \ \eta_A^a(\thh)=0\ (a\in\Tfr_A), \] 
the above equality means that $\vp\ci\eta\ppr(\thh)=\eta_A(\thh)$ holds in $V(A)$.

It remains to show $\vp\ci\eta\ppr(e_a)=\eta_A(e_a)$ for $a\in Q_0$. By Proposition \ref{PropIdempBasis} and the definition of $\eta\ppr,\eta_A$, the following holds.
\begin{enumerate}
\item If $|E_a|=1$, equivalently if $a\in\Tfr_A$, then we have
\[ \vp\ci\eta\ppr(e_a)=\vp(\emn+\ifr_a)=\ep_u^{\rho}+\efr_a=\eta_A(e_a), \]
where $E_a=\{\mn_a \}$ and
\[\rho=\al_1\cdots\al_l,\ \ \mu=\al_1\cdots\al_u,\ \ \nu=\al_{u+1}\cdots\al_l. \]
\item If there are $\mn_a\ne\mnp_a$ in $E_a$ with $\mu\nu=\rho\ne\rho\ppr=\mu\ppr\nu\ppr$, then we have
\[ \vp\ci\eta\ppr(e_a)=\vp(\emn+\emnp)=\ep_u^{\rho}+\ep_{u\ppr}^{\rho\ppr}=\eta_A(e_a), \]
where
\begin{eqnarray*}
&\rho=\al_1\cdots\al_l,\ \ \mu=\al_1\cdots\al_u,\ \ \nu=\al_{u+1}\cdots\al_l,&\\
&\rho\ppr=\al\ppr_1\cdots\al\ppr_{l\ppr},\ \ \mu\ppr=\al\ppr_1\cdots\al\ppr_{u\ppr},\ \ \nu\ppr=\al\ppr_{u\ppr+1}\cdots\al\ppr_{l\ppr}.&
\end{eqnarray*}
\item If there are $\mn_a\ne\mnp_a$ in $E_a$ with $\mu\nu=\rho=\mu\ppr\nu\ppr$, then we have
\[ \vp\ci\eta\ppr(e_a)=\vp(\emn+\emnp)=\ep_u^{\rho}+\ep_{v}^{\rho}=\eta_A(e_a), \]
where
\begin{eqnarray*}
&\rho=\al_1\cdots\al_l,\ \ \mu=\al_1\cdots\al_u,\ \ \nu=\al_{u+1}\cdots\al_l,&\\
&u\ne v,\ \ \mu\ppr=\al_1\cdots\al_{v},\ \ \nu\ppr=\al_{v+1}\cdots\al_l.&
\end{eqnarray*}
\end{enumerate}
Thus in either case, we have $\vp\ci\eta\ppr(e_a)=\eta_A(e_a)$ for any $a\in Q_0$.
\end{proof}

Now we are ready to prove the theorem.
\begin{proof}[Proof of Theorem \ref{ThmCharacterizeVA}]
It can be easily checked that $\vp$ sends elements in $(\ref{BasisForV})$ to elements in $(\ref{BasisForVA})$ surjectively. Since $\dim_KV=2\dim_KA=\dim_K V(A)$, this means that $\vp\co V\to V(A)$ is an isomorphism of $K$-modules. Moreover, by Lemma \ref{Lemphieta}, it satisfies $\vp\ci\eta\ppr=\eta_A$. It remains to show that $\vp$ respects the multiplication for these basis elements.

First we remark that for any $a\in Q_0$ and any $\mn_a\in E_a$, by definition we have $\emn=\nu(\mu\nu)\ust\mu$. If we put $\mu\nu=\rho=\al_1\cdots\al_l$ and take $0\le u\le l$ which gives $\mu=\al_1\cdots\al_u$ and $\nu=\al_{u+1}\cdots\al_l$, then we have
\[ \vp(\ep_{\mn})=\ep_u^{\rho}=\rho_{u,u}=\rho_{u,l}\rho_{l,0}\rho_{0,u}=\vp(\nu)\vp((\mu\nu)\ust)\vp(\mu), \]
which shows that the definition of $\vp$ is compatible with the equation $\emn=\nu(\mu\nu)\ust\mu$, with respect to the multiplication.
Since the subset
\[ \{\ifr_a\mid a\in \Tfr_A\}\amalg\{\thh\mid \thh\in A\ \text{is a path},\, l(\thh)>0\}\amalg\{\thh\ust\mid \thh\in A\ \text{is a path},\, l(\thh)>0\} \]
of $(\ref{BasisForV})$ generates $V$ as a $K$-algebra, it suffices to show the following {\rm (1),\ldots,(9)} for any vertices $a,b\in\Tfr_A$ and any path $\thh,\tau\in A$ of positive lengths.
\[ \begin{array}{lll}
\mathrm{(1)}\ \ \vp(\ifr_a\ifr_b)=\vp(\ifr_a)\vp(\ifr_b)
&\mathrm{(2)}\ \ \vp(\thh\ifr_b)=\vp(\thh)\vp(\ifr_b)
&\mathrm{(3)}\ \ \vp(\thh\ust\ifr_b)=\vp(\thh\ust)\vp(\ifr_b)\\
\mathrm{(4)}\ \ \vp(\ifr_a\tau)=\vp(\ifr_a)\vp(\tau)
&\mathrm{(5)}\ \ \vp(\thh\tau)=\vp(\thh)\vp(\tau)
&\mathrm{(6)}\ \ \vp(\thh\ust\tau)=\vp(\thh\ust)\vp(\tau)\\
\mathrm{(7)}\ \ \vp(\ifr_a\tau\ust)=\vp(\ifr_a)\vp(\tau\ust)
&\mathrm{(8)}\ \ \vp(\thh\tau\ust)=\vp(\thh)\vp(\tau\ust)
&\mathrm{(9)}\ \ \vp(\thh\ust\tau\ust)=\vp(\thh\ust)\vp(\tau\ust)
\end{array} \]

\smallskip

{\rm (1)} follows from
\[ \ifr_a\ifr_b=\begin{cases}
\ifr_a&\ \text{if}\ a=b\\
0&\ \text{otherwise}
\end{cases}
\quad, \quad
\efr_a\efr_b=\begin{cases}
\efr_a&\ \text{if}\ a=b\\
0&\ \text{otherwise}
\end{cases}
\]
and the definition $\vp(\ifr_a)=\efr_a$, $\vp(\ifr_b)=\efr_b$.

{\rm (2)} follows from $\thh\ifr_b=0$ and $\vp(\thh)\vp(\ifr_b)=\vp(\thh)\efr_b=0$. Similarly for {\rm (3),(4),(7)}.
Also, {\rm (5)} follows from Lemma \ref{Lemphieta}.

To show the other equalities, take maximal paths $\rho=\al_1\cdots\al_l,\om=\be_1\cdots\be_m\in A$ with which $\thh=\al_u\cdots\al_v$ and $\tau=\be_y\cdots\be_w$ holds for some $1\le u\le v\le l$ and $1\le y\le w\le m$.
Let us show {\rm (6)}. We use Lemma \ref{LemForThm2}.

If $\thh\ust\tau\ne0$, then $\thh$ and $\tau$ start with a common arrow, which implies $\rho=\om$, $u=y$ and $\tau=\al_u\cdots\al_w$. Thus we have
\[ \vp(\thh\ust)\vp(\tau)=\rho_{v,u-1}\rho_{u-1,v}=\rho_{v,w}=\vp(\thh\ust\tau) \]
by $(\ref{Eq_thhtau})$.
If $\thh\ust\tau=0$, then $\thh$ and $\tau$ do not start with a common arrow. This means that one of the following holds. In either case, we obtain $\vp(\thh\ust)\vp(\tau)=0=\vp(\thh\ust\tau)$.
\begin{itemize}
\item $\rho\ne\om$, in which case $\vp(\thh)\in\Mcal_{\rho}$ and $\vp(\tau)\in\Mcal_{\om}$ satisfy $\vp(\thh\ust)\vp(\tau)=0$ obviously.
\item $\rho=\om$ and $u\ne y$, in which case we have $\vp(\thh\ust)\vp(\tau)=\rho_{v,u-1}\rho_{y-1,w}=0$ in $\Mcal_{\rho}$.
\end{itemize}
{\rm (8)} can be shown in a similar way.

It remains to show {\rm (9)}. By Lemma \ref{LemForThm1}, we have
\[ \thh\ust\tau\ust=\begin{cases}
(\tau\thh)\ust=(\al_y\cdots\al_v)\ust&\ \text{if}\ \rho=\om\ \text{and}\ u-1=w,\\0&\ \text{otherwise}.
\end{cases} \]
On the other hand, we have
\[ \vp(\thh\ust)\vp(\tau\ust)=\rho_{v,u-1}\om_{w,y-1}=\begin{cases}
\rho_{v,y-1}&\ \text{if}\ \rho=\om\ \text{and}\ u-1=w,\\0&\ \text{otherwise}.
\end{cases} \]
Since $\vp((\al_y\cdots\al_v)\ust)=\rho_{v,y-1}$ by definition, it follows $\vp(\thh\ust\tau\ust)=\vp(\thh\ust)\vp(\tau\ust)$.
\end{proof}

\section{A case preserving derived equivalences}\label{section_Preservation}

Throughout this section, let $A,B$ be gentle algebras, and suppose that a tilting complex $T^{\bullet}\in K^b(\proj A)$ gives an isomorphism $\End_{K^b(\proj A)}\cong B$ of $K$-algebras, as in \cite{Ri}. Let $k$ be any positive integer.
Under the assumption of some conditions, we will show that $T^{\bullet}$ induces a tilting complex which gives derived equivalence $\Ak\sim B^{(k)}$. The proof is analogous to the one in \cite{As1}.

As before, we identify $\Ak$ with $\Akp$.
\begin{dfn}\label{Def_mui}
For each $1\le i\le k$, let $(-)^{[i]}\co A\hookrightarrow \Ak$ denote the inclusion to the $i$-th diagonal component. Especially, the unit $1\in A$ is sent to an idempotent element $1^{[i]}\in \Ak$ by this map. Put $N_i=1^{[i]}\Ak\in\proj \Ak$. This is an $A$-$\Ak$-bimodule, on which $A$ acts through $(-)^{[i]}$.

Since there is a natural isomorphism $e_aA\otimes_A N_i\cong e_a^{[i]}\Ak=e_{a^{[i]}}\Ak$ of right $\Ak$-modules for any $a\in Q_0$, we have an additive functor
\[ F_i=-\otimes_AN_i\co \proj A\to\proj \Ak.  \]
By a degreewise tensor product on complexes, it naturally extends to an additive functor between the categories of bounded complexes
\[ F_i\co C^b(\proj A)\to C^b(\proj \Ak), \]
and induces a triangle functor between homotopy categories
\[ F_i\co K^b(\proj A)\to K^b(\proj \Ak), \]
which we denote by the same symbol $F_i$.
\end{dfn}

\begin{dfn}\label{Def_Ttilde}
Let $T^{\bullet}$ be as above. Put $T^{\bullet}_i=F_i(T^{\bullet})\in\Ob(C^b(\proj \Ak))=\Ob(K^b(\proj A^{(k)}))$ for each $1\le i\le k$, and put $\wt{T}^{\bullet}=\displaystyle\bigoplus_{1\le i\le k}T_i^{\bullet}$.
\end{dfn}

\begin{lem}\label{LemThick}
$\thick(\wt{T}^{\bullet})=K^b(\proj \Ak)$ holds.
\end{lem}
\begin{proof}
For any $1\le i\le k$, since $F_i\co K^b(\proj A)\to K^b(\proj \Ak)$ is a triangle functor, we see that $F_i\iv(\thick(\wt{T}^{\bullet}))\se K^b(\proj A)$ is a thick subcategory. Thus their intersection
\[ \Ccal=\displaystyle\bigcap_{1\le i\le k} F_i\iv(\thick(\wt{T}^{\bullet}))\se K^b(\proj A) \]
also becomes a thick subcategory. Since $T^{\bullet}\in\Ccal$, it follows $\Ccal=K^b(\proj A)$. In particular we have $A\in\Ccal$, which means that $N_i=F_i(A)\in\thick(\wt{T}^{\bullet})$ holds for any $1\le i\le k$. This shows
\[ \Ak=\displaystyle\bigoplus_{1\le i\le k}N_i\in\thick(\wt{T}^{\bullet}), \]
and thus $\thick(\wt{T}^{\bullet})=K^b(\proj \Ak)$ follows.
\end{proof}

We consider the following condition.
\begin{cond}\label{CondRestrictive}
Cokernel of $\eta_A\co A\hookrightarrow V(A)$ admits an isomorphism of $A$-$A$-bimodules $\Cok\eta_A\cong DA$. In other words, we have a short exact sequence
\begin{equation}\label{Seq_AVADA}
0\to A\ov{\eta_A}{\lra}V(A)\ov{\pi_A}{\lra}DA\to0
\end{equation}
of $A$-$A$-bimodules.
\end{cond}

In the following, we denote the category $K^b(\modd A)$ by $\Kcal$ for simplicity.
\begin{lem}\label{LemPartialTilting}
Assume that $A$ satisfies Condition \ref{CondRestrictive}. Then the following holds for $n\in\mathbb{Z}$.
\begin{enumerate}
\item For any $n\ne0$, we have $\Kcal(T^{\bullet},(T^{\bullet}\otimes_AV(A))[n])=0$.
\item Sequence $(\ref{Seq_AVADA})$ induces a short exact sequence of $B$-$B$-bimodules
\[ 0\to B\ov{\eta\ppr}{\lra}V\ppr\ov{\pi\ppr}{\lra}DB\to0 \]
with $V\ppr=\End_{K^b(\modd V(A))}(T^{\bullet}\otimes_AV(A))$, which satisfies the following.
\begin{itemize}
\item[{\rm (i)}] $V\ppr$ is a semisimple $K$-algebra.
\item[{\rm (ii)}] $\eta\ppr$ is a morphism of $K$-algebras.
\item[{\rm (iii)}] $DB$ is the standard dual of $B$.
\end{itemize}
\end{enumerate}
Here, $T^{\bullet}\otimes_A V(A)\in\Ob(K^b(\modd V(A)))$ is given by the degreewise tensor product.
\end{lem}
\begin{proof}
By the degreewise tensor product with $T^{\bullet}$, the sequence $(\ref{Seq_AVADA})$ induces an extension of complexes
\begin{equation}\label{ExaTensor}
0\to T^{\bullet}\otimes_AA\ov{T^{\bullet}\otimes_A\eta_A}{\lra}T^{\bullet}\otimes_AV(A)\ov{T^{\bullet}\otimes_A\pi_A}{\lra}T^{\bullet}\otimes_ADA\to0
\end{equation}
in $C^b(\modd A)$. As we have a natural isomorphism $T^{\bullet}\otimes_AA\cong T^{\bullet}$, this induces an extension of total $\mathrm{Hom}$-complexes
\[ 0\to \HHom^{\bullet}_A(T^{\bullet},T^{\bullet})\to \HHom^{\bullet}_A(T^{\bullet},T^{\bullet}\otimes_AV(A))\to \HHom^{\bullet}_A(T^{\bullet},T^{\bullet}\otimes_ADA)\to0 \]
in $C^b(\modd K)$. By taking their cohomologies, we obtain a long exact sequence
\begin{equation}\label{LongExact}
\cdots\to \Kcal(T^{\bullet},T^{\bullet}[n])\to
 \Kcal(T^{\bullet},(T^{\bullet}\otimes_AV(A))[n])\to
 \Kcal(T^{\bullet},(T^{\bullet}\otimes_ADA)[n])\to
 \Kcal(T^{\bullet},T^{\bullet}[n+1])\to\cdots.
\end{equation}

\medskip

{\rm (1)} Since $T^{\bullet}$ is a tilting complex, we have $\Kcal(T^{\bullet},T^{\bullet}[n])=0$ for $n\ne0$. Moreover, by \cite[Lemma 2.12]{As1} we have $\Kcal(T^{\bullet},(T^{\bullet}\otimes_ADA)[n])\cong D\Kcal(T^{\bullet}[n],T^{\bullet})=0$. Exactness of $(\ref{LongExact})$ shows $\Kcal(T^{\bullet},(T^{\bullet}\otimes_AV(A))[n])=0$ for any $n\ne0$.

\medskip

{\rm (2)} As a part of $(\ref{LongExact})$ we have a short exact sequence
\[ 0\to\Kcal(T^{\bullet},T^{\bullet}\otimes_A A)\ov{(T^{\bullet}\otimes _A\eta_A)\ci-}{\lra}\Kcal(T^{\bullet},T^{\bullet}\otimes_AV(A))\ov{(T^{\bullet}\otimes _A\pi_A)\ci-}{\lra}\Kcal(T^{\bullet},T^{\bullet}\otimes_ADA)\to0, \]
which can be obtained from $(\ref{ExaTensor})$ by applying $\Kcal(T^{\bullet},-)$. This sequence is functorial in $T^{\bullet}$ in both of the 1st and 2nd components. Through the natural isomorphisms $\Kcal(T^{\bullet},T^{\bullet}\otimes_AA)\cong\Kcal(T^{\bullet},T^{\bullet})\cong B$ and $\Kcal(T^{\bullet},T^{\bullet}\otimes_ADA)\cong D(\Kcal (T^{\bullet},T^{\bullet}))\cong DB$, we obtain a short exact sequence of $B$-$B$-bimodules
\begin{equation}\label{ExMid}
0\to B\ov{\eta\pprr}{\lra}\Kcal(T^{\bullet},T^{\bullet}\otimes_AV(A))\ov{\pi\pprr}{\lra} DB\to 0.
\end{equation}
Let $\vartheta\co\Kcal(T^{\bullet},T^{\bullet}\otimes_AV(A))\ov{\cong}{\lra}K(\modd V(A))(T^{\bullet}\otimes_AV(A),T^{\bullet}\otimes_AV(A))=V\ppr$ be the adjoint isomorphism. The composition of 
\[ \End_{\Kcal}(T^{\bullet})\ov{\cong}{\lra}\Kcal(T^{\bullet},T^{\bullet}\otimes_AA)\ov{(T^{\bullet}\otimes _A\eta_A)\ci-}{\lra}\Kcal(T^{\bullet},T^{\bullet}\otimes_AV(A))\un{\cong}{\ov{\vartheta}{\lra}}V\ppr
\]
is nothing but the homomorphism on the endomorphism rings, induced by the $K$-linear functor
\[ -\otimes_AV(A)\co K^b(\modd A)\to K^b(\modd V(A)). \]
Thus in the short exact sequence 
\[ 0\to B\ov{\vartheta\ci\eta\pprr}{\lra}V\ppr\ov{\vartheta\iv\ci\pi\pprr}{\lra} DB\to 0 \]
obtained from $(\ref{ExMid})$, we see that $\vartheta\ci\eta\pprr$ is a homomorphism of $K$-algebras. Since $V(A)$ is semisimple, so is $V\ppr=\End_{K^b(\modd V(A))}(T^{\bullet}\otimes_AV(A))$.
\end{proof}

\begin{thm}\label{ThmDerEq}
Assume that $A$ satisfies Condition \ref{CondRestrictive}. The following holds.
\begin{enumerate}
\item $\wt{T}^{\bullet}\in K^b(\proj \Ak)$ is a tilting complex.
\item If moreover $B$ satisfies Condition \ref{CondMaxLength1} and $\Char(K)\ne 2$, then $\wt{T}^{\bullet}$ gives a derived equivalence $\Ak\sim B^{(k)}$.
\end{enumerate}
\end{thm}
\begin{proof}
For any $1\le i,j\le k$, we have natural isomorphisms
\begin{eqnarray*}
\HHom_{\Ak}^{\bullet}(T_j^{\bullet},T^{\bullet}_i)
&=&\HHom_{\Ak}^{\bullet}(T^{\bullet}\otimes_AN_j,T^{\bullet}\otimes_AN_i)\\
&\cong&\HHom_A^{\bullet}(T^{\bullet},\HHom_{\Ak}^{\bullet}({}_AN_j,T^{\bullet}\otimes_AN_i))\\
&=& \HHom_A^{\bullet}(T^{\bullet},\HHom_{\Ak}^{\bullet}(1^{[j]}\Ak,T^{\bullet}\otimes_A1^{[i]}\Ak))\\
&\cong&\HHom_A^{\bullet}(T^{\bullet},T^{\bullet}\otimes_A (1^{[i]}\Ak1^{[j]}))\\
&\cong&
\begin{cases}
\ 0&\text{if}\ i>j\\
\ \HHom_A^{\bullet}(T^{\bullet},T^{\bullet})&\text{if}\ i=j\\
\ \HHom_A^{\bullet}(T^{\bullet},T^{\bullet}\otimes_AV(A))&\text{if}\ i<j
\end{cases}
\end{eqnarray*}
in $C^b(\modd K)$. 

\medskip

{\rm (1)} For any $n\ne0$ and any $1\le i,j\le k$, taking cohomologies in the above isomorphism, we obtain
\begin{eqnarray*}
K^b(\proj \Ak)(T_j^{\bullet},T_i^{\bullet}[n])
&=&H^n(\HHom_{\Ak}^{\bullet}(T_j^{\bullet},T_i^{\bullet}))\\
&\cong&
\begin{cases}
\ 0&\text{if}\ i>j\\
\ H^n(\HHom_A^{\bullet}(T^{\bullet},T^{\bullet}))&\text{if}\ i=j\\
\ H^n(\HHom_A^{\bullet}(T^{\bullet},T^{\bullet}\otimes_AV(A))&\text{if}\ i<j
\end{cases}\\
&\cong&
\begin{cases}
\ 0&\text{if}\ i>j\\
\ \Kcal(T^{\bullet},T^{\bullet}[n])&\text{if}\ i=j\\
\ \Kcal(T^{\bullet},(T^{\bullet}\otimes_AV(A))[n]))&\text{if}\ i<j
\end{cases}\\
&=0&
\end{eqnarray*}
by Lemma \ref{LemPartialTilting}. This means
\[ K^b(\proj \Ak)(\wt{T}^{\bullet},\wt{T}^{\bullet}[n])\cong\un{1\le i,j\le k}{\bigoplus}K^b(\proj \Ak)(T_j^{\bullet},T_i^{\bullet}[n])=0 \]
for $n\ne0$, and Lemma \ref{LemThick} shows that $T^{\bullet}\in K^b(\proj \Ak)$ is a tilting complex.
\end{proof}

{\rm (2)} For $n=0$, similarly we have
\begin{eqnarray*}
K^b(\proj \Ak)(T_j^{\bullet},T_i^{\bullet})
&=&H^0(\HHom_{\Ak})^{\bullet}(T_j^{\bullet},T_i^{\bullet})\\
&\cong&
\begin{cases}
\ 0&\text{if}\ i>j\\
\ \Kcal(T^{\bullet},T^{\bullet}))\cong B&\text{if}\ i=j\\
\ \Kcal(T^{\bullet},(T^{\bullet}\otimes_AV(A))))\cong V\ppr&\text{if}\ i<j
\end{cases}
\end{eqnarray*}
for any $1\le i,j\le k$. This induces the following isomorphism of $K$-algebras.
\begin{equation}\label{EqDerMat}
\End_{K^b(\proj \Ak)}(\wt{T}^{\bullet})
\cong\big[K^b(\proj \Ak)(T_j^{\bullet},T_i^{\bullet})\big]_{i,j}=\left[
\begin{array}{ccccc}
B&V\ppr&V\ppr&\cdots&V\ppr\\
0&B&V\ppr&\cdots&V\ppr\\
0&0&B&\ddots&\vdots\\
0&\cdots&&\ddots&V\ppr\\
0&\cdots&\cdots&0&B\\
\end{array}
\right]
\end{equation}
By Theorem \ref{ThmCharacterizeVA} we have an isomorphism of $K$-algebras
$\vp\co V(B)\ov{\cong}{\lra}V\ppr$
which is compatible with $\eta\ppr$ and $\eta_B\co B\hookrightarrow V(B)$, and thus $(\ref{EqDerMat})$ induces an isomorphism
$\End_{K^b(\proj \Ak)}(\wt{T}^{\bullet})\cong B^{(k)}$.


\begin{thebibliography}{BDMTY}
\bibitem[As1]{As1} Asashiba, H.: \emph{A covering technique for derived equivalence}. J. Algebra \textbf{191} (1997), no. 1, 382--415.

\bibitem[As2]{As2} Asashiba, H.: \emph{Derived equivalence classification of algebras}. Sugaku Expositions \textbf{29} (2016), no. 2, 145--175.

\bibitem[ASS]{ASS} Assem, I.; Simson, D.; Skowro\'{n}ski, A.: \emph{Elements of the representation theory of associative algebras}. Vol. 1. Techniques of representation theory. London Mathematical Society Student Texts, 65. Cambridge University Press, Cambridge, 2006. x+458 pp.

\bibitem[Av]{Av} Avella-Alaminos, D.: \emph{Derived classification of gentle algebras with two cycles}. Bol. Soc. Mat. Mexicana (3) \textbf{14} (2008), no. 2, 177--216. 

\bibitem[AG]{AG} Avella-Alaminos, D.; Geiss, C.: \emph{Combinatorial derived invariants for gentle algebras}. J. Pure Appl. Algebra \textbf{212} (2008), no. 1, 228--243.

\bibitem[Bo]{Bo} Bobi\'{n}ski, G.: \emph{Derived equivalence classification of the gentle two-cycle algebras}. Algebr. Represent. Theory \textbf{20} (2017), no. 4, 857--869. 

\bibitem[BM]{BM} Bobi\'{n}ski, G.; Malicki, P.: \emph{On derived equivalence classification of gentle two-cycle algebras}. Colloq. Math. \textbf{112} (2008), no. 1, 33--72.

\bibitem[BDMTY]{BDMTY} Br\"{u}stle, T.; Douville, G.; Mousavand, K.; Thomas, H.; Y\i ld\i r\i m, E.:
\emph{On the combinatorics of gentle algebras}. arXiv:1707.07665



\bibitem[La]{La} Ladkani, S.: \emph{Hochschild cohomology of gentle algebras}. Algebr. Represent. arXiv:1208.2230. 

\bibitem[LP]{LP} Lekili, Y.; Polishchuk, A.: 
\emph{Derived equivalences of gentle algebras via Fukaya categories}. arXiv:1801.06370.

\bibitem[PPP1]{PPP1} Palu, Y.; Pilaud, V.; Plamondon, P.-G.: \emph{Non-kissing complexes and $\tau$-tilting for gentle algebras}, arXiv:1707.07574.

\bibitem[PPP2]{PPP2} Palu, Y.; Pilaud, V.; Plamondon, P.-G.: \emph{Non-kissing and non-crossing complexes for locally gentle algebras}. arXiv:1807.04730.

\bibitem[Ri]{Ri} Rickard, J.: \emph{Derived categories and stable equivalence}. J. Pure Appl. Algebra \textbf{61} (1989), no. 3, 303--317.

\bibitem[RR]{RR} Redondo, M.J.; Rom\'{a}n, L.: \emph{Gerstenhaber algebra structure on the Hochschild cohomology of quadratic string algebras}. Algebr. Represent. Theory \textbf{21} (2018), no. 1, 61--86.

\bibitem[SZ]{SZ} Schr\"{o}er, J.; Zimmermann, A.:
\emph{Stable endomorphism algebras of modules over special biserial algebras}.
Math. Z. \textbf{244} (2003), no. 3, 515--530. 
\end{thebibliography}
\end{document}